\documentclass[11pt, twoside, a4paper]{amsart}

\usepackage{amsmath}
\usepackage{latexsym}
\usepackage{amsthm}
\usepackage{thmtools}
\usepackage{hyperref}
\usepackage[nameinlink]{cleveref}
\usepackage{IEEEtrantools}
\usepackage{crossreftools}
\usepackage{enumitem}
\setlist[enumerate]{topsep=10pt, itemsep=5pt, parsep=0pt}
\usepackage{mathrsfs}
\usepackage{blkarray}
\usepackage{adjustbox,lipsum}
\usepackage{stmaryrd}
\usepackage{mathdots}
\usepackage{amssymb}
\usepackage{tikz}
\usepackage{tikz-cd}
\usepackage{yfonts}
\usepackage{bbm}
\usepackage{etoolbox}
\usepackage{float}
\usepackage{xcolor}

\author{Yao Ming Chan}
\address{Yao Ming Chan, Mathematical Institute, Andrew Wiles Building, Woodstock Rd, OX2 6GG, Oxford, United Kingdom \newline} 
\email{yao.chan@maths.ox.ac.uk} 
\title{The Cuntz semigroup of the graph C*-algebra of a finite graph}

\theoremstyle{theorem}
\newtheorem{theorem}{Theorem}[section]
\newtheorem{lemma}[theorem]{Lemma}
\newtheorem{corollary}[theorem]{Corollary}
\newtheorem{proposition}[theorem]{Proposition}

\theoremstyle{definition}
\newtheorem{definition}[theorem]{Definition}
\newtheorem{example}[theorem]{Example}
\newtheorem{remark}[theorem]{Remark}

\newtheorem{IntroTheorem}[theorem]{Theorem}

\newtheorem{IntroCorollary}[theorem]{Corollary}

\numberwithin{equation}{section}

\crefname{subsection}{Subsection}{Subsections}
\Crefname{subsection}{Subsection}{Subsections}

\makeatletter 
\def\l@subsection{\@tocline{2}{0pt}{2pc}{5pc}{}} 
\def\l@subsection{\@tocline{2}{0pt}{3pc}{6pc}{}}
\makeatother

\DeclareMathOperator{\Cu}{Cu}

\DeclareMathOperator{\supp}{supp}
\DeclareMathOperator{\rank}{rank}
\DeclareMathOperator{\Lsc}{Lsc}
\DeclareMathOperator{\Lat}{Lat}
\DeclareMathOperator{\Idl}{Idl}

\DeclareMathOperator{\Span}{span}

\newcommand{\stack}{ \;
    \begin{tikzpicture}[scale=0.125, line width = 0.7pt]
        \draw (-1,0) rectangle (0,1);
        \draw (0,1) rectangle (1,2);
    \end{tikzpicture}
\; }

\begin{document}
\date{}

\begin{abstract} 
We compute the Cuntz semigroup of a graph C*-algebra associated to a finite graph by combining Hong and Szymański's realisation of the entire ideal lattice of a graph C*-algebra in \cite{HS04},  Ara, Moreno and Pardo's computation of the Murray-von Neumann semigroup of a unital graph C*-algebra in \cite{AMP07} and various Cuntz semigroup techniques. We also investigate certain properties of such Cuntz semigroups, and demonstrate that at the level of the Cuntz semigroup, there is no obstruction to the conjecture made in \cite{ENSW25} that the nuclear dimension of any graph C*-algebra is at most one. \newline
\phantom{} \newline
\textbf{Keywords.} Cuntz semigroup, Graph C*-algebra
\end{abstract}
\maketitle
\tableofcontents
\vspace*{-1mm}

\section*{Introduction}
\phantom{} 
The goal of this paper is to describe the Cuntz semigroup of a graph C*-algebra coming from a finite graph: one with finitely many vertices and edges. The Cuntz semigroup of a C*-algebra $A$, denoted by $\Cu(A)$, is a semigroup consisting of positive elements in $A \otimes K$ modulo Cuntz equivalence, an equivalence relation first considered in \cite[Section 1]{Cun78} which generalises Murray-von Neumann equivalence of projections. Due to the plentiful supply of positive elements in any C*-algebra, the Cuntz semigroup is notable for its ability to generally recover far more information about the underlying C*-algebra than other invariants commonly considered for C*-algebras. By exploiting this, Toms constructed his landmark counterexample to the Elliott conjecture in \cite[Theorem 1.1]{Tom08}, and in the process, demonstrated the Cuntz semigroup's faculty for distinguishing pairs of C*-algebras other invariants may remain blind to. Consequently, the theory of Cuntz semigroups rose to prominence as an active area of research, resulting in powerful techniques involving positive elements which are consistently used to study the structure of C*-algebras with nice regularity properties; see \cite{APS11}, \cite{ABP11}, \cite{Rob11}, \cite{KNZ17}, \cite{APRT22}, \cite{APTV24} and \cite{TV24} for example. \newline
\phantom{} \newline
So far, classification has been achieved primarily for classes of C*-algebras satisfying strong regularity conditions; see \cite[Corollary 6.2.4]{Rob12}, \cite[Corollary H]{Gab24} and \cite[Theorem A]{CGSTW26} for instance. One class of mostly non-simple C*-algebras with both stably finite and purely infinite building blocks which has a complete classification are unital graph C*-algebras. This was achieved in the paper \cite{ERRS21} using ordered reduced filtered $K$-theory --- developed in \cite[Section 3.3]{ERRS18} --- and moves for graph C*-algebras, a concept originating from symbolic dynamics (see \cite{Wil73}) which systematically evolved over the past twenty years; see \cite{BP04}, \cite{Sor13}, \cite{ERRS21} and \cite{ER25} for instance. Therefore, graph C*-algebras provide very natural and useful examples of  non-simple C*-algebras to test various techniques from the classification programme on.  \newline
\phantom{} \newline
Many invariants for C*-algebras can be explicitly realised for graph C*-algebras by a condition on the underlying graph. For instance, Cuntz computed the $K$-theory of the subclass of Cuntz-Krieger algebras in \cite[Section 4]{Cun81}, \cite[Theorem 3.4]{DHS03} described the stable rank of a graph C*-algebra in terms of certain conditions on the graph and \cite[Theorem 4.3, Theorem 4.6]{JPS01} characterised real rank zero via Condition (K). The computability of invariants for graph C*-algebras is very often due to the Cuntz-Krieger relations, which force the projections at the vertices and the partial isometries at the edges to interact with each other in a way that is dictated by the structure of the underlying graph. The most relevant example of this to the paper is \cite[Theorem 7.1]{AMP07}, which says that for a row-finite graph $E$, the Murray-von Neumann semigroup of the graph C*-algebra $C^*(E)$ is generated by the Murray-von Neumann equivalences classes of vertex projections, with the relations amongst the generators being precisely the ones induced by the Cuntz-Krieger relations. Due to the importance of positive elements in probing the structure of C*-algebras, \cite[Theorem 7.1]{AMP07} naturally leads to the following question, which serves as one of the primary motivations for this paper: \emph{what do the Cuntz-Krieger relations say about the behaviour of Cuntz classes of positive elements?} \newline
\phantom{} \newline
Another source of motivation stems from the question of computing the nuclear dimension of a graph C*-algebra. Introduced in \cite[Definition 2.1]{WZ10}, nuclear dimension is one of many non-commutative generalisations of the covering dimension of a topological space, notable for its intimate connection to nuclearity (see \cite[Remark 2.2 (i)]{WZ10}) and its tame behaviour with respect to  extensions of C*-algebras (see \cite[Proposition 2.9]{WZ10}). Due to its importance to the classification programme, an active area of research is to compute the precise nuclear dimension of certain classes of C*-algebras. For graph C*-algebras, it was conjectured in \cite[Question C]{ENSW25} that the nuclear dimension of any graph C*-algebra is at most one. With our description of the Cuntz semigroup of a graph C*-algebra from a finite graph, we will show that there is no obstruction to this conjecture which arises from the Cuntz semigroup since for instance, the covering dimension of such a Cuntz semigroup, defined in \cite[Definition 3.1]{TV22}, will be shown to be at most one. 

\subsection*{Description of the main theorem}

In order to succinctly describe the Cuntz semigroup of a graph C*-algebra coming from a finite graph, we will begin with an example. Suppose that $E$ is the following graph, whose graph C*-algebra $C^*(E)$ was identified in \cite[Page 89]{ENSW25} as an important example to compute the nuclear dimension of:
\begin{equation*}
\begin{tikzpicture}[baseline=(current  bounding  box.center)] 
\node (1) at (0, 0) {$v$};
\node (2) at (3, 0) {$w$};
\draw[->] (1) to (2);
\draw[->] (1) to [out = 45, in = 135, looseness = 10] (1);
\draw[->] (2) to [out = 45, in = 135, looseness = 10] (2);
\end{tikzpicture}
\end{equation*}
First, recall that the graph C*-algebra corresponding to the single loop graph
\begin{equation*}
\begin{tikzpicture}[baseline=(current  bounding  box.center)] 
\node (1) at (0, 0) {$v$};
\draw[->] (1) to [out = 45, in = 135, looseness = 10] (1);
\end{tikzpicture}
\end{equation*}
is isomorphic to $C(\mathbb{T})$ and that by \cite[Theorem 1.1]{Rob13b}, the Cuntz semigroup of $C(\mathbb{T})$ is isomorphic to lower semicontinuous functions from $\mathbb{T}$ to $\overline{\mathbb{N}} = \{0, 1, \dots \} \cup \{ \infty \}$, which we denote by $\Lsc(\mathbb{T}, \overline{\mathbb{N}})$. The addition and order on $\Lsc(\mathbb{T}, \overline{\mathbb{N}})$ is pointwise and the isomorphism is implemented by the pointwise rank map. \newline
\phantom{} \newline
The graph C*-algebra $C^*(E)$ decomposes as an extension of $C(\mathbb{T})$ by $C(\mathbb{T}) \otimes K$. Its Cuntz semigroup is given by stacking the Cuntz semigroup of the quotient $S_w \cong \Lsc(\mathbb{T}, \overline{\mathbb{N}})$ on top of the Cuntz semigroup of the ideal $S_v \cong \Lsc(\mathbb{T}, \overline{\mathbb{N}})$ in such a way that the zero element of $S_w$ is identified with the maximal element of $S_v$. The addition and order is pointwise on each disjoint component. Now by the Cuntz-Krieger relations, if $1_w$ and $1_v$ are the constant functions 1 in $S_w$ and $S_v$ respectively, then $1_w + 1_v = 1_w$. It turns out that this behaviour extends to all Cuntz classes of $\Cu(C^*(E))$. Indeed, the key feature of $\Cu(C^*(E))$ is that a Cuntz class in $S_w$ automatically absorbs any Cuntz class in $S_v$. That is, if $f \in S_w$ and $g \in S_v$ then $f + g = f$. Consequently, any Cuntz class in $S_v$ is automatically Cuntz dominated by a Cuntz class in $S_w$. We call the construction of $\Cu(C^*(E))$ from $S_v$ and $S_w$ the stack of $S_v$ below $S_w$; see \cref{subsec: Stack} for further details about the stack. \newline
\phantom{} \newline
Now suppose that we add an extra loop to the vertex $w$ so that the graph $E$ is the following:
\begin{equation*}
\begin{tikzpicture}[baseline=(current  bounding  box.center)] 
\node (1) at (0, 0) {$v$};
\node (2) at (3, 0) {$w$};
\draw[->] (1) to (2);
\draw[->] (1) to [out = 45, in = 135, looseness = 10] (1);
\draw[->] (2) to [out = 45, in = 135, looseness = 10] (2);
\draw[->] (2) to [out = 315, in = 225, looseness = 10] (2);
\end{tikzpicture}
\end{equation*}
This time, $C^*(E)$ is an extension of $\mathcal{O}_2$ by $C(\mathbb{T}) \otimes K$. Recall that the Cuntz semigroup of a simple purely infinite C*-algebra (such as $\mathcal{O}_2$) is $\{ 0, \infty \}$ with the natural addition and order. Just like the previous example, the Cuntz semigroup of $C^*(E)$ is given by stacking the Cuntz semigroup of the quotient $S_w \cong \{ 0, \infty \}$ on top of the Cuntz semigroup of the ideal $S_v \cong \Lsc(\mathbb{T}, \overline{\mathbb{N}})$. \newline
\phantom{} \newline
A key property shared by the above two examples is that as extensions, they are both \emph{purely large}. By \cite[Definition 4.8]{BG24} and \cite[Proposition 4.14]{BG24}, pure largeness imposes a strong restriction on the Cuntz semigroup of a purely large extension. If the extension also happens to be a unital graph C*-algebra, then its Cuntz semigroup also contends with further restrictions coming from the Cuntz-Krieger relations. Combining both of these, we obtain the following Cuntz semigroup computation, which not only generalises the Cuntz semigroups of the previous two examples, but also serves as an important special case of the main theorem. 

\begin{IntroTheorem} \label{introtheorem: General Cuntz Krieger}
Let $E$ be a finite graph. Suppose that there exists a vertex $v \in E^0$ such that $v$ either supports a single loop or supports multiple cycles and every vertex in $E$ admits a finite path to $v$. Let $$H = \{ x \in E^0 \; \vert \; \text{$v$ does not admit a finite path to $x$} \}$$ so that $\Cu(C^*(H))$ is an ideal of $\Cu(C^*(E))$. The Cuntz semigroup of $C^*(E)$ is isomorphic to the stack of the Cuntz semigroup of the quotient $\Cu(C^*(E \backslash H))$ above the Cuntz semigroup of the ideal $\Cu(C^*(H))$.  
\end{IntroTheorem}

In the situation of Theorem \ref{introtheorem: General Cuntz Krieger}, if $w$ is a vertex distinct from $v$, then by the Cuntz-Krieger relations, $[p_v] + [p_w] = [p_v]$ because $w$ admits a finite path to $v$. Noting that $[p_v] \not \in \Cu(C^*(H))$, we see by Theorem \ref{introtheorem: General Cuntz Krieger} that this behaviour in fact generalises directly to Cuntz classes of positive elements, similarly to the first example. Consequently, we may think of Theorem \ref{introtheorem: General Cuntz Krieger} as a generalisation of the Cuntz-Krieger relations for $\Cu(C^*(E))$ with $E$ as in Theorem \ref{introtheorem: General Cuntz Krieger}. \newline
\phantom{} \newline
The computation of the Cuntz semigroup of $C^*(E)$ where $E$ is a finite graph is done in two steps. Firstly, by using moves for graph C*-algebras, we may assume without loss of generality that $E$ satisfies the following three properties: 
\begin{enumerate}
\item[{\crtcrossreflabel{(S1)}[property: S1]}] $E$ has no regular sinks. 
\item[{\crtcrossreflabel{(S2)}[property: S2]}] $E$ has no regular transitional vertices (vertices that are neither sources nor sinks and support no cycles). 
\item[{\crtcrossreflabel{(S3)}[property: S3]}] If a vertex $v$ supports a single cycle $\mu$, then $\mu$ is a loop (has length one). 
\end{enumerate}
The assumptions \ref{property: S1}, \ref{property: S2} and \ref{property: S3} will be collectively referred throughout the paper as the \textbf{standing assumptions}. The point of the standing assumptions is that they greatly simplify our analysis of the ideal lattice of $C^*(E)$. At a high level, the second step is to merge the topological analysis of the entire ideal lattice of a graph C*-algebra from \cite{HS04} with the more algebraic computation of the Murray-von Neumann semigroup of a unital graph C*-algebra from \cite[Theorem 7.1]{AMP07} to produce the desired computation, which we informally describe below. 

\begin{IntroTheorem} \label{introtheorem: Cu E Iso}
Let $E$ be a finite graph satisfying the standing assumptions \ref{property: S1}, \ref{property: S2} and \ref{property: S3}. Place a copy of $\overline{\mathbb{N}}$ at every source vertex in $E$, a copy of $\Lsc(\mathbb{T}, \overline{\mathbb{N}})$ at every vertex supporting a single loop and a copy of $\{ 0, \infty \}$ at every vertex supporting more than one cycle. If a vertex $v$ receives an edge from another vertex $w$ which cannot be reached by a finite path beginning at $v$, then remove the zero element from the Cuntz semigroup at $v$. The Cuntz semigroup of $C^*(E)$ is generated by the elements of the Cuntz semigroups at every vertex of $E$. \newline
\phantom{} \newline
When restricted to each individual disjoint component, the addition and order are the usual ones. Across components, the elements of the Cuntz semigroups at every vertex are subject to the following relation:
\begin{enumerate}
\item If $v, w$ are vertices in $E$, $w$ admits a finite path to $v$ and $w \neq v$, then a (non-zero) Cuntz class coming from the Cuntz semigroup at $v$ absorbs any Cuntz class coming from the Cuntz semigroup at $w$. 
\end{enumerate}
\end{IntroTheorem}
More precisely, for a finite graph $E$ satisfying the standing assumptions, we abstractly characterise the Cuntz semigroup of $C^*(E)$ using the Cu-semigroup $\Cu_E$; see \Cref{sec: Cu E} for the details. Hong and Szymański's computation of the entire ideal lattice of $C^*(E)$ in \cite{HS04} is used to construct a map $\Delta_E$ from $\Cu_E$ to $\Cu(C^*(E))$. Theorem \ref{introtheorem: General Cuntz Krieger} plays a crucial role in ensuring $\Delta_E$ is a well-defined Cu-morphism. \newline
\phantom{} \newline
The proof that $\Delta_E$ is a Cu-isomorphism is done by induction over the finite gauge-invariant ideal lattice of $C^*(E)$, akin to the proof of \cite[Theorem 7.1]{AMP07}. The primary difficulty at this stage is to prove that $\Delta_E$ is surjective; that any element of $\Cu(C^*(E))$ is either a finite sum of Cuntz classes of vertex projections and strictly positive elements of certain corners of non-gauge-invariant ideals or the supremum of an increasing sequence of such finite sums. To accomplish this, we argue in a similar fashion to \cite[Theorem 7.1]{AMP07} and divide into two cases, depending on whether $C^*(E)$ has a simple purely infinite ideal or a stable rank one, gauge-simple ideal. In the former case, we rely on two results from the paper \cite{KR00}: Theorem \ref{theorem: Purely Infinite Properly Infinite} and Lemma \ref{lemma: Polar Part Absorbed}. In the latter case, we apply Theorem \ref{theorem: SR1 Corner Positive Cancellation}: a generalisation of \cite[Theorem 4.3]{RW10} which says that a projection cancels from Cuntz subequivalences if the corner it generates has stable rank one. Theorem \ref{theorem: SR1 Corner Positive Cancellation} can be thought of as a Cuntz semigroup version of the cancellation of modules result \cite[Theorem 2.1]{War80}, which was implicitly used in the proof of \cite[Theorem 7.1]{AMP07}. The proof of Theorem \ref{theorem: SR1 Corner Positive Cancellation} is done by adapting the proofs of \cite[Theorem 4.8]{RS21} and \cite[Theorem 4.9]{RS21}. \newline
\phantom{} \newline
Note that if $E$ is an arbitrary row-finite graph, then $C^*(E)$ can be written as an inductive limit of graph C*-algebras associated to finite graphs (see \cite[Lemma 3.3]{AMP07}). Since the Cuntz semigroup functor preserves inductive limits (see \cite[Theorem 2]{CEI08}) then by Theorem \ref{introtheorem: Cu E Iso}, we can in principle compute the Cuntz semigroup of any graph C*-algebra, as in the case where $E$ is not row-finite, we may use Drinen-Tomforde desingularisation to produce a row-finite graph $F$ such that $C^*(E)$ is Morita equivalent to $C^*(F)$ by \cite[Theorem 2.11]{DT05}.

\subsection*{Consequences}

By relying on the properties of the Cu-semigroup $\Cu_E$, we obtain various consequences of Theorem \ref{introtheorem: Cu E Iso}. Firstly, the Cuntz semigroup of a graph C*-algebra associated to a finite graph is almost unperforated.

\begin{IntroCorollary} \label{introcorollary: Cu E One Comparison}
Let $E$ be a finite graph satisfying the standing assumptions. Then $\Cu_E$ is almost unperforated. In particular, if $F$ is an arbitrary finite graph, then $\Cu(C^*(F))$ is almost unperforated. 
\end{IntroCorollary}

From Corollary \ref{introcorollary: Cu E One Comparison}, we recover the fact that a graph C*-algebra coming from a finite graph satisfies the \emph{corona factorisation property} of \cite[Definition 1.3]{KN06a}; see Remark \ref{remark: Graph CFP 2}. By \cite[Theorem 5.11]{OPR12}, this is consistent with the fact that a unital graph C*-algebra always has finite nuclear dimension (see \cite[Proposition 2.9]{WZ10}). \newline
\phantom{} \newline
As mentioned previously, we also compute Thiel and Vilalta's covering dimension (see \cite[Definition 3.1]{TV22}) for the Cuntz semigroup of a graph C*-algebra coming from a finite graph and as a result, prove the following Cuntz semigroup version of the conjecture that any graph C*-algebra has nuclear dimension at most one. 

\begin{IntroTheorem} \label{introtheorem: Covering Dimension}
Let $E$ be an arbitrary graph. Then $\dim \Cu(C^*(E)) \leq 1$. Moreover, if $E$ satisfies Condition (K), then $\dim \Cu(C^*(E)) = 0$. 
\end{IntroTheorem}

Theorem \ref{introtheorem: Covering Dimension} supports the conjecture made in \cite[Question C]{ENSW25} because if there exists a graph C*-algebra whose Cuntz semigroup has covering dimension strictly greater than one, then by \cite[Theorem 4.1]{TV22}, its nuclear dimension must also be strictly greater than one. From Theorem \ref{introtheorem: Cu E Iso}, we also obtain further structural properties of the Cuntz semigroup of a graph C*-algebra associated to a finite graph and of the extensions used to build such a graph C*-algebra. 

\subsection*{Structure of the paper}

In \Cref{sec: Prelims}, we collect background material on Cu-semigroups, graph C*-algebras and purely large extensions. We also carefully establish the definitions and notation we will use for graph C*-algebras, since these vary considerably in the dense literature on the subject. \Cref{sec: Weak Cancellation} serves two main purposes --- firstly, to prove a Cuntz semigroup analogue of the cancellation of modules result from \cite[Theorem 2.1]{War80} which we will need for the proof of Theorem \ref{introtheorem: Cu E Iso} and secondly, to generalise the fact that the Cuntz semigroup of a stable rank one C*-algebra is weakly cancellative, first proved in \cite[Theorem 4.3]{RW10}. \newline
\phantom{} \newline
\Cref{sec: Refined Moves} details two important applications of moves for graph C*-algebras: deriving the standing assumptions \ref{property: S1}, \ref{property: S2} and \ref{property: S3} and simplifying the proof of Theorem \ref{introtheorem: General Cuntz Krieger} in the case where $v$ supports a single loop. In \Cref{sec: Cu E}, we define for a finite graph $E$ satisfying the standing assumptions, the positively ordered monoid $\Cu_E$ and prove that it is a Cu-semigroup. Moreover, in \cref{subsec: Cu E Ideal Lattice}, we prove that there is a bijection from the ideal lattice of $C^*(E)$ to the ideal lattice of $\Cu_E$. \newline
\phantom{} \newline
\Cref{sec: General Cuntz Krieger} is dedicated to proving Theorem \ref{introtheorem: General Cuntz Krieger}. First, we deal with the simpler case where $v$ supports multiple cycles in \cref{subsec: Kirchberg Quotient}. The harder case where $v$ supports a single loop is done in \cref{subsec: Circle Quotient}. In \Cref{sec: Delta E}, we prove Theorem \ref{introtheorem: Cu E Iso} by first constructing the map $\Delta_E$ from $\Cu_E$ to $\Cu(C^*(E))$, and then showing that it is a Cu-isomorphism. Finally, we prove various consequences of Theorem \ref{introtheorem: Cu E Iso} in \Cref{sec: Consequences}. \newline
\phantom{} \newline
\textbf{Acknowledgements.} The author is grateful to Aidan Sims, Eduard Vilalta, Jeremy Hume and Mike Whittaker for helpful conversations regarding some of the topics in this paper. The author thanks Eduard Vilalta and Jamie Gabe for allowing their privately communicated results, Proposition \ref{proposition: Purely Large Axiom (O5)} and Theorem \ref{theorem: Purely Large Tensor Product} respectively, to be utilised in this paper. The author also thanks Chris Schafhauser for pointing out the graph at the beginning of \cref{subsec: Relatively Purely Large}, which became a useful test case for the Cuntz semigroup computations done in this paper. Finally, the author is incredibly grateful to his supervisor Stuart White for his careful proofreading and detailed feedback on earlier drafts of the paper and for his illuminating advice on writing research papers in general. \newline
\phantom{} \newline
\textbf{AI statement.} The first version of this paper was posted to the arXiv on October 1st 2026. No AI tools were used for the first version. \newline
\phantom{} \newline
After the first version of the paper became available online, ChatGPT Astra was asked to produce a technical report on it, with the goal of identifying typos and gaps in the various arguments used. As a result, Definition \ref{definition: Delta E}, the proof of Theorem \ref{theorem: Purely Large Tensor Product} and various typos were corrected. ChatGPT Astra also identified that the proof of Corollary \ref{introcorollary: Cu E One Comparison} in the first draft of the paper in fact shows that the Cuntz semigroup of any graph C*-algebra associated to a finite graph is almost unperforated, rather than having the 1-comparison property defined more generally in \cite[Definition 2.8]{OPR12}. This version of the paper was updated to reflect this. More in-depth explanations were also added throughout this paper. The author takes full responsibility for the mathematical and literary content of this paper. 

\section{Preliminaries} 
\label{sec: Prelims}

This section is dedicated to preliminary material that is essential to the Cuntz semigroup computation undertaken in the paper. \Cref{subsec: Cu-semigroups} is a quick introduction to the theory of Cu-semigroups. Since graph C*-algebras are mostly non-simple, \cref{subsec: Cu-semigroups} was written with a particular emphasis on ideals and quotients of Cu-semigroups. In \cref{subsec: Moves}, we establish the notation and terminology we will use throughout the paper for graph C*-algebras, and define the four moves (O), (I-), (S) and (CO) which admit two important applications detailed in \Cref{sec: Refined Moves}. In \cref{subsec: Ideal Lattice Prelims}, we recall fundamental material about the ideal lattice of a graph C*-algebra developed in \cite{BPRS00} and \cite{HS04}. \newline
\phantom{} \newline
\Cref{subsec: Purely Large} introduces the notion of a purely large extension, and highlights how Cuntz subequivalence behaves for a C*-algebra coming from a purely large extension; see Theorem \ref{theorem: Purely Large Tensor Product} and Corollary \ref{corollary: Purely Large Order Stabilised}. Pure largeness is an important tool for the proof of Theorem \ref{introtheorem: General Cuntz Krieger}, as we will see in the proofs of Theorem \ref{theorem: Purely Large Purely Infinite Quotient} and Lemma \ref{lemma: Purely Large Circle Quotient}. Finally, graph C*-algebras can have purely infinite ideals and quotients. Due to this, we collect certain results about purely infinite C*-algebras from \cite{KR00} in \cref{subsec: PI}, which we will use to treat purely infinite ideals and quotients of graph C*-algebras. 

\subsection{Cu-semigroups}
\label{subsec: Cu-semigroups}

In this subsection, we recall various aspects of the theory of Cu-semigroups that will be used throughout the paper. For comprehensive references on the abstract theory of Cu-semigroups, the reader is referred to \cite{APT18} and \cite{GP26}. Recall that a \textbf{positively ordered monoid} is a commutative monoid $(S, +)$ together with a partial order $\leq$ such that the zero element of $S$ is the smallest element with respect to $\leq$ and $\leq$ respects addition on $S$. \newline
\phantom{} \newline
In this paper, we will focus exclusively on positively ordered monoids with the property that every increasing sequence has a supremum. If $S$ is such a positively ordered monoid and $s, t \in S$ then $s$ is said to be \textbf{way-below} $t$ if for any increasing sequence $\{ r_n \}_{n > 0}$ in $S$ such that $t \leq {\sup}_{n > 0} \; r_n$, there exists $m > 0$ such that $s \leq r_m$. This is denoted by $s \ll t$. An element $s \in S$ is said to be \textbf{compact} if $s \ll s$. 

\begin{definition}
Let $S$ be a positively ordered monoid. $S$ is said to be a \textbf{Cu-semigroup} if $S$ satisfies the following four axioms:
\begin{enumerate}
\item[{\crtcrossreflabel{(O1)}[axiom: O1]}] Every increasing sequence in $S$ has a supremum.
\item[{\crtcrossreflabel{(O2)}[axiom: O2]}] If $s \in S$ then there exists a sequence $\{ s_n \}_{n > 0}$ in $S$ such that $s_m \ll s_{m+1}$ for $m > 0$ and  $s = \sup_{n > 0} s_n$.
\item[{\crtcrossreflabel{(O3)}[axiom: O3]}] If $x, x', y, y' \in S$, $x' \ll x$ and $y' \ll y$ then $x' + y' \ll x + y$. 
\item[{\crtcrossreflabel{(O4)}[axiom: O4]}] If $\{ f_n \}_{n > 0}$ and $\{ g_n \}_{n > 0}$ are increasing sequences in $S$ then $${\sup}_{n > 0} \; f_n + {\sup}_{n > 0} \; g_n = {\sup}_{n > 0} \; (f_n + g_n).$$
\end{enumerate}
\end{definition}

As first unveiled in \cite[Theorem 1]{CEI08}, the archetypal example of a Cu-semigroup is the Cuntz semigroup of a C*-algebra, which is defined using a notion of equivalence of positive elements first considered in \cite[Section 1]{Cun78}. 

\begin{definition}
Let $A$ be a C*-algebra and $a, b \in A_+$. The element $a$ is \textbf{Cuntz subequivalent} to or Cuntz below $b$ (denoted by $a \precsim b$) if for $\epsilon > 0$, there exists $r \in A$ such that $\lVert a - r b r^* \rVert < \epsilon$. Moreover, $a$ is \textbf{Cuntz equivalent} to $b$ (denoted by $a \sim b$) if $a \precsim b$ and $b \precsim a$. \newline
\phantom{} \newline
The \textbf{Cuntz semigroup} of $A$, denoted by $\Cu(A)$, is defined to be the set
\begin{equation*}
\Cu(A) = (A \otimes K)_+ / \sim \; = \{ [a] \; \vert \; a \in (A \otimes K)_+ \}.
\end{equation*}
If $[a], [b] \in \Cu(A)$ then the partial order and addition on $\Cu(A)$ are given respectively by setting $[a] \leq [b]$ if and only if $a \precsim b$ and
\begin{equation*}
[a] + [b] = [a \oplus b] = \left[ 
\begin{pmatrix}
a & 0 \\
0 & b
\end{pmatrix}
\right].
\end{equation*}
\end{definition}

The notion of an ideal of a Cu-semigroup respects both the monoid and Cu-semigroup structures. 

\begin{definition}
Let $S$ be a Cu-semigroup and $I$ be a submonoid of $S$. $I$ is said to be an \textbf{ideal} of $S$ if it satisfies the following properties:
\begin{enumerate}
\item $I$ is downwards-hereditary. That is, if $x, y \in S$, $x \leq y$ and $y \in I$ then $x \in I$.
\item $I$ is closed under suprema of increasing sequences. That is, if $\{ x_n \}_{n > 0}$ is an increasing sequence in $I$ then its supremum must also lie in $I$.
\end{enumerate}
The lattice of ideals of $S$ is denoted by $\Lat(S)$. 
\end{definition}

\begin{definition} \label{definition: Countably Based}
A Cu-semigroup $S$ is said to be \textbf{countably based} if there exists a countable subset $X \subseteq S$ such that every element in $S$ is the supremum of a rapidly increasing sequence of elements in $X$. Such a subset $X$, if it exists, is called \textbf{sup-dense} in $S$.
\end{definition}

If $S$ is a Cu-semigroup, then any element $x \in S$ gives rise to the ideal
\begin{equation} \label{equation: Singly Gen Ideal}
\Idl(x) = \{ y \in S \; \vert \; y \leq \infty \cdot x = {\sup}_{k > 0} \; k x \}. 
\end{equation}
Moreover, if $S$ is countably based then by \cite[Proposition 5.1.7]{APT18}, every ideal of $S$ is of the form $\Idl(x)$ for some $x \in S$. By \cite[Lemma 1.3]{APS11}, the Cuntz semigroup of a separable C*-algebra is always countably based. \newline
\phantom{} \newline
Below, we isolate and state a consequence of the proof of \cite[Lemma 1.3]{APS11} which will be needed for the proof of Theorem \ref{introtheorem: General Cuntz Krieger} in the specific case where the vertex $v$ supports a single loop. If $\epsilon > 0$ and $x$ is a positive element of a C*-algebra then we define $(x - \epsilon)_+ = f_\epsilon(x)$ where $f_\epsilon: [0, \infty) \rightarrow [0, \infty)$ is the function $f_\epsilon(\lambda) = \max \{ \lambda - \epsilon, 0 \}$.

\begin{lemma}[{see \cite[Lemma 1.3]{APS11}}] \label{lemma: Countably Based}
Let $A$ be a separable C*-algebra and $X$ be a countable dense subset of $A_+$. If $a, b \in A_+$ and $[a] \ll [b]$ then there exists $x \in X$ and $m > 0$ such that in $\Cu(A)$,
\begin{equation*}
[a] \ll \Big[ \Big( x - \frac{1}{m} \Big)_+ \Big] \ll [b].
\end{equation*}
\end{lemma}

In the Cuntz semigroup of a C*-algebra $A$, it is often useful to keep track of the Cuntz classes of elements in $A_+$. One way to do this is to consider the subset 
\begin{equation} \label{equation: Scale}
\Sigma_A = \big\{ [x] \in \Cu(A) \; \vert \; \text{There exists $a \in A_+$ such that $[x] \leq [a]$} \big\}
\end{equation}
which was shown in \cite[Lemma 3.3 (2)]{TV23} to satisfy the nice properties listed in \cite[Definition 4.1]{APT20}. Below, we record the statement of \cite[Lemma 3.3 (1)]{TV23}, which we will use in the proof of Theorem \ref{introtheorem: General Cuntz Krieger}. 

\begin{lemma}[{\cite[Lemma 3.3 (1)]{TV23}}] \label{lemma: Sandwich}
Let $A$ be a C*-algebra, $x \in \Sigma_A$ and $x' \in \Cu(A)$ be such that $x' \ll x$. Then there exists $a \in A_+$ such that $x' \ll [a] \ll x$. 
\end{lemma}

The definition of a quotient Cu-semigroup is motivated by the concrete result \cite[Theorem 4.1]{CRS10}. We will state a shorter version of \cite[Theorem 4.1]{CRS10} that is suitable for the paper. 

\begin{theorem}[{\cite[Theorem 4.1]{CRS10}}] \label{theorem: CRS}
Let 
\begin{equation*}
\begin{tikzcd}
0 \arrow[r] & I \arrow[r, "\iota"] & A \arrow[r, "\pi"] & A/I \arrow[r] & 0
\end{tikzcd}
\end{equation*}
be an extension of C*-algebras and suppose that $I$ is a $\sigma$-unital ideal of $A$. The induced sequence of positively ordered monoids
\begin{equation*}
\begin{tikzcd}
0 \arrow[r] & \Cu(I) \arrow[r, "\Cu(\iota)"] & \Cu(A) \arrow[r, "\Cu(\pi)"] & \Cu(A/I) \arrow[r] & 0
\end{tikzcd}
\end{equation*}
is short exact. Furthermore, if $x, y \in \Cu(A)$ and $\Cu(\pi)(x) \leq \Cu(\pi)(y)$ then there exists $z \in \Cu(I)$ such that $x \leq y + z$ in $\Cu(A)$.
\end{theorem}

Specifically, the second half of Theorem \ref{theorem: CRS} is the blueprint for the notion of a quotient Cu-semigroup. 

\begin{definition} \label{definition: Quotient Cu Semigroup}
 Let $S$ be a Cu-semigroup and $I$ be an ideal of $S$. If $a, b \in S$ then we say $a \leq_I b$ if and only if there exists $i \in I$ such that $a \leq b + i$. Define an equivalence relation $\sim_I$ by setting $a \sim_I b$ if and only if $a \leq_I b$ and $b \leq_I a$. The \textbf{quotient Cu-semigroup} $S/I$ is given by the set
\begin{equation*}
S/I = S/\sim_I \; = \{ [s]_I \; \vert \; s \in S \}.
\end{equation*}
The partial order on $S/I$ is given by $[a]_I \leq [b]_I$ if and only if $a \leq_I b$ in $S$ and addition is given by $[a]_I + [b]_I = [a + b]_I$. 
\end{definition}

By \cite[Lemma 5.1.2]{APT18}, a quotient Cu-semigroup is a Cu-semigroup. Using Definition \ref{definition: Quotient Cu Semigroup}, Theorem \ref{theorem: CRS} tells us that if $I$ is a $\sigma$-unital ideal of $A$, then as Cu-semigroups, $\Cu(A/I)$ is isomorphic to the quotient Cu-semigroup $\Cu(A)/\Cu(I)$. 

\subsection{Notation and moves for graph C*-algebras}
\label{subsec: Moves}

For an accessible reference on graph C*-algebras, the reader is referred to the book \cite{Rae05}. A directed graph is \textbf{row-finite} if there are finitely many edges going into any vertex. \newline
\phantom{} \newline
Note that in this paper, we will adopt the notation and conventions for graph C*-algebras set out in \cite{Rae05}. In particular, if $E$ is a row-finite directed graph, $E^0$ is the set of vertices in $E$ and $E^1$ is the set of edges in $E$, then the graph C*-algebra $C^*(E)$ is the universal C*-algebra generated by the set $\{ s_e, p_v \; \vert \; e \in E^1, \; v \in E^0 \}$ consisting of partial isometries $s_e$ and mutually orthogonal projections $p_v$ which satisfy the Cuntz-Krieger relations:
\begin{enumerate}
\item If $e \in E^1$ then $p_{s(e)} = s_e^* s_e$. 
\item If $v \in E^0$ and $0 < \lvert r^{-1}(v) \rvert$ then $p_v = \sum_{r(e) = v} s_e s_e^*$. 
\end{enumerate}
We will mainly focus on finite graphs: \textbf{graphs with finitely many vertices and edges}. Note that a graph C*-algebra has finitely many vertices if and only if it is unital by \cite[Proposition 1.4]{KPR98}. \newline
\phantom{} \newline
Akin to Reidemeister moves for knots, moves for unital graph C*-algebras are special methods of changing the underlying graph of a unital graph C*-algebra in such a way that certain structures are preserved, such as its diagonal (see for instance, \cite[Definition 2.1]{NR12}) or its gauge action (see for instance, \cite[Proposition 2.1]{Rae05}). The first two moves, outsplitting and insplitting, were first studied in \cite{Wil73} with the goal of classifying subshifts of finite type up to topological conjugacy. In this paper, we will use moves (O), (I-), (CO) and (S) as defined in \cite[Section 2]{ER25} and \cite[Section 5]{Sor13}. Of particular importance is the fact that applying these moves results in a pair of graph C*-algebras that are at least, stably isomorphic. We state this result as Proposition \ref{proposition: Refined Moves Stable Iso} below; the proof is contained in \cite[Theorem 2.1.2]{ER25}, \cite[Theorem 2.2.4]{ER25}, \cite[Theorem 2.3.2]{ER25} and \cite[Theorem 2.3.5]{ER25} and uses the fact that move (CO) by definition is given by applying move (R+) from \cite[Definition 2.3.1]{ER25} and then move (S). 

\begin{proposition} \label{proposition: Refined Moves Stable Iso}
Let $E$ and $F$ be finite graphs and suppose that $F$ is obtained by applying a sequence of the moves (O), (I-), (CO) and (S) to $E$. Then $C^*(E) \otimes K \cong C^*(F) \otimes K$. In particular, $\Cu(C^*(E))$ is isomorphic to $\Cu(C^*(F))$.
\end{proposition}
Since we have adopted opposing conventions to the paper \cite{ER25}, we will carefully define moves (O), (I-), (CO) and (S) and the necessary terminology. \newline
\phantom{} \newline
Let $E$ be a row-finite graph and $v, w \in E^0$. We say that $v$ is \textbf{singular} if $v$ is either a source or an infinite receiver (receives infinitely many edges). Otherwise, we say that $v$ is \textbf{regular}.   \newline
\phantom{} \newline
Let $E^*$ denote the set of finite paths in $E$. We say that $v \geq w$ if there exists $\mu \in E^*$ such that $s(\mu) = v$ and $r(\mu) = w$. A \textbf{cycle} is a finite length path $\mu = \mu_n \dots \mu_2 \mu_1 \in E^*$ such that $r(\mu_n) = s(\mu_1)$ and for $j \in \{2, 3, \dots, n \}$, $s(\mu_j) \neq s(\mu_1)$. We say that the cycle $\mu$ is a \textbf{loop} if $n = 1$ (the cycle has length one).  A cycle $\mu = \mu_n \dots \mu_1 \in E^*$ is said to have an \textbf{entry} if there exist $e \in E^1$ and $i \in \{1, 2, \dots, n \}$ such that $r(e) = r(\mu_i)$ but $e \neq \mu_i$. \newline
\phantom{} \newline
Moreover, let $M$ be a subset of $E^0$ and suppose that if $j \in \{1, 2, \dots, n \}$, then $r(\mu_j) \in M$. The cycle $\mu$ has an \textbf{entry in $M$} if there exist an edge $e \in E^1$ and $i \in \{1, 2, \dots, n \}$ as above with $s(e) \in M$. Finally, a vertex $v \in E^0$ is called \textbf{transitional} if it is not a sink and does not support a cycle.  \newline
\phantom{} \newline
The outsplitting move (O) for unital graph C*-algebras was first defined in \cite[Section 3]{BP04}. It was proved in \cite[Theorem 3.2]{BP04} that applying move (O) induces a *-isomorphism between the resulting pair of graph C*-algebras which preserves the gauge action. Moreover, the *-isomorphism is also diagonal preserving by the proof of \cite[Theorem 2.1.2]{ER25}.

\begin{definition}[Move (O): Outsplitting] \label{definition: O}
Let $E$ be a graph with finitely many vertices and $w \in E^0$ be a vertex that is not a source. Partition the set $r^{-1}(w)$ into a disjoint union of finitely many non-empty sets so that
\begin{equation*}
r^{-1}(w) = \mathcal{E}_1 \sqcup \mathcal{E}_2 \sqcup \dots \sqcup \mathcal{E}_n
\end{equation*}
and at most one of the $\mathcal{E}_i$ has infinite cardinality. Define the graph $E_O$ by
\begin{equation*}
E_O^0 = \{ v^1 \; \vert \; \text{$v \in E^0$ and $v \neq w$} \} \cup \{ w^1, \dots, w^n \},
\end{equation*} 
\begin{equation*}
E_O^1 = \{ e^1 \; \vert \; \text{$e \in E^1$ and $s(e) \neq w$} \} \cup \{ e^1, \dots, e^n \; \vert \; \text{$e \in E^1$ and $s(e) = w$} \},
\end{equation*}
\vspace{-1 mm}
\begin{equation*}
s_O(e^i) = 
\begin{cases}
s(e)^1, \text{ if $e \in E^1$ and $s(e) \neq w$,} \\
w^i, \text{ if $e \in E^1$ and $s(e) = w$.}
\end{cases}
\end{equation*} \vspace{1 mm}
\begin{equation*}
r_O(e^i) = 
\begin{cases}
r(e)^1, \text{ if $e \in E^1$ and $r(e) \neq w$,} \\
r(e)^j, \text{ if $e \in E^1$, $r(e) = w$ and $e \in \mathcal{E}_j$.}
\end{cases}
\end{equation*}
The graph $E_O$ is said to be obtained by applying move (O) to $E$. 
\end{definition}

Move (I-) was first introduced in \cite[Definition 2.2.1]{ER25} as a generalisation of the earlier insplitting move from \cite[Section 5]{BP04}. Unlike move (O), the insplitting move (I-) only preserves the graph C*-algebra up to stable isomorphism (see for instance, \cite[Theorem 2.2.4]{ER25}). By Proposition \ref{proposition: Refined Moves Stable Iso}, this is nevertheless enough for our purposes. 

\begin{definition}[Move (I-): Insplitting] \label{definition: I-}
Let $E$ be a graph with finitely many vertices and $w \in E^0$ be a regular vertex (not a source or an infinite receiver). Partition the set $s^{-1}(w)$ into a disjoint union of finitely many sets (some of which could be empty) so that
\begin{equation*}
s^{-1}(w) = \mathcal{E}_1 \sqcup \mathcal{E}_2 \sqcup \dots \sqcup \mathcal{E}_n.
\end{equation*}
Define the graph $E_I$ by
\begin{equation*}
E_I^0 = \{ v^1 \; \vert \; \text{$v \in E^0$ and $v \neq w$} \} \cup \{ w^1, \dots, w^n \},
\end{equation*} 
\begin{equation*}
E_I^1 = \{ e^1 \; \vert \; \text{$e \in E^1$ and $r(e) \neq w$} \} \cup \{ e^1, \dots, e^n \; \vert \; \text{$e \in E^1$ and $r(e) = w$} \},
\end{equation*} \vspace{-1 mm}
\begin{equation*}
s_I(e^i) = 
\begin{cases}
s(e)^1, \text{ if $e \in E^1$ and $s(e) \neq w$,} \\
w^j, \text{ if $e \in E^1$, $s(e) = w$ and $e \in \mathcal{E}_j$.}
\end{cases}
\end{equation*} \vspace{1 mm}
\begin{equation*}
r_I(e^i) = 
\begin{cases}
r(e)^1, \text{ if $e \in E^1$ and $r(e) \neq w$,} \\
w^i, \text{ if $e \in E^1$ and $r(e) = w$.}
\end{cases}
\end{equation*}
The graph $E_I$ is said to be obtained by applying move (I-) to $E$. 
\end{definition}

The move (S), first introduced in \cite[Proposition 3.1]{Sor13}, simply removes a regular sink from the graph. 

\begin{definition}[Move (S): Remove a regular sink] \label{definition: S}
Let $E$ be a a graph with finitely many vertices and $w \in E^0$ be a regular sink. Define the graph $E_S$ by
\begin{equation*}
E_S^0 = E^0 \backslash \{ w \}, \quad E_S^1 = E^1 \backslash r^{-1}(w), \quad s_S = s \vert_{E_S^1} \quad \text{and} \quad r_S = r \vert_{E_S^1}. 
\end{equation*}
The graph $E_S$ is said to be obtained by applying move (S) to $E$. 
\end{definition}

The move (CO), called \emph{collapse} in \cite[Section 5.1]{Sor13}, was originally defined in \cite[Theorem 5.2]{Sor13} and can be applied to any regular vertex which does not support a loop. Visually, the move (CO) removes the regular vertex it is applied to, merging any pair of edges which forms a path of length two passing through the regular vertex into a single edge. Move (CO) is closely connected to the unital reduction move (R+) which was first defined in \cite[Definition 14.4]{ERRS21}; the only difference between the two is that applying move (R+) also replaces the regular vertex it is applied to with a regular sink. Thus as mentioned previously, move (CO) can be thought of as first applying move (R+) and then removing the regular sink introduced by (R+) with move (S). 

\begin{definition}[Move (CO): Collapse] \label{definition: CO}
Let $E$ be a graph with finitely many vertices and $w \in E^0$ be a regular vertex which does not support a loop. Define the graph $E_{\text{CO}}$ by $E_{\text{CO}}^0 = E^0 \backslash \{ w \}$ and
\begin{equation*}
E_{\text{CO}}^1 = \big(E^1 \backslash (s^{-1}(w) \cup r^{-1}(w)) \big) \; \sqcup \;  \{ [ef] \; \vert \; e \in s^{-1}(w), \; f \in r^{-1}(w) \}.
\end{equation*}
The source and range maps of $E_{\text{CO}}$ extend those of $E$ and satisfy
\begin{equation*}
s_{\text{CO}}([ef]) = s(f) \qquad \text{and} \qquad  \quad r_{\text{CO}}([ef]) = r(e). 
\end{equation*}
The graph $E_{\text{CO}}$ is said to be obtained by applying move (CO) to $E$. 
\end{definition}

\subsection{The ideal lattice of a row-finite graph C*-algebra} 
\label{subsec: Ideal Lattice Prelims}

In this subsection, we collect various definitions about ideals of graph C*-algebras from \cite{BPRS00} and \cite{HS04}. We begin with the gauge-invariant ideals --- those that are invariant under the gauge action of the graph C*-algebra. \newline
\phantom{} \newline
Let $E$ be a row-finite graph and $H$ be a subset of $E^0$. $H$ is said to be \textbf{hereditary} if for $x, y \in E^0$, if $x \geq y$ and $y \in H$ then $x \in H$. Moreover, $H$ is said to be \textbf{saturated} if for $v \in E^0$ such that $0 < \lvert r^{-1}(v) \rvert$ and $\{ s(e) \; \vert \; e \in E^1, \; r(e) = v \} \subseteq H$, $v \in H$. By \cite[Theorem 4.1]{BPRS00}, there is a bijective correspondence:
\begin{equation} \label{equation: Hereditary Saturated}
\begin{matrix}
\left\{ 
\begin{matrix}
\text{Gauge-invariant ideals} \\
\text{of $C^*(E)$}
\end{matrix}
\right\} & \longleftrightarrow & \left\{ 
\begin{matrix}
\text{Hereditary, saturated subsets} \\
\text{of $E^0$}
\end{matrix}
\right\} \\
I & \longmapsto & H_I = \{ v \in E^0 \; \vert \; p_v \in I \} \\
I_H = \langle p_v \; \vert \; v \in H \rangle & \longmapsfrom & H.
\end{matrix}
\end{equation}
To be clear, $I_H$ is the ideal generated by the set $\{ p_v \; \vert \; v \in H \}$. Moreover, if we define $p_S = \sum_{v \in S} p_v$ for an arbitrary subset $S \subseteq E^0$, then
\begin{equation} \label{equation: Full Corner Iso}
C^*(H) \cong p_H I_H p_H
\end{equation} 
where the RHS of \eqref{equation: Full Corner Iso} is a full corner of $I_H$. A major consequence of \cite[Theorem 4.1]{BPRS00} is that any graph C*-algebra coming from a finite graph can be written as iterative extensions involving three building blocks:
\begin{enumerate}
\item C*-algebras stably isomorphic to a direct sum of finitely many copies of $K$, 
\item Kirchberg (separable, nuclear, simple purely infinite) algebras,
\item C*-algebras stably isomorphic to $C(\mathbb{T})$.
\end{enumerate}

\begin{definition}
Let $E$ be a row-finite graph. The graph C*-algebra $C^*(E)$ is said to be \textbf{gauge-simple} if the only gauge-invariant ideals of $C^*(E)$ are $\{ 0 \}$ and $C^*(E)$. Equivalently by \eqref{equation: Hereditary Saturated}, the only hereditary, saturated subsets of $E^0$ are $\emptyset$ and $E^0$. 
\end{definition}
The non-gauge-invariant ideals of a graph C*-algebra in general do not correspond to certain subsets of vertices of the graph. Fortunately, primitive non-gauge-invariant ideals can be recovered from certain subsets of vertices called maximal tails. 

\begin{definition}[{\cite[Page 321]{BPRS00}}]
Let $E$ be a graph and $M \subseteq E^0$ be non-empty. $M$ is said to be a \textbf{maximal tail} if the following three conditions are satisfied:
\begin{enumerate}
\item[{\crtcrossreflabel{(M1)}[property: M1]}] If $v \in E^0$, $w \in M$ and there exists a path from $w$ to $v$ then $v \in M$.
\item[{\crtcrossreflabel{(M2)}[property: M2]}] If $v \in M$ and $0 < \lvert r^{-1}(v) \rvert < \infty$ then there exists $e \in E^1$ such that $r(e) = v$ and $s(e) \in M$. 
\item[{\crtcrossreflabel{(M3)}[property: M3]}] If $v, w \in M$ then there exists $y \in M$ such that $y \geq v$ and $y \geq w$. 
\end{enumerate}
\end{definition}

It is straightforward to check that if $M$ is a maximal tail in $E$, then by property \ref{property: M1}, $E^0 \backslash M$ is a hereditary subset of $E^0$. 

\begin{definition}[{\cite[Page 49]{HS04}}] \label{definition: Maximal Tail Tau}
Let $E$ be a graph. Define $\mathcal{MT}(E)$ to be the set of all maximal tails in $E$. Define $\mathcal{MT}_\gamma(E)$ to be the set of all maximal tails in $E$ such that each cycle in $M$ has an entry in $M$. Finally, define $\mathcal{MT}_\tau(E) = \mathcal{MT}(E) \backslash \mathcal{MT}_\gamma(E)$. 
\end{definition}

By \cite[Corollary 2.12]{HS04}, if $E$ is a row-finite graph, then the primitive ideal space of $C^*(E)$ is in bijection with the set $\mathcal{MT}_\gamma(E) \cup (\mathcal{MT}_\tau(E) \times \mathbb{T})$. In particular, the primitive non-gauge-invariant ideals of $C^*(E)$ correspond to the set $\mathcal{MT}_\tau(E) \times \mathbb{T}$. 

\subsection{Purely large extensions}
\label{subsec: Purely Large}

The notion of a purely large extension was first defined by Elliott and Kucerovsky in \cite[Section 2]{EK01}. In this paper, we are interested in the consequences pure largeness has for the Cuntz semigroup of an extension. Hence, we will use the definition of a purely large extension stated in \cite[Definition 4.8]{BG24}, instead of the original one in \cite[Section 2]{EK01}. By \cite[Proposition 4.14]{BG24} (which was originally proved for the paper \cite{CGSTW}), the two definitions coincide whenever the ideal is $\sigma$-unital; this will be the case throughout the paper. In the context of this paper, pure largeness is relied on in \cref{subsec: Kirchberg Quotient} and \cref{subsec: Circle Quotient} to establish Theorem \ref{introtheorem: General Cuntz Krieger}: an important special case of Theorem \ref{introtheorem: Cu E Iso}. 

\begin{definition}[{\cite[Definition 4.8]{BG24} and \cite{CGSTW}}] \label{definition: Purely Large Bouwen Gabe}
Let 
\begin{equation*}
\begin{tikzcd}
0 \arrow[r] & I \arrow[r] & A \arrow[r] & A/I \arrow[r] & 0
\end{tikzcd}
\end{equation*}
be an extension of C*-algebras. The extension is said to be \textbf{purely large} if for $y \in I_+$ and $x \in A_+ \backslash I$, $y \precsim x$ in $A$. 
\end{definition}

It was proved by Gabe that pure largeness is preserved by tensoring by a simple, exact C*-algebra --- this result will appear in \cite{CGSTW}. In this paper, we only require this result in the special case where we tensor by $K$. The author is grateful to Jamie Gabe for allowing Theorem \ref{theorem: Purely Large Tensor Product} to be included and used in this paper. 

\begin{theorem}[Gabe] \label{theorem: Purely Large Tensor Product}
Let 
\begin{equation*}
\begin{tikzcd}
0 \arrow[r] & I \arrow[r] & A \arrow[r] & A/I \arrow[r] & 0
\end{tikzcd}
\end{equation*}
be a purely large extension. Then the stabilised extension
\begin{equation*}
\begin{tikzcd}
0 \arrow[r] & I \otimes K \arrow[r] & A \otimes K \arrow[r] & A/I \otimes K \arrow[r] & 0
\end{tikzcd}
\end{equation*}
is also purely large.
\end{theorem}

\begin{proof}
Assume that $x \in (A \otimes K)_+ \backslash (I \otimes K)$ and let $\pi: A \otimes K \rightarrow A/I \otimes K$ denote the quotient map. We need to show that $x$ Cuntz dominates any positive element in $I \otimes K$. By applying Kirchberg's slice lemma (see \cite[Lemma 2.15]{BK04}), we will first reduce the problem to the special case where $x$ is an elementary tensor. \newline
\phantom{} \newline
Since $K$ is nuclear, there exist faithful representations $A/I \subseteq B(H_{A/I})$ and $K \subseteq B(H_K)$ such that $A/I \otimes K \subseteq B(H_{A/I} \otimes H_K)$ faithfully; see \cite[Proposition 1.22.9]{Sak71}. Since $\pi(x)$ is non-zero, there exist non-zero $\zeta_{A/I} \in H_{A/I}$ and $\zeta_K \in H_K$ such that $$\lVert \pi(x)^{\frac{1}{2}}(\zeta_{A/I} \otimes \zeta_K) \rVert^2 = \langle \pi(x) (\zeta_{A/I} \otimes \zeta_K), (\zeta_{A/I} \otimes \zeta_K) \rangle > 0.$$ Let $\phi_{A/I}$ and $\phi_K$ be the functionals induced by $\zeta_{A/I}$ and $\zeta_K$ respectively. Then
\begin{equation*}
(\phi_{A/I} \otimes \phi_K)(\pi(x)) = \langle \pi(x)(\zeta_{A/I} \otimes \zeta_K), \zeta_{A/I} \otimes \zeta_K \rangle = \lVert \pi(x)^{\frac{1}{2}}(\zeta_{A/I} \otimes \zeta_K) \rVert^2 > 0.  
\end{equation*}
Hence, there exist pure states $\rho_{A/I}$ and $\rho_K$ on $A/I$ and $K$ respectively such that $(\rho_{A/I} \otimes \rho_K)(\pi(x)) > 0$, as otherwise by \cite[Theorem 5.1.8]{Mur90}, $(\eta_{A/I} \otimes \eta_K)(\pi(x)) = 0$ for arbitrary positive linear functionals $\eta_{A/I}$ and $\eta_K$ on $A/I$ and $K$. \newline
\phantom{} \newline
Now let $\rho_A$ be the pure state on $A$ induced from $\rho_{A/I}$. Then $\rho_A(I) = \{ 0 \}$ and $(\rho_A \otimes \rho_K)(x) = (\rho_{A/I} \otimes \rho_K)(\pi(x)) > 0$. By Kirchberg's slice lemma, there exist $z \in (A \otimes K) \backslash \{ 0 \}$, $e \in A_+$ and $d \in K_+$ such that
\begin{equation*}
zz^* \in \overline{x ( A \otimes K) x}, \qquad z^* z = e \otimes d \qquad \text{and} \qquad \rho_A(e) \rho_K(d) > 0.
\end{equation*} 
Observe that $e \in A_+ \backslash I$ and by \cite[Proposition 2.7 (i)]{KR00}, $e \otimes d \sim z z^* \precsim x$. Consequently by Rørdam's lemma (see \cite[Proposition 2.4]{Ror92}), it suffices to show that $e \otimes d$ Cuntz dominates all positive elements of the form $(y - \delta)_+$ where $y \in (I \otimes K)_+$ and $\delta > 0$. \newline
\phantom{} \newline
To this end, assume that $y \in (I \otimes K)_+$ and $\delta > 0$. Using approximate units on both $I$ and $K$, there exist $a \in I_+$ and $b \in K_+$ such that 
\begin{equation*}
\lVert (a \otimes b) y (a \otimes b) - y \rVert < \frac{\delta}{2}. 
\end{equation*}
So by choosing $\epsilon > 0$ small enough, we can arrange for
\begin{equation*}
\lVert (a \otimes (b - \epsilon)_+) y (a \otimes (b - \epsilon)_+) - y \rVert < \delta. 
\end{equation*} 
Applying \cite[Lemma 2.5 (ii)]{KR00}, we have $(y - \delta)_+ \precsim a \otimes (b - \epsilon)_+$. Consequently, it suffices to show that if $a \in I_+$, $b \in K_+$ and $\epsilon > 0$ then
\begin{equation} \label{equation: Final Reduction}
a \otimes (b - \epsilon)_+ \precsim e \otimes d. 
\end{equation}
First, since $\Cu(K) \cong \overline{\mathbb{N}}$ is simple then it follows from \cite[Proposition 2.7 (v)]{KR00} that there exists $n > 0$ large enough such that 
\begin{equation} \label{equation: Simplicity Application}
(b - \epsilon)_+ \otimes e_{11} \sim (b - \epsilon)_+ \precsim d \otimes 1_n
\end{equation}
where $1_n$ is the unit of $M_n(\mathbb{C})$ and $e_{11} \in M_n(\mathbb{C})$ is the matrix with a 1 in the top left corner and zeros elsewhere. Now since $A$ is a purely large extension then by \cite[Proposition 4.14]{BG24} and \cite[Proposition 6.10]{CGSTW26}, $I$ is separably stable and so there exists a separable stable C*-subalgebra $C$ of $I$ such that $a \in C$. By stability of $C$, there is a unital embedding of $\mathcal{O}_n$ into the multiplier algebra $\mathcal{M}(C)$, which yields isometries $s_1, \dots, s_n \in \mathcal{M}(C)$ such that if $i \in \{1, 2, \dots, n \}$ then $$s_i^* s_i = {\sum}_{j = 1}^n s_j s_j^* = 1_{\mathcal{M}(C)}.$$ Now define $a' = \sum_{j = 1}^n s_j a s_j^* \in C_+$ and for $k > 0$, $$v_k = {\sum}_{j = 1}^n s_j a^{\frac{1}{k}} \otimes e_{1j} \in C \otimes M_n(\mathbb{C}).$$ By direct computation, $v_k^* (a' \otimes e_{11}) v_k \rightarrow a \otimes 1_n$ as $k \rightarrow \infty$. Hence by Definition \ref{definition: Purely Large Bouwen Gabe}, $a \otimes 1_n \precsim a' \otimes e_{11} \precsim e \otimes e_{11}$. In tandem with \eqref{equation: Simplicity Application}, we have in $(A \otimes K) \otimes M_n(\mathbb{C})$ that
\begin{equation*}
a \otimes (b - \epsilon)_+ \otimes e_{11} \precsim a \otimes d \otimes 1_n \precsim e \otimes d \otimes e_{11}. 
\end{equation*}
We obtain the desired Cuntz subequivalence in \eqref{equation: Final Reduction}. 
\end{proof}

The following corollary of Theorem \ref{theorem: Purely Large Tensor Product} tells us that pure largeness imposes a strong restriction on the partial order of the Cuntz semigroup of the extension. It is an essential ingredient in the Cuntz semigroup computations carried out in \cref{subsec: Kirchberg Quotient} and motivates the construction of the stack of two Cu-semigroups; see Definition \ref{definition: Cuntz Stack}. 

\begin{corollary} \label{corollary: Purely Large Order Stabilised}
Let
\begin{equation*}
\begin{tikzcd}
0 \arrow[r] & I \arrow[r] & A \arrow[r] & A/I \arrow[r] & 0
\end{tikzcd}
\end{equation*} 
be an extension of C*-algebras. The extension is purely large if and only if for $y \in \Cu(I)$ and $x \in \Cu(A) \backslash \Cu(I)$, $y \leq x$ in $\Cu(A)$. 
\end{corollary}
Recall that an extension $A$ of $A/I$ by $I$ is \textbf{full} if the associated Busby invariant $\beta: A/I \rightarrow \mathcal{M}(I)/I$ is a full *-homomorphism; that is, if $x \in A/I$ is non-zero then the ideal generated by $\beta(x)$ is equal to $\mathcal{M}(I)/I$. Furthermore, a C*-algebra $B$ is said to have the \textbf{corona factorisation property} (first defined in \cite[Definition 1.3]{KN06a}) if every norm-full projection $p$ in the stable multiplier algebra $\mathcal{M}(B \otimes K)$ is Murray-von Neumann equivalent to the unit $1_{\mathcal{M}(B \otimes K)}$. \newline
\phantom{} \newline
Later on in the paper, we will rely on Theorem \ref{theorem: Kucerovsky Ng} below to show that the extensions of graph C*-algebras coming from the statement of Theorem \ref{introtheorem: General Cuntz Krieger} are purely large. The proof of Theorem \ref{theorem: Kucerovsky Ng} is contained in the proof of \cite[Theorem 3.5]{KN06b}, which in turn relies on the deep connection between absorbing extensions and purely large extensions established by Elliott and Kucerovsky in \cite{EK01}.
\begin{theorem}[from {\cite[Theorem 3.5]{KN06b}}] \label{theorem: Kucerovsky Ng}
Let 
\begin{equation*}
\begin{tikzcd}
0 \arrow[r] & I \arrow[r] & A \arrow[r] & A/I \arrow[r] & 0
\end{tikzcd}
\end{equation*} 
be an extension of C*-algebras. Suppose that the ideal $I$ is $\sigma$-unital, stable and satisfies the corona factorisation property. If $A$ is a full extension then it is purely large. 
\end{theorem}

Proposition \ref{proposition: Graph CFP} ensures that Theorem \ref{theorem: Kucerovsky Ng} can be applied to the extensions involving the graph C*-algebras considered in this paper.  

\begin{proposition} \label{proposition: Graph CFP}
Let $E$ be a finite graph. Then $C^*(E)$ and $C^*(E) \otimes K$ have the corona factorisation property. 
\end{proposition}

\begin{proof}
Since every graph C*-algebra coming from a finite graph is built by iteratively taking extensions involving AF-algebras, Kirchberg algebras and C*-algebras stably isomorphic to $C(\mathbb{T})$ then by \cite[Proposition 2.9]{WZ10}, any graph C*-algebra associated to a finite graph must have finite nuclear dimension. By \cite[Corollary 2.8 (i)]{WZ10}, the stabilisation of such a graph C*-algebra also has finite nuclear dimension. Thus by \cite[Corollary 3.5]{Rob11}, a graph C*-algebra associated to a finite graph and its stabilisation have the corona factorisation property. 
\end{proof}

\subsection{Results on purely infinite C*-algebras}
\label{subsec: PI}

From the three fundamental building blocks of a graph C*-algebra, a gauge-simple graph C*-algebra is either simple purely infinite or has stable rank one. Since pure infiniteness and stable finiteness are diametrically opposed properties, we require different sets of tools to handle these two types of gauge-simple ideals and quotients which appear throughout the paper; see \cref{subsec: Kirchberg Quotient}, \cref{subsec: Circle Quotient} and the proof of Theorem \ref{introtheorem: Cu E Iso} at the end of \cref{subsec: Proof Main Theorem} for instance. In this subsection, we record results about purely infinite C*-algebras which we will use later in the paper. The key result we will use concerning stable rank one C*-algebras will be developed in \Cref{sec: Weak Cancellation}. \newline
\phantom{} \newline
Recall that if $A$ is a C*-algebra and $a \in A_+$ is non-zero, then $a$ is said to be \textbf{properly infinite} if 
\begin{equation*}
\begin{pmatrix}
a & 0 \\
0 & a 
\end{pmatrix} = a \oplus a \precsim a. 
\end{equation*}
In \cite[Theorem 4.16]{KR00}, Kirchberg and Rørdam provide a useful characterisation of pure infiniteness in terms of properly infinite elements.

\begin{theorem}[{\cite[Theorem 4.16]{KR00}}] \label{theorem: Purely Infinite Properly Infinite}
Let $A$ be a C*-algebra. Then $A$ is purely infinite if and only if every non-zero positive element of $A$ is properly infinite. 
\end{theorem}

In \autoref{subsec: Kirchberg Quotient}, we will consider purely large extensions with a purely infinite quotient. A key lemma we will use to understand such extensions is the following slight generalisation of \cite[Lemma 4.14]{KR00}.

\begin{lemma}[from {\cite[Lemma 4.14]{KR00}}] \label{lemma: KR Lemma}
Let
\begin{equation*}
\begin{tikzcd}
0 \arrow[r] & I \arrow[r] & A \arrow[r] & B \arrow[r] & 0 
\end{tikzcd}
\end{equation*}
be an extension of C*-algebras. Assume that all non-zero positive elements in $B$ are properly infinite. If $a \in A_+$ is non-zero, $n \in \{2, 3, \dots \}$ and $\epsilon > 0$ then there exist mutually orthogonal, positive elements $a_1, a_2, \dots, a_n \in \overline{aAa}$ and a positive element $c \in \overline{a I a}$ such that $a_1 \sim a_2 \dots \sim a_n$ and $(a - \epsilon)_+ \precsim a_j \oplus c$ for $j \in \{1, 2, \dots, n \}$. 
\end{lemma}

The proof of Lemma \ref{lemma: KR Lemma} is done in exactly the same way as the proof of \cite[Lemma 4.14]{KR00}; note that \cite[Lemma 4.14]{KR00} is a special case of Lemma \ref{lemma: KR Lemma} for $n = 3$. Finally, \cite[Lemma 3.17 (ii)]{KR00} says that the Cuntz class of any properly infinite projection is automatically absorbed by any larger Cuntz class; this will be used in the proof of Theorem \ref{introtheorem: Cu E Iso}.

\begin{lemma}[{\cite[Lemma 3.17 (ii)]{KR00}}] \label{lemma: Polar Part Absorbed}
Let $A$ be a C*-algebra, $a \in A_+$ and $p$ be a projection in $A$ such that $p \precsim a$. If $p$ is properly infinite then $a \oplus p \precsim a$.
\end{lemma}

\section{Stable rank one and weak cancellation}
\label{sec: Weak Cancellation}

In a C*-algebra $A$ with stable rank one, projections can be cancelled from Cuntz subequivalences in the following sense: if $a, b \in (A \otimes K)_+$, $p$ is a projection in $A \otimes K$ and $a \oplus p \precsim b \oplus p$ then $a \precsim b$. This fact was first proved in \cite[Proposition 4.2]{RW10}. More generally by \cite[Theorem 4.3]{RW10}, the Cuntz semigroup of a stable rank one C*-algebra is weakly cancellative.

\begin{definition}
Let $S$ be a Cu-semigroup. $S$ is said to be \textbf{weakly cancellative} if for $a, b, x \in S$ satisfying $a + x \ll b + x$, $a \ll b$. 
\end{definition}

See \cite[Lemma 2.5]{APRT22} for three equivalent formulations of a weakly cancellative Cu-semigroup. In particular, we will use the following equivalent characterisation of weak cancellation --- if $a, b, x \in S$ and $a + x \ll b + x$ then $a \leq b$. \newline
\phantom{} \newline
In general, graph C*-algebras usually do not have stable rank one (see \cite[Theorem 3.4]{DHS03}) and so most of the time, their Cuntz semigroups are not weakly cancellative. However, it turns out that weak cancellation is a local property --- that is, if the associated corner to a projection has stable rank one, then it can still be cancelled from Cuntz subequivalences. 

\begin{theorem} \label{theorem: SR1 Corner Positive Cancellation}
Let $A$ be a unital C*-algebra and $p \in A$ be a projection such that $pAp$ has stable rank one. Let $x, y \in \Cu(A)$. If $x + [p] \leq y + [p]$ then $x \leq y$. 
\end{theorem}

A similar phenomenon can be observed for Murray-von Neumann subequivalences in the proof of \cite[Theorem 7.1]{AMP07} which follows from the cancellation of modules result in \cite[Theorem 2.1]{War80}. In the proof of Theorem \ref{introtheorem: Cu E Iso}, Theorem \ref{theorem: SR1 Corner Positive Cancellation} will be used to establish surjectivity of the map $\Delta_E$ which we need to prove is a Cu-isomorphism. Specifically, it handles a case where the graph C*-algebra has a gauge-simple, stable rank one ideal. \newline
\phantom{} \newline
The two main results we will prove in this section are Theorem \ref{theorem: SR1 Corner Positive Cancellation} and Theorem \ref{theorem: General Weak Cancellation}. Theorem \ref{theorem: General Weak Cancellation} generalises Theorem \ref{theorem: SR1 Corner Positive Cancellation} in the same way \cite[Theorem 4.3]{RW10} generalises \cite[Proposition 4.2]{RW10}, allowing for the cancellation of the Cuntz class of a positive element if the hereditary subalgebra it generates has stable rank one. \newline
\phantom{} \newline
The proof of Theorem \ref{theorem: SR1 Corner Positive Cancellation} is inspired by \cite[Theorem 4.8]{RS21} and \cite[Theorem 4.9]{RS21} and proceeds from the countably generated right Hilbert module picture of the Cuntz semigroup of a C*-algebra; see \cite{CEI08} and \cite[Chapter 4]{APT11}. Recall that this picture of the Cuntz semigroup is equivalent to the positive element picture of the Cuntz semigroup used throughout this paper; see for instance, \cite[Theorem 4.33]{APT11}. If $M$ and $N$ are countably generated right Hilbert $A$-modules, then $M$ is Cuntz subequivalent to $N$ if any compactly contained submodule of $M$ (see \cite[Page 162]{CEI08} for the definition) is isometrically isomorphic to a compactly contained submodule of $N$. In particular, if $M$ is a Hilbert submodule of $N$, then $M$ is Cuntz below $N$; see \cite[Lemma 4.8]{APT11}. \newline
\phantom{} \newline
An important ingredient of the proof is that if $A$ is a unital C*-algebra and $N$ is an arbitrary (right) Hilbert $A$-module, then any Hilbert $A$-submodule $M \subseteq N$ which is a direct summand of $A^n$ for some $n \in \mathbb{N}$ sits inside $N$ in the nicest possible manner. 

\begin{lemma}[{\cite[Proposition 1]{DF81}}] \label{lemma: Dupre Fillmore}
Let $A$ be a unital C*-algebra, $N$ be a (right) Hilbert $A$-module and $M$ be a Hilbert $A$-submodule of $N$. If there exists $n \in \mathbb{N}$ such that $M$ is a direct summand of $A^n$ then $N = M \oplus M^{\bot}$. 
\end{lemma}

\begin{proof}[Proof of Theorem \ref{theorem: SR1 Corner Positive Cancellation}]
First, we prove that if $M$ and $N$ are countably generated (right) Hilbert $A$-modules and $M \oplus pA \cong N \oplus pA$ then $M \cong N$ as Hilbert $A$-modules. Let $\mathcal{B}(M, N)$ denote the set of bounded, adjointable operators from $M$ to $N$; if $M = N$ then we write $\mathcal{B}(M)$ instead of $\mathcal{B}(M, M)$. Let $V: pA \oplus M \rightarrow pA \oplus N$ be an isomorphism of (right) Hilbert $A$-modules. Write $V$ in matrix form so that 
\begin{equation*}
V = 
\begin{pmatrix}
V_{11} & V_{12} \\
V_{21} & V_{22}
\end{pmatrix}
\in 
\begin{pmatrix}
\mathcal{B}(pA) & \mathcal{B}(M, pA) \\
\mathcal{B}(pA, N) & \mathcal{B}(M, N)
\end{pmatrix}.
\end{equation*}
Observe that $\mathcal{B}(pA) \cong pAp$ which has stable rank one by assumption. Hence, by replacing $V_{11}$ with an invertible operator $R \in \mathcal{B}(pA)$ satisfying $\lVert V_{11} - R \rVert < 1$, we may assume without loss of generality that $V$ is an invertible adjointable operator such that $V_{11}$ is itself invertible. By direct computation,
{ \small
\begin{equation} \label{equation: Quasidet Decomp}
\begin{pmatrix}
\text{id}_{pA} & 0 \\
- V_{21} V_{11}^{-1} & \text{id}_N
\end{pmatrix}
V
\begin{pmatrix}
\text{id}_{pA} & - V_{11}^{-1} V_{12} \\
0 & \text{id}_M
\end{pmatrix} = 
\begin{pmatrix}
V_{11} & 0 \\
0 & V_{22} - V_{21} V_{11}^{-1} V_{12}
\end{pmatrix}.
\end{equation}}
Since the LHS of \eqref{equation: Quasidet Decomp} is invertible and adjointable, the RHS of \eqref{equation: Quasidet Decomp} is also an invertible adjointable operator from $pA \oplus M$ to $pA \oplus N$. Consequently, $V_{22} - V_{21} V_{11}^{-1} V_{12}$ is an invertible adjointable operator from $M$ to $N$ and by \cite[Proposition 3.8]{Lan95}, $M \cong N$ as Hilbert $A$-modules. Combining this result with Lemma \ref{lemma: Dupre Fillmore}, we see that if $N$ is a countably generated Hilbert $A$-module, $N_1$ and $N_2$ are Hilbert $A$-submodules of $N$ both isomorphic to $pA$ and $pAp$ has stable rank one then $N_1^{\bot} \cong N_2^{\bot}$. \newline
\phantom{} \newline
Now let $x = [M]$ and $y = [N]$ where $M$ and $N$ are countably generated (right) Hilbert $A$-modules. Let $E$ be a compactly contained submodule of $M$. Since $pA$ is compactly contained in itself then $E \oplus pA$ is a compactly contained submodule of $M \oplus pA$. Using the assumption that $[M] + [pA] \leq [N] + [pA]$, there exists an embedding $\phi: E \oplus pA \rightarrow N \oplus pA$. Observe that $\phi(\{ 0 \} \oplus pA)$ is a Hilbert submodule of $N \oplus pA$ which is isomorphic to $pA$. So by the previous result, $\phi(\{ 0 \} \oplus pA)^{\bot} \cong N$ as Hilbert $A$-modules. Since $\phi(E \oplus \{ 0 \}) \subseteq \phi(\{ 0 \} \oplus pA)^{\bot}$ then by \cite[Lemma 4.8]{APT11}, $[E] \leq [N]$. By passing to the supremum of all compactly contained submodules of $M$, $[M] \leq [N]$ as required. 
\end{proof}

\begin{remark} \label{remark: Non Unital SR1 Corner}
It is easy to see that Theorem \ref{theorem: SR1 Corner Positive Cancellation} generalises to the non-unital case. If $A$ is non-unital, then apply Theorem \ref{theorem: SR1 Corner Positive Cancellation} to the minimal unitisation $\tilde{A}$, and use the fact that $\Cu(A)$ is an ideal of $\Cu(\tilde{A})$. 
\end{remark}

We now turn to the statement and proof of Theorem \ref{theorem: General Weak Cancellation}. In order to upgrade Theorem \ref{theorem: SR1 Corner Positive Cancellation} to Theorem \ref{theorem: General Weak Cancellation}, we need to replace the projection in Theorem \ref{theorem: General Weak Cancellation} with a positive element. By a routine application of axiom (O5) for Cu-semigroups (see \cite[Definition 4.1]{APT18}), we are able to cancel Cuntz classes in the Cuntz semigroup of a C*-algebra if they are below the Cuntz class of a projection whose associated corner has stable rank one. 

\begin{lemma} \label{lemma: Bounded Projection Cancellation}
Let $A$ be a C*-algebra, $p \in A$ be a projection such that $pAp$ has stable rank one and $x, y, z \in \Cu(A)$. If $x + z \ll y + z$ and $z \leq [p]$ then $x \leq y$. 
\end{lemma}

\begin{proof}
By writing $z$ as the supremum of a rapidly increasing sequence, there exists $z' \in \Cu(A)$ such that $z' \ll z$ and $x + z \leq y + z'$. Recall that the Cuntz semigroup of any C*-algebra satisfies (O5); see \cite[Proposition 4.6]{APT18} for the proof. By applying (O5) to the inequality $z' \ll z \leq [p]$, there exists $c \in \Cu(A)$ such that $z' + c \leq [p] \leq z + c$. We compute directly that $$x + [p] \leq (x + z) + c \leq y + (z' + c) \leq y + [p].$$ By Remark \ref{remark: Non Unital SR1 Corner}, $x \leq y$. 
\end{proof}

The key to removing the projection entirely in Lemma \ref{lemma: Bounded Projection Cancellation} is \cite[Proposition 2.8]{APRT22}. Given a stable C*-algebra $A$ and a positive element $a \in A$, the proof of \cite[Proposition 2.8]{APRT22} constructs a C*-algebra $B$ and a projection $p_a \in B$ satisfying specific properties. Below, we state the properties of the construction in \cite[Proposition 2.8]{APRT22} we need for the proof of Theorem \ref{theorem: General Weak Cancellation}. 

\begin{proposition}[{\cite[Proposition 2.8]{APRT22}}] \label{proposition: SR1 Envelope}
Let $A$ be a stable C*-algebra and $a \in A_+$. Then there exist a C*-algebra $B$ and a projection $p_a \in B$ such that 
\begin{enumerate}
\item $A$ is a closed, two-sided ideal of $B$, 
\item the quotient $B/A$ is isomorphic to $\mathbb{C}$ and its unit is the image of $p_a$ in $B/A$,
\item if $x \in \Cu(A)$ then $x \leq [a]$ in $\Cu(A)$ if and only if $x \leq [p_a]$ in $\Cu(B)$,
\item as right Hilbert $A$-modules, $p_a A \cong \overline{aA}$. 
\end{enumerate}
\end{proposition}

We are now ready to state and prove Theorem \ref{theorem: General Weak Cancellation}.

\begin{theorem} \label{theorem: General Weak Cancellation}
Let $A$ be a C*-algebra and $x, y \in \Cu(A)$. Let $a \in (A \otimes K)_+$ be such that $\overline{a (A \otimes K) a}$ has stable rank one. If $x + [a] \ll y + [a]$ then $x \leq y$. 
\end{theorem}

\begin{proof}
By Proposition \ref{proposition: SR1 Envelope}, there exist a C*-algebra $B$ and a projection $p_a \in B$ satisfying the listed properties in Proposition \ref{proposition: SR1 Envelope}. In particular, $B$ fits into the extension
\begin{equation} \label{equation: Envelope Extension}
\begin{tikzcd}
0 \arrow[r] & A \otimes K \arrow[r] & B \arrow[r, "\pi"] & \mathbb{C} \arrow[r] & 0.
\end{tikzcd}
\end{equation}
Since $\overline{a (A \otimes K)} \cong p_a (A \otimes K)$ as right Hilbert $(A \otimes K)$-modules then $\overline{(A \otimes K) a}$ is isomorphic to $(A \otimes K) p_a$ and $$p_a (A \otimes K) p_a = p_a (A \otimes K) \cap (A \otimes K) p_a \cong \overline{a (A \otimes K)} \cap \overline{(A \otimes K) a} = \overline{a (A \otimes K) a}.$$ So by assumption, the hereditary subalgebra $p_a (A \otimes K) p_a$ has stable rank one. Now compress the extension in \eqref{equation: Envelope Extension} by $p_a$ to produce the following extension:
\begin{equation} \label{equation: Compressed Extension}
\begin{tikzcd}
0 \arrow[r] & p_a (A \otimes K) p_a \arrow[r] & p_a B p_a \arrow[r, "\pi"] & \mathbb{C} \arrow[r] & 0.
\end{tikzcd}
\end{equation}
Note that $\pi(p_a) \mathbb{C} \pi(p_a) = \mathbb{C}$ because the unit of the quotient $B/(A \otimes K) \cong \mathbb{C}$ is $\pi(p_a)$. Since the ideal and quotient of the extension in \eqref{equation: Compressed Extension} both have stable rank one then by \cite[Proposition 4]{LR95}, $p_a B p_a$ must also have stable rank one. \newline
\phantom{} \newline
Now since $\Cu(A)$ is an ideal of $\Cu(B)$ then $x + [a] \ll y + [a]$ in $\Cu(B)$. By Proposition \ref{proposition: SR1 Envelope}, $[a] \leq [p_a]$ in $\Cu(B)$. Hence by Lemma \ref{lemma: Bounded Projection Cancellation}, $x \leq y$ in $\Cu(B)$. Since $x, y \in \Cu(A)$, we conclude that $x \leq y$ in $\Cu(A)$. This completes the proof. 
\end{proof}

Observe that in the statement of Theorem \ref{theorem: General Weak Cancellation}, if $A$ has stable rank one, then $A \otimes K$ also has stable rank one by \cite[Theorem 3.6]{Rie83}. By combining \cite[Corollary 2.6]{Bro77} with \cite[Theorem 4.4]{Rie83}, any hereditary subalgebra of $A \otimes K$ also has stable rank one. Subsequently, we recover weak cancellation of $\Cu(A)$. 

\section{Two applications of moves for graph C*-algebras}
\label{sec: Refined Moves}

This section details two separate and important applications of the moves (O), (I-), (S) and (CO) defined in \cref{subsec: Moves}. \Cref{subsec: Standing Assumptions} demonstrates that in the overall computation of the Cuntz semigroup of a graph C*-algebra coming from a finite graph, we may assume without loss of generality that the underlying graph satisfies standing assumptions \ref{property: S1}, \ref{property: S2} and \ref{property: S3}. The standing assumptions render Hong and Szymański's analysis of the entire ideal lattice of a graph C*-algebra in \cite{HS04} more tractable than in the general case; see Corollary \ref{corollary: Auto Saturated} and Corollary \ref{corollary: Maximal Tails Standing Assumptions} for example. This is important for establishing a bijection from the lattice of ideals of $C^*(E)$ to to the lattice of ideals of the abstract Cu-semigroup $\Cu_E$ defined in Definition \ref{definition: Cu E}; see \cref{subsec: Cu E Ideal Lattice}. \newline
\phantom{} \newline
The second application of the moves in \cref{subsec: QDS} will be invoked in \cref{subsec: Circle Quotient} to establish the pure largeness needed to prove the difficult special case of Theorem \ref{introtheorem: General Cuntz Krieger} where the vertex $v$ supports a single loop. 

\subsection{The standing assumptions on a finite graph}
\label{subsec: Standing Assumptions}

Recall from the introduction the standing assumptions for a finite graph, which we list below once more for convenience:
\begin{enumerate}
\item[\crtcrossreflabel{(S1)}] $E$ has no regular sinks.
\item[\crtcrossreflabel{(S2)}] $E$ has no regular transitional vertices.
\item[\crtcrossreflabel{(S3)}] if a vertex $v$ supports a single cycle $\mu$ then $\mu$ is a loop (has length one).
\end{enumerate}
By applying moves (CO) and (S) (see Definition \ref{definition: CO} and Definition \ref{definition: S} respectively) to an arbitrary finite graph, we now prove the following proposition.

\begin{proposition} \label{proposition: Collapse Graph}
Let $E$ be a finite graph. Then there exists a finite graph $F$ satisfying the standing assumptions \ref{property: S1}, \ref{property: S2} and \ref{property: S3} such that $C^*(E) \otimes K \cong C^*(F) \otimes K$. In particular, the Cuntz semigroup $\Cu(C^*(E))$ is isomorphic to $\Cu(C^*(F))$. 
\end{proposition}

\begin{proof}
Assume that $E$ is a finite graph. First, repeatedly apply move (S) to $E$ until $E$ has no regular sinks. Next, suppose that $v \in E^0$ supports a single cycle $\mu$ with length strictly greater than one. We repeatedly apply move (CO) to $\mu$ until it becomes a loop. Finally, we remove all regular transitional vertices from $E$ by applying move (CO) to them. In the end, we obtain a finite graph $F$ satisfying the standing assumptions. By Proposition \ref{proposition: Refined Moves Stable Iso}, $C^*(E)$ and $C^*(F)$ are stably isomorphic. 
\end{proof}

\begin{example}
Suppose that $E$ is the following graph:
\begin{equation*}
\begin{tikzpicture}[baseline=(current  bounding  box.center)] 
\node (1) at (-2, 0) {$a$};
\node (2) at (0, 0) {$b$};
\node (3) at (2, 0) {$c$};
\node (4) at (4, 0) {$d$};
\node (5) at (6, 0) {$e$};
\node (6) at (-4, 0) {$f$};
\node (7) at (0, 2) {$g$};
\draw[->] (2) to (1);
\draw[->] (2) to (3);
\draw[->] (4) to (3);
\draw[->] (5) to (4);
\draw[->] (1) to [bend right = 20] (6);
\draw[->] (6) to [bend right = 20] (1);
\draw[->] (2) to [bend right = 30] (7);
\draw[->] (7) to [bend right = 30] (2);
\draw[->] (1) to [out = 45, in = 135, looseness = 10] (1);
\end{tikzpicture}
\end{equation*}
By applying the procedure outlined in the proof of Proposition \ref{proposition: Collapse Graph}, we obtain the following graph $F$: 
\begin{equation*}
\begin{tikzpicture}[baseline=(current  bounding  box.center)] 
\node (1) at (-2, 0) {$a$};
\node (2) at (0, 0) {$b$};
\node (5) at (2, 0) {$e$};
\node (6) at (-4, 0) {$f$};
\draw[->] (2) to (1);
\draw[->] (1) to [bend right = 20] (6);
\draw[->] (6) to [bend right = 20] (1);
\draw[->] (1) to [out = 45, in = 135, looseness = 10] (1);
\draw[->] (2) to [out = 45, in = 135, looseness = 10] (2);
\end{tikzpicture}
\end{equation*}
It is easy to see that $F$ satisfies the standing assumptions. Note that vertex $e$ does not emit or receive any edges; such a vertex is nevertheless allowed by standing assumption \ref{property: S1} because it is singular. By Proposition \ref{proposition: Collapse Graph}, $C^*(E)$ and $C^*(F)$ are stably isomorphic. 
\end{example}

With the standing assumptions \ref{property: S1}, \ref{property: S2} and \ref{property: S3}, the saturated condition appearing in \eqref{equation: Hereditary Saturated} becomes superfluous.

\begin{corollary} \label{corollary: Auto Saturated}
Let $E$ be a finite graph satisfying the standing assumptions and $H \subseteq E^0$. If $H$ is hereditary then it is automatically saturated.
\end{corollary}

\begin{proof}
Assume that $E$ is a finite graph satisfying the standing assumptions and $H$ is a hereditary subset of $E^0$. To see that $H$ is saturated, suppose that $v \in E^0$ is such that $0 < \lvert r^{-1}(v) \rvert < \infty$ and $$\{ s(e) \; \vert \; e \in E^1, \; r(e) = v \} \subseteq H.$$ By assumption, $v$ is not a source vertex. Since $E$ satisfies standing assumptions \ref{property: S1} and \ref{property: S2}, $v$ supports at least one cycle. So there exists an edge $f$ such that $r(f) = v$ and $v \geq s(f)$. Now since $s(f) \in H$ and $v \geq s(f)$ then $v \in H$ and $H$ is saturated. 
\end{proof}

Recall from the comment after Definition \ref{definition: Maximal Tail Tau} that the set of primitive non-gauge-invariant ideals of a unital graph C*-algebra $C^*(E)$ is in bijective correspondence with the set $\mathcal{MT}_{\tau}(E) \times \mathbb{T}$. We will now show that in the presence of the standing assumptions, the elements of $\mathcal{MT}_\tau(E)$ correspond to vertices which support a single loop. First, if $E$ is a graph and $v \in E^0$, then define
\begin{equation*} 
R(v) = \{ x \in E^0 \; \vert \; v \geq x \}. 
\end{equation*}

\begin{corollary} \label{corollary: Maximal Tails Standing Assumptions}
Let $E$ be a finite graph satisfying the standing assumptions. Then 
\begin{equation*}
\mathcal{MT}_\tau(E) = \{ R(v) \; \vert \; \text{$v \in E^0$ supports a single loop} \}. 
\end{equation*}
\end{corollary}

\begin{proof}
First, assume that $v \in E^0$ supports a single loop. The set $R(v)$ satisfies properties \ref{property: M1} and \ref{property: M3} by definition. Now suppose that $w \in R(v)$ and $0 < \lvert r^{-1}(w) \rvert < \infty$. By the standing assumptions on $E$, $w$ supports at least one cycle and consequently, $R(v)$ satisfies property \ref{property: M2} and is a maximal tail. Moreover, the single loop based at $v$ cannot have an entry lying in $R(v)$ by definition of $R(v)$. So $R(v) \in \mathcal{MT}_\tau(E)$. \newline
\phantom{} \newline
Conversely, assume that $M \in \mathcal{MT}_\tau(E)$. By \cite[Lemma 2.1]{HS04}, there exist $v \in M$ and a cycle $\mu = \mu_m \dots \mu_2 \mu_1$ based at $v$ which has no entries in $M$; that is, if $e \in E^1$ does not lie on $\mu$ and $r(e)$ lies on $\mu$ then $s(e) \not \in M$. By standing assumption \ref{property: S3}, $m = 1$ and since $\mu$ has no entries in $M$ then it must be the unique cycle based at $v$. Using \cite[Lemma 2.1]{HS04} once more, we have
\begin{equation*}
E^0 \backslash M = \big\{ x \in E^0 \backslash \{ v \} \; \vert \; v \not \geq x \big\} = E^0 \backslash R(v). 
\end{equation*}
So $M = R(v)$. 
\end{proof}

Corollary \ref{corollary: Maximal Tails Standing Assumptions} says that with the standing assumptions, a primitive non-gauge-invariant ideal is simply given by a vertex supporting a single loop together with an element of $\mathbb{T}$. This motivates the following definition. 

\begin{definition} \label{definition: Primitive Non Gauge Invariant}
Let $E$ be a finite graph satisfying the standing assumptions and $v \in E^0$ be a vertex supporting a single loop $g$. Let $t \in \mathbb{T}$. Define $\mathcal{R}_{v, t}$ to be the primitive non-gauge-invariant ideal generated by the gauge-invariant ideal $I_{E^0 \backslash R(v)}$ and $s_g - t p_v$. \newline
\phantom{} \newline
More generally, let $K \subseteq \mathbb{T}$ be closed. Define $\mathcal{R}_{v, K}$ to be the non-gauge-invariant ideal $$\mathcal{R}_{v, K} = {\bigcap}_{t \in K} \mathcal{R}_{v, t}.$$
\end{definition}

In the situation of Definition \ref{definition: Primitive Non Gauge Invariant}, the bijective correspondence in \cite[Corollary 2.12]{HS04} maps the pair $(R(v), t) \in \mathcal{MT}_{\tau}(E) \times \mathbb{T}$ to the primitive ideal $\mathcal{R}_{v, t}$. Moreover, we have by \cite[Lemma 2.6]{HS04} the inclusion of ideals
\begin{equation} \label{equation: Sandwich}
I_{E^0 \backslash R(v)} \subseteq \mathcal{R}_{v, t} \subseteq I_{(E^0 \backslash R(v)) \sqcup \{ v \}}.
\end{equation}
In particular, $I_{E^0 \backslash R(v)}$ is the largest gauge-invariant ideal contained by $\mathcal{R}_{v, t}$ and $I_{(E^0 \backslash R(v)) \sqcup \{ v \}}$ is the smallest gauge-invariant ideal containing $\mathcal{R}_{v, t}$. \newline
\phantom{} \newline
Later on, we will need to consider strictly positive elements of appropriate corners of non-gauge-invariant ideals, such as the ones from Definition \ref{definition: Primitive Non Gauge Invariant}. In Definition \ref{definition: Delta E}, we construct a map $\Delta_E$ from the Cu-semigroup $\Cu_E$, defined in \cref{subsec: Cu E}, to $\Cu(C^*(E))$. To ensure that $\Delta_E$ is order preserving, we need to choose strictly positive elements of corners of non-gauge-invariant ideals --- rather than the ideals themselves --- so that they are Cuntz subequivalent to the appropriate sums of vertex projections. Example \ref{example: Why Corner} below illustrates this difference. 

\begin{example} \label{example: Why Corner}
Let $E$ be the following graph which satisfies the standing assumptions:
\begin{equation*}
\begin{tikzpicture}[baseline=(current  bounding  box.center)] 
\node (1) at (0, 0) {$v$};
\node (2) at (3, 0) {$w$};
\draw[->] (1) to (2);
\draw[->] (1) to [out = 45, in = 135, looseness = 10] (1);
\draw[->] (2) to [out = 45, in = 135, looseness = 10] (2);
\end{tikzpicture}
\end{equation*}
Let $U \subseteq \mathbb{T}$ be non-empty and open. The gauge-invariant ideal $I_{\{ v \}}$ is isomorphic to $C(\mathbb{T}) \otimes K$. Hence, the ideal $\mathcal{R}_{v, \mathbb{T} \backslash U}$ is also stable and the Cuntz class of a strictly positive element in this ideal is the function $\infty \cdot \chi_U$ in $\Lsc(\mathbb{T}, \overline{\mathbb{N}})$. This is not Cuntz below $[p_v] = \chi_{\mathbb{T}}$. \newline
\phantom{} \newline
On the other hand, the corner $p_v \mathcal{R}_{v, \mathbb{T} \backslash U} p_v$ is an ideal of $p_v I_{\{ v \}} p_v \cong C(\mathbb{T})$ by \eqref{equation: Full Corner Iso}. Under this isomorphism, $p_v \mathcal{R}_{v, \mathbb{T} \backslash U} p_v$ is sent to the ideal $C_0(U)$ and the Cuntz class of a strictly positive element in this corner is the function $\chi_U$, which is Cuntz below $[p_v] = \chi_{\mathbb{T}}$ in $\Lsc(\mathbb{T}, \overline{\mathbb{N}})$. 
\end{example}

Before we introduce our notation for strictly positive elements, we first define for an arbitrary graph $E$ and a vertex $v \in E^0$ the set
\begin{equation*}
L(v) = \{ x \in E^0 \; \vert \; x \geq v \}
\end{equation*}
which is the smallest hereditary subset of $E^0$ containing $v$.

\begin{definition} \label{definition: Strictly Positive}
Let $E$ be a finite graph satisfying the standing assumptions and $v \in E^0$ support a single loop. Let $U \subseteq \mathbb{T}$ be non-empty and open. If $U$ is proper then define $h_{v, U}$ to be a strictly positive element of the hereditary subalgebra $$p_{L(v)} \big( I_{L(v)} \cap \mathcal{R}_{v, \mathbb{T} \backslash U} \big) p_{L(v)}.$$ Otherwise, define $h_{v, \mathbb{T}} = p_{L(v)}$. 
\end{definition}

By Definition \ref{definition: Strictly Positive}, $h_{v, \mathbb{T}}$ is a strictly positive element of the hereditary subalgebra $p_{L(v)} I_{L(v)} p_{L(v)}$, which is isomorphic to $C^*(L(v))$ by \eqref{equation: Full Corner Iso}. Subsequently, if $U, V \subseteq \mathbb{T}$ are non-empty and open and $U \subseteq V$, then $h_{v, U} \precsim h_{v, V} \precsim p_{L(v)}$. \newline
\phantom{} \newline
Finally by \eqref{equation: Hereditary Saturated}, Corollary \ref{corollary: Auto Saturated} and Corollary \ref{corollary: Maximal Tails Standing Assumptions}, any ideal $J$ of $C^*(E)$ where $E$ is a finite graph satisfying the standing assumptions can be written as
\begin{equation} \label{equation: Ideal Form}
J = I_H \cap {\bigcap}_{i = 1}^n \mathcal{R}_{v_i, K_i}
\end{equation}
where $H$ is a hereditary subset of $E^0$ and for $i \in \{1, 2, \dots, n \}$, $v_i \in E^0$ supports a single loop and $K_i$ is a closed subset of $\mathbb{T}$. If $J$ is a non-gauge-invariant ideal, then we may assume $J$ satisfies the following properties:
\begin{enumerate}
\item[{\crtcrossreflabel{(I1)}[property: I1]}] The vertices $v_1, \dots, v_n \in E^0$ are all distinct.
\item[{\crtcrossreflabel{(I2)}[property: I2]}] $K_1, \dots, K_n$ are all proper, non-empty, closed subsets of $\mathbb{T}$. 
\item[{\crtcrossreflabel{(I3)}[property: I3]}] If $j \in \{1, 2, \dots, n \}$ then $v_j \in H \cap \bigcap_{i = 1, \; i \neq j}^n (E^0 \backslash R(v_i))$.
\end{enumerate}
Let us justify property \ref{property: I3} further. Suppose that there exists $j \in \{1, 2, \dots, n \}$ such that $v_j \not \in H \cap \bigcap_{i = 1, \; i \neq j}^n (E^0 \backslash R(v_i))$. Note that $v_j \not \in E^0 \backslash R(v_j)$. By \eqref{equation: Sandwich}, 
\begin{IEEEeqnarray*}{rCl}
J & \subseteq & I_H \cap {\bigcap}_{i = 1, \; i \neq j}^n \mathcal{R}_{v_i, K_i} \cap I_{(E^0 \backslash R(v_j)) \sqcup \{ v_j \}}
\\
& = & I_H \cap {\bigcap}_{i = 1, \; i \neq j}^n \big( I_{E^0 \backslash R(v_i)} \cap \mathcal{R}_{v_i, K_i} \big) \cap I_{(E^0 \backslash R(v_j)) \sqcup \{ v_j \}} 
\\
& = & I_H \cap {\bigcap}_{i = 1, \; i \neq j}^n \mathcal{R}_{v_i, K_i} \cap I_{E^0 \backslash R(v_j)}
\\
& \subseteq & I_H \cap {\bigcap}_{i = 1}^n \mathcal{R}_{v_i, K_i} = J. 
\end{IEEEeqnarray*}
Therefore, 
\begin{equation*}
J = I_{H \cap \; (E^0 \backslash R(v_j))} \cap {\bigcap}_{i = 1, \; i \neq j}^n \mathcal{R}_{v_i, K_i} 
\end{equation*}
which shows that we can assume without loss of generality that $J$ satisfies property \ref{property: I3}. Assuming properties \ref{property: I1}, \ref{property: I2} and \ref{property: I3} will significantly simplify our treatment of non-gauge-invariant ideals in \cref{subsec: Cu E Ideal Lattice}. 

\subsection{Quantum double suspension graphs}
\label{subsec: QDS}

Let $E$ be a finite graph and suppose that there exists a vertex $v \in E^0$ supporting a single loop such that if $w \in E^0$ then $w \geq v$. Then the set 
\begin{equation*}
H = \{ x \in E^0 \; \vert \; v \not \geq x \} = E^0 \backslash \{ v \}
\end{equation*}
is hereditary and saturated. By \cite[Lemma 1.6]{DHS03}, there exists a row-finite graph $E_H$ such that $I_H \cong C^*(E_H)$. Applying \cite[Theorem 3.2]{Tom04} to $E_H$, we find that $I_H$ is stable. Thus by \cite[Theorem 4.1]{BPRS00}, $C^*(E)$ fits into the extension
\begin{equation*}
\begin{tikzcd}
0 \arrow[r] & C^*(H) \otimes K \arrow[r] & C^*(E) \arrow[r] & C(\mathbb{T}) \arrow[r] & 0. 
\end{tikzcd}
\end{equation*} 
In this subsection, we will show with the above assumption on $E$ that $C^*(E)$ is stably isomorphic to the graph C*-algebra of a \emph{quantum double suspension graph}, first defined in \cite[Section 6]{HS02}. Informally, the quantum double suspension of a finite graph $F$ is a finite graph $S^2 F$ formed by connecting every vertex in $F$ to a new vertex supporting a single loop with a single edge. The point is that in order to prove Theorem \ref{introtheorem: General Cuntz Krieger} in the case where the distinguished vertex $v$ supports a single loop, it suffices by invariance of the Cuntz semigroup under stabilisation to prove it for graph C*-algebras of quantum double suspension graphs, which is made accessible by the concrete realisations of the associated Busby invariant from $C(\mathbb{T})$ to the corona algebra $\mathcal{Q}(C^*(F) \otimes K)$ and the isomorphism $I_F \cong C^*(F) \otimes K$ from \cite[Proposition 6.2]{HS02}.

\begin{definition}[{\cite[Page 185]{HS02}}] \label{definition: QDS}
Let $E$ be a finite graph. The \textbf{quantum double suspension} of $E$, denoted by $S^2 E$, is the finite graph given by connecting all vertices $w \in E^0$ with a single edge $f_w$ to a new vertex $v$ supporting a single loop $g$ so that the vertices and edges of $S^2 E$ are
\begin{equation*}
(S^2 E)^0 = E^0 \cup \{ v \} \qquad \text{and} \qquad (S^2 E)^1 = E^1 \cup \{ f_w \; \vert \; w \in E^0 \} \cup \{ g \}
\end{equation*}
respectively. Furthermore, the source and range maps of $S^2 E$ extend those of $E$ and satisfy
\begin{equation*}
s_{S^2 E}(e) = 
\begin{cases}
w, \text{ if $w \in E^0$ and $e = f_w$,} \\
v, \text{ if $e = g$,}
\end{cases}
\end{equation*}
and
\begin{equation*}
r_{S^2 E}(e) = 
\begin{cases}
v, \text{ if $w \in E^0$ and $e = f_w$,} \\
v, \text{ if $e = g$}
\end{cases}
\end{equation*}
respectively.
\end{definition}

\begin{proposition} \label{proposition: QDS Forcing}
Let $E$ be a finite graph and suppose that there exists $v \in E^0$ such that $v$ supports a single loop and if $w \in E^0$ then $w \geq v$. Let $H = E^0 \backslash \{ v \}$. Then there exists a finite graph $F$ such that $C^*(E) \otimes K \cong C^*(S^2 F) \otimes K$ and $C^*(H) \otimes K \cong C^*(F) \otimes K$.
\end{proposition}

\begin{proof}
We first modify $E$ so that every vertex in $E^0 \backslash \{ v \}$ emits at least one edge to $v$. \newline
\phantom{} \newline
\textbf{Step 1: Introducing edges to $v$.} \newline
\phantom{} \newline
Assume that $w \in E^0 \backslash \{ v \}$ emits edges to $v$. If $w$ does not receive any edges, then $w$ is a source vertex and there is nothing to do. If $w$ receives edges then there are two cases to consider. First suppose that $w$ does not support a loop. By applying move (CO) to the vertex $w$, we produce in the resulting graph edges whose sources lie in the set $\{ s(e) \; \vert \; e \in E^1, \; r(e) = w \}$ and whose ranges are $v$. \newline
\phantom{} \newline
Now if $w$ supports at least one loop, then let $L_w$ be the set of all loops based at $w$. We apply move (I-) to the vertex $w$ with the following partition of $s^{-1}(w)$:
\begin{equation} \label{equation: QDS I- Partition}
s^{-1}(w) = L_w \sqcup \big(s^{-1}(w) \backslash L_w \big).
\end{equation}
The vertex $w$ is replaced with two vertices $w_1$ and $w_2$. There are $k$ loops based at $w_1$ and $k$ edges from $w_1$ to $w_2$. On the other hand, there are no loops based at $w_2$ because there are no loops in the set $s^{-1}(w) \backslash L_w$. Moreover, the move (I-) introduces edges whose sources lie in the set
\begin{equation} \label{equation: I- Set}
\{ s(e) \; \vert \; e \in E^1, \; s(e) \neq w, \; r(e) = w \}
\end{equation}
and whose ranges are either $w_1$ or $w_2$.  \newline
\phantom{} \newline
Now since $w_2$ does not support any loops, we may apply (CO) to it. The edges whose sources lie in the set given in \eqref{equation: I- Set} and whose ranges are $w_2$ induce a new set of edges whose sources lie in the same set and whose ranges are now $v$. Similarly, the $k$ edges from $w_1$ to $w_2$ induce new edges from $w_1$ to $v$. Hence, we obtain a graph such that the vertices in the set $$\{ s(e) \; \vert \; e \in E^1, \; s(e) \neq w, \; r(e) = w \} \cup \{ w_1 \}$$ all emit edges to $v$. \newline
\phantom{} \newline
\textbf{Step 2:} We repeatedly apply step 1 to all vertices $w \in E^0 \backslash \{ v \}$ which emit edges to $v$, but receive edges from vertices which themselves do not emit an edge to $v$. The result is a graph where every vertex which emitted an edge to $w$ but previously did not emit an edge to $v$, now emits edges to $v$. The vertex $w$ is either deleted from the graph or replaced by a vertex which still emits edges to $v$. With each application of step 1, the number of vertices which do not emit edges to $v$ decreases. Since $E$ has finitely many vertices then this process must eventually terminate. We obtain a graph $F'$ where every vertex $w \in (F')^0 \backslash \{ v \}$ emits at least one edge to $v$. \newline
\phantom{} \newline
\textbf{Step 3: Ensuring every vertex emits exactly one edge to $v$.} \newline
\phantom{} \newline
Let $g$ denote the single loop based at $v$. If $w \in (F')^0 \backslash \{ v \}$ then let
\begin{equation*}
(F')_{w}^1 = \{ e \in E^1 \; \vert \; \text{$s(e) = w$ and $r(e) = v$} \} 
\end{equation*}
and $h_w$ be an edge in $(F')_{w}^1$. Apply move (O) to $v$ with the following partition of $r^{-1}(v)$:
\begin{equation} \label{equation: QDS O Partition}
{\bigsqcup}_{w \in (F')^0 \backslash \{ v \}, \; \lvert (F')_w^1 \rvert > 1} \big((F')_{w}^1 \backslash \{ h_w \} \big) \sqcup \big\{ g, h_w \; \vert \; w \in (F')^0 \backslash \{ v \} \big\}. 
\end{equation}
The idea behind this partition is that the set containing the single loop $g$ takes exactly one edge from all the sets $(F')_{w}^1$. Applying move (O) with this partition replaces $v$ with vertices $\{ x_1, x_2, \dots, x_r \} \sqcup \{ y \}$ such that for $i \in \{1, 2, \dots, r \}$, $x_i$ emits exactly one edge to $y$. The vertex $y$ receives exactly one edge from every vertex $w \in (F')^0 \backslash \{ v \}$ and supports a single loop. So if $G$ is the resulting graph whose vertices are
\begin{equation*}
G^0 = \big( (F')^0 \backslash \{ v \} \big) \sqcup \{ x_1, \dots, x_r \} \sqcup \{ y \}
\end{equation*}
and $F$ is the subgraph of $G$ whose vertices are $G^0 \backslash \{ y \}$ then $G = S^2 F$. By Proposition \ref{proposition: Refined Moves Stable Iso}, $C^*(E)$ and $C^*(S^2 F)$ are stably isomorphic. To see that $C^*(H)$ and $C^*(F)$ are also stably isomorphic, we track how the application of moves to $E$ in each step affects $H$. \newline
\phantom{} \newline
In step 2, the repeated application of step 1 consists of moves which all happen within $H$. Let $H'$ be the finite graph obtained from $H$ after the application of such moves. The application of move (O) in step 3 produces extra vertices which emit exactly one edge to a vertex supporting a single loop. Within the subgraph $F$ of $S^2 F$, these new vertices are regular sinks. By applying move (S) to remove them, we recover the finite graph $H'$ from $F$. Thus by Proposition \ref{proposition: Refined Moves Stable Iso}, $C^*(F)$ is stably isomorphic to $C^*(H')$ and hence, stably isomorphic to $C^*(H)$. This completes the proof. 
\end{proof}

\begin{example}
To illustrate the steps in the proof of Proposition \ref{proposition: QDS Forcing}, suppose that $E$ is the graph
\begin{equation*}
\begin{tikzpicture}[baseline=(current  bounding  box.center)] 
\node (1) at (0, 0) {$x$};
\node (2) at (3, 0) {$w$};
\node (3) at (6, 0) {$v$};
\draw[->] (1) to [bend left = 20] (2);
\draw[->] (1) to [bend right = 20] (2);
\draw[->] (2) to [out = 45, in = 135, looseness = 10] (2);
\draw[->] (2) to [out = 225, in = 315, looseness = 10] (2);
\draw[->] (2) to (3);
\draw[->] (3) to [out = 45, in = 135, looseness = 10] (3);
\end{tikzpicture}
\end{equation*}
Applying move (I-) to the vertex $w$ with the partition in \eqref{equation: QDS I- Partition} produces the graph \vspace{-1 em}
\begin{equation} \label{equation: QDS Example 1}
\begin{tikzpicture}[baseline=(current  bounding  box.center)] 
\node (1) at (0, 0) {$x$};
\node (2) at (3, 1) {$w_1$};
\node (4) at (3, -1) {$w_2$};
\node (3) at (6, 0) {$v$};
\draw[->] (1) to node[midway, above] {(2)} (2);
\draw[->] (1) to node[midway, below] {(2)} (4);
\draw[->] (2) to node[midway, right] {(2)} (4);
\draw[->] (4) to (3);
\draw[->] (2) to [out = 15, in = 75, looseness = 10] (2);
\draw[->] (2) to [out = 90, in = 150, looseness = 10] (2);
\draw[->] (3) to [out = 45, in = 135, looseness = 10] (3);
\end{tikzpicture}
\end{equation}
In graph \ref{equation: QDS Example 1}, the numbers labelling the edges denote their multiplicities. Next, applying move (CO) to the vertex $w_2$ in graph \eqref{equation: QDS Example 1} produces the graph \vspace{-1 em}
\begin{equation} \label{equation: QDS Example 2}
\begin{tikzpicture}[baseline=(current  bounding  box.center)] 
\node (1) at (0, -1) {$x$};
\node (2) at (0, 1) {$w_1$};
\node (3) at (3, 0) {$v$};
\draw[->] (1) to node[midway, left] {(2)} (2);
\draw[->] (1) to node[midway, below] {(2)} (3);
\draw[->] (2) to node[midway, above] {(2)} (3);
\draw[->] (2) to [out = 15, in = 75, looseness = 10] (2);
\draw[->] (2) to [out = 90, in = 150, looseness = 10] (2);
\draw[->] (3) to [out = 45, in = 135, looseness = 10] (3);
\end{tikzpicture}
\end{equation}
Finally, applying move (O)  to the vertex $v$ in graph \ref{equation: QDS Example 2} with the partition in \eqref{equation: QDS O Partition} produces the following graph, which as promised by Proposition \ref{proposition: QDS Forcing}, is a quantum double suspension graph: \vspace{-1 em}
\begin{equation*} 
\begin{tikzpicture}[baseline=(current  bounding  box.center)] 
\node (1) at (0, -1) {$x$};
\node (2) at (0, 1) {$w_1$};
\node (3) at (4, -1) {$x_2$};
\node (4) at (4, 1) {$x_1$};
\node (5) at (6, 0) {$y$};
\draw[->] (1) to node[midway, left] {(2)} (2);
\draw[->] (2) to (4);
\draw[->] (1) to (3);
\draw[->] (1) to (5);
\draw[->] (2) to (5);
\draw[->] (3) to (5);
\draw[->] (4) to (5);
\draw[->] (2) to [out = 15, in = 75, looseness = 10] (2);
\draw[->] (2) to [out = 90, in = 150, looseness = 10] (2);
\draw[->] (5) to [out = 315, in = 45, looseness = 10] (5);
\end{tikzpicture}
\end{equation*}
\end{example} 

\section{The Cu-semigroup $\Cu_E$ associated to a finite graph $E$}
\label{sec: Cu E}

In this section, we introduce and study the central object of this paper: a Cu-semigroup $\Cu_E$ associated to a finite graph $E$ satisfying the standing assumptions. The definition of $\Cu_E$ is based on a finite, acyclic graph $S(E)$ called the \emph{skeleton} of $E$, which retains only the graphical information from $E$ that is necessary to describe the Cuntz semigroup of $C^*(E)$. \newline
\phantom{} \newline
In \cref{subsec: Cu E Ideal Lattice}, we utilise the results on the ideals of $C^*(E)$ from \cref{subsec: Standing Assumptions} to prove that there is a bijection from the ideal lattice of $C^*(E)$ to the ideal lattice of $\Cu_E$. This observation is the primary motivation behind the construction of the map $\Delta_E: \Cu_E \rightarrow \Cu(C^*(E))$ in Definition \ref{definition: Delta E}; see Remark \ref{remark: Delta E From Singleton} for further details. In \cref{subsec: Exact Sequence}, we show that $\Cu_E$ fits into a short exact sequence of positively ordered monoids, mirroring the one induced by writing $C^*(E)$ as an extension by a gauge-invariant ideal. The Cu-morphisms comprising the short exact sequence in \eqref{equation: Cu E SES} will be extensively used in the proof of Theorem \ref{introtheorem: Cu E Iso}; see Lemma \ref{lemma: Ideal Square} and Lemma \ref{lemma: Quotient Square} in particular. 

\subsection{The skeleton $S(E)$ of a finite graph $E$}

Let $F$ be a finite graph. The following structural features of $F$ make no difference to $\Cu(C^*(F))$:
\begin{enumerate}
\item If $v \in F^0$ is either a regular sink or a regular transitional vertex, then $[p_v]$ is a linear combination of Cuntz classes $[p_w]$ such that $w$ emits an edge to $v$. Thus, $[p_v]$ does not contribute any novel Cuntz classes to $\Cu(C^*(F))$. By Proposition \ref{proposition: Collapse Graph}, such vertices can be removed from $F$ without changing $\Cu(C^*(F))$. This is why we may assume that a finite graph $E$ satisfies the standing assumptions in order to define $\Cu_E$. 
\item If $v, w \in F^0$ are vertices that lie on the same cycle, then by the Cuntz-Krieger relations, $[p_v] = [p_w]$ because $v$ and $w$ admit finite paths to each other. Hence, after we reduce to the case where $F$ satisfies the standing assumptions using Proposition \ref{proposition: Collapse Graph}, we will further simplify $F$ by collapsing all cycles, including their vertices, to a single vertex.
\item Let $n \geq 1$ and $e_1, e_2, \dots, e_n \in F^1$ be distinct edges with the same source and range vertices $v$ and $w$ respectively. Suppose also that $v \neq w$ and $w$ supports at least one cycle. By the Cuntz-Krieger relations, if $i \in \{1, 2, \dots, n \}$ then $[p_{w}] + [s_{e_i} s_{e_i}^*] = [p_{w}]$ and $[p_w] + \infty \cdot [p_v] = [p_w]$. Notice that this holds regardless of the value of $n$. So similarly to the second reduction, we will collapse such edges to a single edge. 
\end{enumerate}
By applying the second and third reductions to a finite graph $E$ satisfying the standing assumptions, we obtain a graph with no cycles which together with a colouring of the vertices with three colours, we call the skeleton of $E$. By construction, the skeleton of $E$ will retain only the information from $E$ that is essential to the Cuntz semigroup of $C^*(E)$ and in \cref{subsec: Cu E}, will be used to defined $\Cu_E$. 

\begin{definition} 
A \textbf{skeleton} is a pair $(S, c)$ where $S$ is a finite, acyclic graph such that there exists at most one edge between any pair of distinct vertices in $S$ and $c$ is a function from $S^0$ to $\{0, 1, \infty \}$. 
\end{definition}

Throughout the paper, we will more simply refer to a skeleton $(S, c)$ by its underlying finite graph $S$. Now to a finite graph $E$ satisfying the standing assumptions, we associate to it a skeleton by carrying out the reductions outlined for the second and third structural features of a finite graph which are redundant to the Cuntz semigroup. 

\begin{definition} \label{definition: Skeleton Graph}
Let $E$ be a finite graph satisfying the standing assumptions. Define an equivalence relation $\sim_E$ on $E^0$ by setting for $v, w \in E^0$, $v \sim_E w$ if and only if $v \geq w$ and $w \geq v$. That is, $v$ and $w$ both belong to the same strongly connected component in $E$. The \textbf{skeleton} of $E$, denoted by $S(E)$, is a graph defined by the following data:
\begin{enumerate}
\item The vertices $S(E)^0$ are elements of the set $E^0/\sim_E$, which will be denoted by $\check{v}$ for $v \in E^0$. 
\item If $\check{v}, \check{w} \in S(E)^0$ are distinct then connect $\check{v}$ to $\check{w}$ by a single edge if and only if there exist $v', w' \in E^0$ and $e \in E^1$ such that $v \sim_E v'$, $w \sim_E w'$, $s(e) = v'$ and $r(e) = w'$ in the graph $E$. 
\end{enumerate}
The graph $S(E)$ is also accompanied by the function $c: S(E)^0 \rightarrow \{0, 1, \infty \}$ defined by
\begin{equation*}
c(\check{v}) = 
\begin{cases}
0, \text{ if $v \in E^0$ is a source,} \\
1, \text{ if $v \in E^0$ supports a single loop,} \\
\infty, \text{ if $v \in E^0$ supports more than one cycle.} 
\end{cases}
\end{equation*}
\end{definition}

By construction, the pair $(S(E), c)$ is a skeleton. Observe that if $v, w \in E^0$ then $v \geq w$ if and only if $\check{v} \geq \check{w}$ in $S(E)$. Hence, there is a bijective correspondence
\begin{equation} \label{equation: Hereditary Skeleton}
\begin{matrix}
\big\{ \text{Hereditary subsets of $E^0$} \big\} & \longleftrightarrow & \big\{ \text{Hereditary subsets of $S(E)^0$}  \big\} \\
H & \mapsto & \langle H \rangle = \{ \check{v} \in S(E)^0 \; \vert \; v \in H \} \\
S^{\dagger} & \mapsfrom & S
\end{matrix}
\end{equation}
where for a hereditary subset $S$ of $S(E)^0$,
\begin{equation} \label{equation: Dagger}
S^{\dagger} = \{ x \in E^0 \; \vert \; \text{There exists $y \in E^0$ such that $\check{y} \in S$ and $x \geq y$} \}.
\end{equation}
Subsequently, the hereditary subsets of $S(E)^0$ recover the gauge-invariant ideal lattice of $C^*(E)$ by \eqref{equation: Hereditary Saturated} and Corollary \ref{corollary: Auto Saturated}. Furthermore, we will show that the number of vertices of $S(E)$ corresponds directly to the size of the gauge-invariant ideal lattice of $C^*(E)$. This fact facilitates the core strategy for proving Theorem \ref{introtheorem: Cu E Iso} --- by induction over the gauge-invariant ideal lattice of $C^*(E)$. 

\begin{definition} \label{definition: Gauge Length}
Let $E$ be a finite graph and $n > 0$. The graph C*-algebra $C^*(E)$ is said to have \textbf{gauge length} $n$ if there exists an ascending chain of gauge-invariant ideals
\begin{equation*}
I_0 = \{ 0 \} \subseteq I_1 \subseteq \dots \subseteq I_{n-1} \subseteq I_n = C^*(E)
\end{equation*}
such that if $j \in \{0, 1, \dots, n-1 \}$ and $H_j$ is the hereditary subset corresponding to $I_j$ via \eqref{equation: Hereditary Saturated} then the quotient $C^*(H_{j+1} \backslash H_j)$ is gauge-simple. 
\end{definition}

\begin{lemma} \label{lemma: Gauge Simple}
Let $E$ be a finite graph satisfying the standing assumptions. Then $C^*(E)$ is gauge-simple (equivalently, has gauge length one) if and only if $\lvert S(E)^0 \rvert = 1$. 
\end{lemma}

\begin{proof}
First assume that $S(E)^0 = \{ \check{v} \}$. If $c(\check{v}) = 0$ then by standing assumptions \ref{property: S1} and \ref{property: S2}, $v$ is a source vertex in $E$. Moreover, $v$ is the unique vertex of $E$ and consequently, $C^*(E) \cong \mathbb{C}$ is simple. If $c(\check{v}) = 1$ then $v$ supports a single loop in $E$. Since $S(E)^0 = \{ \check{v} \}$ then $v$ is again the unique vertex in $E$ and so $C^*(E) \cong C(\mathbb{T})$ is gauge-simple. Finally, if $c(\check{v}) = \infty$ then $v$ supports multiple cycles in $E$ and thus, $E$ cannot consist of a single cycle. Since every vertex $w \in E^0$ satisfies $v \geq w$ and $w \geq v$ then by \cite[Proposition 4.2]{Rae05} and \cite[Remark 5.6]{BPRS00}, $C^*(E)$ is simple purely infinite. In all cases, we have that $C^*(E)$ is gauge-simple. \newline
\phantom{} \newline
Conversely, assume that $\lvert S(E)^0 \rvert > 1$. Let $\check{v}, \check{w} \in S(E)^0$ be distinct. By definition of $S(E)$, either $\check{v} \not \geq \check{w}$ or $\check{w} \not \geq \check{v}$ in $E$. So either $L(v)$ or $L(w)$ is a proper hereditary subset of $E^0$. Hence, $C^*(E)$ is not gauge-simple. 
\end{proof}

Note that without the standing assumptions, Lemma \ref{lemma: Gauge Simple} is not true. For instance, if $E$ consists of a single edge between two vertices then $\lvert S(E)^0 \rvert = 2$, but $C^*(E) \cong M_2(\mathbb{C})$ is gauge-simple.

\begin{lemma} \label{lemma: Gauge Length}
Let $E$ be a finite graph satisfying the standing assumptions. Then the gauge length of $C^*(E)$ is $\lvert S(E)^0 \rvert$. 
\end{lemma}

\begin{proof}
We proceed by induction on the gauge length of $C^*(E)$. The base case where $C^*(E)$ has gauge length one is done by Lemma \ref{lemma: Gauge Simple}. Now suppose that the statement is true for graphs $E$ such that $C^*(E)$ has gauge length $n > 0$. Assume that $C^*(E)$ has gauge length $n + 1$. There exists an ascending chain of gauge-invariant ideals
\begin{equation*}
\{ 0 \} = I_0 \subseteq I_1 \subseteq \dots \subseteq I_n \subseteq I_{n+1} = C^*(E)
\end{equation*}
and its associated ascending chain of hereditary subsets
\begin{equation*}
\emptyset = H_0 \subseteq H_1 \subseteq \dots \subseteq H_n \subseteq H_{n+1} = E^0
\end{equation*}
such that if $j \in \{0, 1, \dots, n \}$ then $C^*(H_{j+1} \backslash H_j)$ is gauge-simple. Restricting to the subgraph $H_n$, we have the ascending chain of gauge-invariant ideals
\begin{equation*}
\{ 0 \} = I_0 \subseteq I_1 \subseteq \dots \subseteq I_{n-1} \subseteq C^*(H_n).
\end{equation*}
So $C^*(H_n)$ has gauge length $n$ and by the inductive hypothesis, $\lvert S(H_n)^0 \rvert = n$. Now by assumption, the quotient $C^*(E \backslash H_n)$ is gauge-simple and subsequently $\lvert S(E \backslash H_n)^0 \rvert = 1$ by the base case. Therefore, 
\begin{equation*}
\lvert S(E)^0 \rvert = \lvert S(H_n)^0 \rvert +  \lvert S(E \backslash H_n)^0 \rvert = n + 1.
\end{equation*}
This completes the induction. 
\end{proof}

In the following definition, we isolate specific vertices of hereditary subsets of a skeleton which are critical to understanding the order and the way-below relation on $\Cu_E$ in the next subsection.

\begin{definition}
Let $(S, c)$ be a skeleton and $H$ be a hereditary subset of $S$. We say that a vertex $v \in H$ is \textbf{terminal} (or a terminal vertex for $H$) if for $w \in H$ satisfying $v \geq w$, $v = w$. The set of vertices in $H$ which are terminal is denoted by $H_t$. 
\end{definition} 

The following lemma provides a useful connection between a hereditary subset of a skeleton and its terminal vertices. 

\begin{lemma} \label{lemma: Hereditary Terminal}
Let $S$ be a skeleton and $H$ be a hereditary subset of $S$. For $v \in S^0$, let $L_S(v)$ be the hereditary subset $\{ x \in S^0 \; \vert \; x \geq v \}$. Then
\begin{equation*}
H = {\bigcup}_{v \in H_t} L_S(v). 
\end{equation*}
\end{lemma}

\begin{proof}
The inclusion $\bigcup_{v \in H_t} L_S(v)\subseteq H$ follows from the assumption that $H$ is hereditary. To see that the reverse inclusion holds, assume that $w \in H$. We claim that there exists $v \in H_t$ such that $w \geq v$. Suppose for the sake of contradiction that $w$ does not admit a finite path to any terminal vertex for $H$. Then $w$ is not a terminal vertex for $H$ and so, there exists $w_1 \in H$ such that $w \geq w_1$ and $w_1 \not \geq w$. By our assumption on $w$, $w_1$ is also not terminal. \newline
\phantom{} \newline
By repeating this argument indefinitely, there exist vertices $w_1, w_2, \dots$ in $H$ which are not terminal vertices for $H$ such that
\begin{equation*}
w \geq w_1 \geq w_2 \geq w_3 \geq \dots
\end{equation*}
and $w_{i+1} \not \geq w_i$ for $i \in \{1, 2, \dots \}$. In particular, the vertices $w, w_1, w_2, \dots$ are all pairwise distinct and hence, we have constructed a path of infinite length in $H$, contradicting the fact that $S$ has finitely many vertices. Thus, there exists a terminal vertex $x \in H$ such that 
\begin{equation*}
w \in L_S(x) \subseteq {\bigcup}_{v \in H_t} L_S(v). 
\end{equation*}
So $H \subseteq \bigcup_{v \in H_t} L_S(v)$. 
\end{proof}

\subsection{Definition of $\Cu_E$ and proof that it is a Cu-semigroup}
\label{subsec: Cu E}

Let $(S, c)$ be a skeleton and $H$ be a hereditary subset of $S$. Define $H_\kappa$ to be the set
\begin{equation*}
H_\kappa = \{ x \; \vert \; x \in H, \; c(x) \neq 1 \} \sqcup \big\{ \mathbb{T}_{x} \; \vert \; x \in H, \; c(x) = 1 \big\}.
\end{equation*}
Intuitively, if $E$ is a finite graph satisfying the standing assumptions, then $\Cu_E$ consists of functions from $\langle E \rangle_\kappa = (S(E)^0)_\kappa$ to $\overline{\mathbb{N}}$ satisfying properties which reflect certain properties of $\Cu(C^*(E))$. In particular, if $f \in \Cu_E$, $f(\check{v})$ is non-zero, $\check{w} \geq \check{v}$ and $\check{w} \neq \check{v}$, then we will insist that $f(\check{w}) = \infty$; this property corresponds to the generalisation of the Cuntz-Krieger relations for $\Cu(C^*(E))$ which come from applying Theorem \ref{introtheorem: General Cuntz Krieger} to subgraphs of $C^*(E)$ satisfying the necessary assumptions. \newline
\phantom{} \newline
More generally, we will associate to a skeleton $S$ a Cu-semigroup $\Cu_S$. To be clear about notation in the definition which follows, if $f \in \Cu_S$, $v \in S^0$ and $c(v) = 1$, then $f(v)$ refers to the restriction of $f$ to $\mathbb{T}_{v}$. When we write $f(v) = 0$, we mean that $f(v)$ is equal to the zero function on $\mathbb{T}_{v}$. Similarly, $f(v) = \infty$ means that $f(v)$ is equal to the constant, infinite-valued function on $\mathbb{T}_{v}$. 

\begin{definition} \label{definition: Cu E}
Let $(S, c)$ be a skeleton. Define $\Cu_S$ to be the set of functions $f: S_\kappa \rightarrow \overline{\mathbb{N}}$ which satisfy the following properties: 
\begin{enumerate}
\item If $v \in S^0$ and $c(v) = 1$ then $f(v) \in \Lsc(\mathbb{T}, \overline{\mathbb{N}})$. 
\item If $v \in S^0$ and $c(v) = \infty$ then $f(v) \in \{0, \infty \}$. 
\item If $v, w \in S^0$, $w \geq v$, $w \neq v$ and $f(v) \neq 0$ then $f(w) = \infty$. 
\end{enumerate}
In the special case where $(S, c)$ is the skeleton of a finite graph $E$ satisfying the standing assumptions, we denote $\Cu_{S(E)}$ more simply by $\Cu_E$. 
\end{definition}

Observe that if $(S, c)$ is a skeleton and $f \in \Cu_S$, then by the third property in Definition \ref{definition: Cu E}, the \textbf{vertex support} of $f$, denoted by
\begin{equation*}
\supp_0(f) = \{ v \in S^0 \; \vert \; f(v) \neq 0 \},
\end{equation*}
is a hereditary subset of $S^0$. Equipped with the pointwise addition and order, it follows from direct computation that $\Cu_S$ is a positively ordered monoid. In particular, $\Cu_S$ is closed under addition because the sum of two functions in $\Cu_S$ remain 0 and $\infty$ at precisely the points in $S_\kappa$ needed for the sum to lie in $\Cu_S$ by Definition \ref{definition: Cu E}. \newline
\phantom{} \newline
Lemma \ref{lemma: Order Alternative} is an alternative characterisation of the pointwise order on $\Cu_S$, which says that the order on $\Cu_S$ only depends on certain terminal vertices of the vertex support of the larger function.

\begin{lemma} \label{lemma: Order Alternative}
Let $(S, c)$ be a skeleton and $f, g \in \Cu_S$. Then $f \leq g$ if and only if the following statements are satisfied:
\begin{enumerate}
\item $\supp_0(f) \subseteq \supp_0(g)$. 
\item Suppose that $w \in \supp_0(g)$ is a terminal vertex such that $c(w) \neq \infty$. Then $f(w) \leq g(w)$. 
\end{enumerate}
\end{lemma}

\begin{proof}
The ``only if" direction follows from the fact that the order on $\Cu_S$ is pointwise. Conversely, assume that $f, g \in \Cu_S$ satisfy the two properties in the statement of the lemma. Assume that $w \in S^0$. If $w \in \supp_0(g)$ is terminal and $c(w) \neq \infty$ then by the second assumption, $f(w) \leq g(w)$. Otherwise, $g(w) \in \{0, \infty \}$ and since $\supp_0(f) \subseteq \supp_0(g)$ then $f(w) \leq g(w)$. Using the fact that the orders on $\overline{\mathbb{N}}$ and $\Lsc(\mathbb{T}, \overline{\mathbb{N}})$ are both pointwise, $f \leq g$ pointwise. 
\end{proof}

We would like a characterisation of the way-below relation $\ll$ on $\Cu_S$ which is similar to Lemma \ref{lemma: Order Alternative}. We may speak of the way-below relation in this context because $\Cu_S$ satisfies \ref{axiom: O1}. Indeed, if $\{ h_n \}_{n > 0}$ is an increasing sequence in $\Cu_S$, then its supremum is simply the pointwise supremum. 

\begin{lemma} \label{lemma: Way Below}
Let $(S, c)$ be a skeleton and $f, g \in \Cu_S$. Then $f \ll g$ if and only if the following statements are satisfied:
\begin{enumerate}
\item $\supp_0(f) \subseteq \supp_0(g)$. 
\item Suppose that $w \in \supp_0(g)$ is a terminal vertex satisfying $c(w) \neq \infty$. Then $f(w) \ll g(w)$. 
\end{enumerate}
\end{lemma}

\begin{proof}
Assume that $f, g \in \Cu_S$ and $f \ll g$. Then $f \leq g$ and by Lemma \ref{lemma: Order Alternative}, $\supp_0(f) \subseteq \supp_0(g)$. To see that the second property is satisfied, assume that $w \in \supp_0(g)$ is a terminal vertex satisfying $c(w) = 0$. Then $g(w) \in \overline{\mathbb{N}} \backslash \{ 0 \}$. Let $\{ n_k \}_{k > 0}$ be an increasing sequence in $\overline{\mathbb{N}}$ such that $n_1 \neq 0$ and $g(w) \leq \sup_{k > 0} n_k$. If $k \in \{1, 2, \dots \}$ then the function 
\begin{equation*}
\begin{matrix}
h_k: & S_{\kappa} & \longrightarrow & \overline{\mathbb{N}} \\
& v & \longmapsto & 
\begin{cases}
n_k, \text{ if $v = w$}, \\
g(v), \text{ otherwise}.
\end{cases}
\end{matrix}
\end{equation*}
gives an increasing sequence $\{ h_k \}_{k}$ in $\Cu_S$. Let $h = \sup_{k > 0} h_k$ and $n = \sup_{k > 0} n_k$. Since the supremum in $\Cu_S$ is given pointwise, $${\sup}_{k > 0} \; h_k = h \geq g.$$ Now using that $f \ll g$, there exists $\ell \in \{1, 2, \dots \}$ such that $f \leq h_\ell$. In particular, $f(w) \leq h_\ell(w) = n_\ell$. So $f(w) \ll g(w)$ in $\overline{\mathbb{N}}$. An analogous argument handles the case where $c(w) = 1$. This proves the ``only if" direction of the lemma. \newline
\phantom{} \newline
Conversely, assume that the two properties in the statement of the lemma are satisfied. Let $\{ a_m \}_{m > 0}$ be an increasing sequence in $\Cu_S$ with supremum $a$ and suppose that $a \geq g$. Then we have an ascending chain 
\begin{equation*}
\supp_0(a_1) \subseteq \supp_0(a_2) \subseteq \supp_0(a_3) \subseteq \dots
\end{equation*}
of hereditary subsets of $S^0$. Since $S$ has finitely many vertices then there are finitely many hereditary subsets of $S^0$. Hence, the ascending chain must terminate and there exists $m' \in \{1, 2, \dots \}$ such that for $j \geq 0$, $$\supp_0(g) \subseteq \supp_0(a) = \supp_0(a_{m' + j}) = \supp_0(a_{m'}).$$ Now assume that $w \in \supp_0(a)$ is a terminal vertex such that $c(w) \neq \infty$. If $w \not \in \supp_0(g)$ then $w \not \in \supp_0(f)$ and so for $\ell \geq m'$, $f(w) = 0 \leq a_\ell(w)$. Otherwise, if $w \in \supp_0(g)$ then by the second assumption, $$f(w) \ll g(w) \leq a(w)$$ in $\overline{\mathbb{N}}$ or $\Lsc(\mathbb{T}, \overline{\mathbb{N}})$ and there exists $\ell' \geq m'$ such that $f(w) \leq a_{\ell'}(w)$. Consequently, there exists $\ell \geq m'$ large enough such that if $w \in \supp_0(a) = \supp_0(a_\ell)$ is terminal  and $c(w) \neq \infty$ then $f(w) \leq a_{\ell}(w)$. Since $\supp_0(f) \subseteq \supp_0(g) \subseteq \supp_0(a) = \supp_0(a_\ell)$ then by Lemma \ref{lemma: Order Alternative}, $f \leq a_\ell$ and thus, $f \ll g$ in $\Cu_S$. This completes the proof. 
\end{proof}

With Lemma \ref{lemma: Way Below}, we can now check that $\Cu_S$ satisfies \ref{axiom: O3}.

\begin{lemma} \label{lemma: Axiom O3}
Let $(S, c)$ be a skeleton and $f, f', g, g' \in \Cu_S$. If $f' \ll f$ and $g' \ll g$ then $f' + g' \ll f + g$. That is, $\Cu_S$ satisfies \ref{axiom: O3}. 
\end{lemma}

\begin{proof}
Assume that $f, f', g, g' \in \Cu_S$ are such that $f' \ll f$ and $g' \ll g$. Then $f' + g' \leq f + g$ pointwise and $\supp_0(f' + g') \subseteq \supp_0(f + g)$. Next, assume that $x \in \supp_0(f + g) = \supp_0(f) \cup \supp_0(g)$ is a terminal vertex for $\supp_0(f + g)$ such that $c(x) \neq \infty$. Then either $x\in \supp_0(f)$ is terminal or $x \in \supp_0(g)$ is terminal. In either case, we have
\begin{equation*}
(f' + g')(x) = f'(x) + g'(x) \ll f(x) + g(x) = (f + g)(x).
\end{equation*}
Hence by Lemma \ref{lemma: Way Below}, $f' + g' \ll f + g$ in $\Cu_S$.
\end{proof}

To verify that $\Cu_S$ satisfies \ref{axiom: O2}, we will rely on a decomposition of an arbitrary function in $\Cu_S$ into a sum of functions with particular vertex supports. 

\begin{definition} \label{definition: Singletons}
Let $(S, c)$ be a skeleton and $v \in S^0$ be such that $c(v) = 0$ and $n \in \overline{\mathbb{N}} \backslash \{ 0 \}$. Let $L_S(v) = \{ x \in S^0 \; \vert \; x \geq v \}$. Define 
\begin{equation*}
\theta_{v, n} = n \cdot \chi_{\{ v \}} + \infty \cdot \chi_{L_S(v)_\kappa \backslash \{ v \}}. 
\end{equation*}
Next suppose that $c(v) = \infty$. Define $\theta_{v, \infty} = \infty \cdot \chi_{L_S(v)_\kappa}$. Finally, suppose that $c(v) = 1$ and $f \in \Lsc(\mathbb{T}, \overline{\mathbb{N}}) \backslash \{ 0 \}$. Define 
\begin{equation*}
\theta_{v, f} = f \vert_{\mathbb{T}_{v}} + \infty \cdot \chi_{L_S(v)_\kappa \backslash \mathbb{T}_{v}}.
\end{equation*}
\end{definition}

The functions in Definition \ref{definition: Singletons} are all elements of $\Cu_S$ with vertex support $L_S(v)$. 

\begin{lemma} \label{lemma: Singleton Decomp}
Let $(S, c)$ be a skeleton and $f \in \Cu_S$. Then
\begin{equation} \label{equation: Singleton Decomp}
f = {\sum}_{v \in \supp_0(f)_t} \theta_{v, f(v)}.
\end{equation}
\end{lemma}

\begin{proof}
By Definition \ref{definition: Singletons},  the RHS of \eqref{equation: Singleton Decomp} is an element of $\Cu_S$, which we denote by $g$. By construction of $g$ and Lemma \ref{lemma: Hereditary Terminal},
\begin{equation*}
\supp_0(g) = {\bigcup}_{v\in \supp_0(f)_t} L_S(v) = \supp_0(f).
\end{equation*}
Lemma \ref{lemma: Order Alternative} finishes the proof.
\end{proof}

\begin{lemma} \label{lemma: Axiom O2}
Let $(S, c)$ be a skeleton and $f \in \Cu_S$. There exists a rapidly increasing sequence $\{ f_n \}_{n > 0}$ in $\Cu_S$ such that $\sup_{n > 0} f_n = f$. That is, $\Cu_S$ satisfies \ref{axiom: O2}. 
\end{lemma}

\begin{proof}
Assume that $f \in \Cu_S$. By Lemma \ref{lemma: Singleton Decomp},
\begin{equation*} 
f = {\sum}_{v \in \supp_0(f)_t} \theta_{v, f(v)}.
\end{equation*}
If $v \in \supp_0(f)_t$ and $c(v) = 0$ then $f(v) \neq 0$ in $\overline{\mathbb{N}}$ and there exists a rapidly increasing sequence $\{ n_{v, k} \}_{k > 0}$ such that $n_{v, 1} \neq 0$ and $\sup_{k > 0} n_{v, k} = f(v)$. Hence the sequence $\{ \theta_{v, n_{v, k}} \}_{k > 0}$ is a rapidly increasing sequence in $\Cu_S$ satisfying ${\sup}_{k > 0} \; \theta_{v, n_{v, k}}  =  \theta_{v, f(v)}$. If $c(v) = 1$ then the same argument shows that we can produce a rapidly increasing sequence in $\Cu_S$ with supremum $\theta_{v, f(v)}$. By Lemma \ref{lemma: Way Below} and Lemma \ref{lemma: Axiom O3}, we obtain a rapidly increasing sequence in $\Cu_S$ with supremum $f$. 
\end{proof}

\begin{theorem}
Let $(S, c)$ be a skeleton. Then $\Cu_S$ is a Cu-semigroup. In particular, if $E$ is a finite graph satisfying the standing assumptions then $\Cu_E$ is a Cu-semigroup. 
\end{theorem}

\begin{proof}
Recall that $\Cu_S$ satisfies \ref{axiom: O1} because the supremum of an increasing sequence in $\Cu_S$ is simply given by the pointwise supremum. It follows directly from this that $\Cu_S$ satisfies \ref{axiom: O4}. Finally, by Lemma \ref{lemma: Axiom O2} and Lemma \ref{lemma: Axiom O3}, $\Cu_S$ satisfies \ref{axiom: O2} and \ref{axiom: O3}. So $\Cu_S$ is a Cu-semigroup. 
\end{proof}

\subsection{The ideal lattice of $\Cu_E$}
\label{subsec: Cu E Ideal Lattice}

Throughout this subsection, $E$ is a finite graph satisfying the standing assumptions. Recall from \eqref{equation: Ideal Form} that any ideal $J$ of $C^*(E)$ can be written as
\begin{equation} \label{equation: Non Gauge Invariant Form}
J = I_H \cap {\bigcap}_{i = 1}^n \mathcal{R}_{v_i, K_i}
\end{equation}
where $H$ is a hereditary subset of $E^0$, for $i \in \{1, 2, \dots, n \}$, $v_i \in E^0$ supports a single loop and $K_i$ is a closed subset of $\mathbb{T}$. If $J$ is non-gauge-invariant, then we will assume without loss of generality that it satisfies properties \ref{property: I1}, \ref{property: I2} and \ref{property: I3}. In the degenerate case where $n = 0$, the ideal in \eqref{equation: Non Gauge Invariant Form} becomes the gauge-invariant ideal $I_H$. \newline
\phantom{} \newline
Define $\mathcal{D}_J$ to be the hereditary subset
\begin{equation} \label{equation: Hereditary DJ}
\mathcal{D}_J = \Big( H \cap {\bigcap}_{i = 1}^n (E^0 \backslash R(v_i)) \Big) \sqcup \{v_1, \dots, v_n \}. 
\end{equation}
In the degenerate case where $n = 0$, we set $\mathcal{D}_J = H$. By \eqref{equation: Sandwich}, the gauge-invariant ideal $I_{\mathcal{D}_J}$ is the smallest gauge-invariant ideal containing $J$. \newline
\phantom{} \newline
In \eqref{equation: Ideal Lattice 2} and \eqref{equation: Ideal Lattice 1}, we define a map $\mathcal{I}$ which maps an ideal of $C^*(E)$ to an appropriate ideal in $\Cu_E$. We will prove that this map is injective in Lemma \ref{lemma: Non Gauge Invariant Ideal}. In the proof of Lemma \ref{lemma: Non Gauge Invariant Ideal}, we need to show that if $I$ and $J$ are two non-gauge-invariant ideals and $\mathcal{I}(I) = \mathcal{I}(J)$, then the finite sets of vertices which correspond to the non-gauge-invariant parts of $I$ and $J$ in \eqref{equation: Non Gauge Invariant Form} are equal. Lemma \ref{lemma: Maximal Tail Is Terminal} is the technical result which allows this step to be realised. 

\begin{lemma} \label{lemma: Maximal Tail Is Terminal}
Let $J$ be an ideal of $C^*(E)$ as in \eqref{equation: Non Gauge Invariant Form}. If $j \in \{1, 2, \dots, n \}$ then $\check{v}_j \in S(E)^0$ is a terminal vertex for the hereditary subset $\langle \mathcal{D}_J \rangle$ of $S(E)^0$. 
\end{lemma}

\begin{proof}
In the degenerate case where $n = 0$, there is nothing to show. So assume that $n \in \{1, 2, \dots \}$ and that $J$ satisfies properties \ref{property: I1}, \ref{property: I2} and \ref{property: I3}. If $j \in \{1, 2, \dots, n \}$ then $\check{v}_j \in \langle \mathcal{D}_J \rangle$. To see that $\check{v}_j$ is terminal for $\langle \mathcal{D}_J \rangle$, let $\check{w} \in \langle \mathcal{D}_J \rangle$ and assume that $\check{v}_j \geq \check{w}$. Then $v_j \geq w$ in $E$. Now since $v_j \not \in H \cap \bigcap_{i = 1}^n (E^0 \backslash R(v_i))$ then $w \not \in H \cap \bigcap_{i = 1}^n (E^0 \backslash R(v_i))$ because the set $H \cap \bigcap_{i = 1}^n (E^0 \backslash R(v_i))$ is hereditary. Since $w \in \mathcal{D}_J$ then by \eqref{equation: Hereditary DJ}, there exists $k \in \{1, 2, \dots, n \}$ such that $w = v_k$. By property \ref{property: I3}, $$w = v_k \in H \cap {\bigcap}_{i = 1, \; i \neq k}^n (E^0 \backslash R(v_i))$$ and since the above set is hereditary, it must also contain $v_j$. Consequently, $j = k$, $v_j = v_k = w$ and $\check{v}_j$ is a terminal vertex for $\langle \mathcal{D}_J \rangle$. 
\end{proof}

We will now describe where ideals of $C^*(E)$ should be sent to in the ideal lattice of $\Cu_E$. First, define for $\check{v} \in S(E)^0$ the function
\begin{equation} \label{equation: Delta V}
\delta_{\check{v}} = 
\begin{cases}
1 \in \overline{\mathbb{N}}, \text{ if $c(\check{v}) = 0$,} \\
\chi_{\mathbb{T}} \in \Lsc(\mathbb{T}, \overline{\mathbb{N}}), \text{ if $c(\check{v}) = 1$,} \\
\infty \in \overline{\mathbb{N}}, \text{ if $c(\check{v}) = \infty$.}
\end{cases}
\end{equation}
Observe that for $\check{v} \in S(E)^0$, the function $\theta_{\check{v}, \delta_{\check{v}}}$ as defined in Definition \ref{definition: Singletons} is always an element of $\Cu_E$. Next, we define for an ideal $J$ as in \eqref{equation: Non Gauge Invariant Form} 
\begin{equation} \label{equation: Ideal Lattice 2}
\mathcal{I}(J) = \left\{ f \in \Cu_E \; \Big\vert \; 
\begin{matrix}
\supp_0(f) \subseteq \langle \mathcal{D}_J \rangle \text{ and for $i \in \{1, 2, \dots, n \}$,} \\
\supp(f(\check{v}_i)) \subseteq \mathbb{T} \backslash K_i
\end{matrix}  \right\}. 
\end{equation}
To be clear, in the degenerate case where $n = 0$ and $J = I_H$,
\begin{equation} \label{equation: Ideal Lattice 1}
\mathcal{I}(J) = \big\{ f \in \Cu_E \; \vert \; \supp_0(f) \subseteq \langle H \rangle \big\}.
\end{equation}

\begin{lemma} \label{lemma: Non Gauge Invariant Ideal}
Let 
\begin{equation*}
I = I_H \cap {\bigcap}_{i = 1}^n \mathcal{R}_{v_i, K_i} \qquad \text{and} \qquad J = I_{H'} \cap {\bigcap}_{j = 1}^q \mathcal{R}_{w_j, L_j}
\end{equation*}
be ideals of $C^*(E)$. Suppose that as ideals of $\Cu_E$, $\mathcal{I}(I) = \mathcal{I}(J)$. Then $I = J$. 
\end{lemma}

\begin{proof}
First, assume that $n = q = 0$ so that $I = I_H$ and $J = I_{H'}$ are both gauge-invariant ideals of $C^*(E)$. Since $\mathcal{I}(I) = \mathcal{I}(J)$ then by \eqref{equation: Ideal Lattice 1} and \eqref{equation: Hereditary Skeleton}, $\langle H \rangle = \langle H' \rangle$ and $I = I_H = I_{H'} = J$. Observe that if $I$ is gauge-invariant and $J$ is not gauge-invariant then by \eqref{equation: Ideal Lattice 1} and \eqref{equation: Ideal Lattice 2} respectively, $\mathcal{I}(I) \neq \mathcal{I}(J)$. \newline
\phantom{} \newline
Now assume that $n, q \in \{1, 2, \dots \}$ so that $I$ and $J$ are non-gauge-invariant. By Lemma \ref{lemma: Maximal Tail Is Terminal}, $\{ \check{v}_1, \dots, \check{v}_n \} \subseteq \langle \mathcal{D}_I \rangle_t$. Using this, we define 
\begin{equation*}
g = {\sum}_{\check{x} \in \langle \mathcal{D}_I \rangle_t \backslash \{ \check{v}_1, \dots, \check{v}_n \}} \; \theta_{\check{x}, \delta_{\check{x}}}  + {\sum}_{i = 1}^n \theta_{\check{v}_i, \; \chi_{\mathbb{T} \backslash K_i}}. 
\end{equation*}
By Lemma \ref{lemma: Hereditary Terminal}, 
\begin{equation*}
\supp_0(g) = {\bigcup}_{\check{x} \in \langle \mathcal{D}_I \rangle_t} L_{S(E)}(\check{x}) = \langle \mathcal{D}_I \rangle.
\end{equation*}
Consequently, $g \in \mathcal{I}(I)$ and $\langle \mathcal{D}_I \rangle = \supp_0(g) \subseteq \langle \mathcal{D}_J \rangle$. By interchanging the roles of $H$ and $H'$ and the roles of $\{ v_1, \dots, v_n \}$ and $\{ w_1, \dots, w_q \}$, $\langle \mathcal{D}_I \rangle = \langle \mathcal{D}_J \rangle$. \newline
\phantom{} \newline
Next we claim that $\{ v_1, \dots, v_n \} = \{w_1, \dots, w_q \}$. Assume that $k \in \{1, 2, \dots, n \}$ and suppose for the sake of contradiction that there is no $j \in \{1, 2, \dots, q \}$ such that $v_k = w_j$. By Lemma \ref{lemma: Maximal Tail Is Terminal}, $\check{v}_k \in S(E)^0$ is a terminal vertex for $\langle \mathcal{D}_I \rangle = \langle \mathcal{D}_J \rangle$. Since $v_k \neq w_j$ for $j \in \{1, 2, \dots, q \}$ then by Lemma \ref{lemma: Hereditary Terminal}, $$\theta_{\check{v}_k, \; \chi_{\mathbb{T}}} \in \mathcal{I}(J) = \mathcal{I}(I).$$ Hence, $\supp(\theta_{\check{v}_k, \; \chi_{\mathbb{T}}}(\check{v}_k)) = \mathbb{T} \subseteq \mathbb{T} \backslash K_k$. This contradicts the fact that by property \ref{property: I2}, $K_k$ is a non-empty subset of $\mathbb{T}$. Hence, there exists $j \in \{1, 2, \dots, q \}$ such that $v_k = w_j$ and by a symmetric argument, $\{ v_1, \dots, v_n \} = \{ w_1, \dots, w_q \}$. Now since $\mathcal{D}_I = \mathcal{D}_J$ as hereditary subsets of $E^0$ then 
\begin{equation*}
H \cap {\bigcap}_{i = 1}^n (E^0 \backslash R(v_i)) = H' \cap {\bigcap}_{j = 1}^q (E^0 \backslash R(w_j))
\end{equation*}
and by relabelling, we may assume that $v_i = w_i$ for $i \in \{1, 2, \dots, n \}$. Now if $i \in \{1, 2, \dots, n \}$ then
\begin{equation*}
\theta_{\check{v}_i, \; \chi_{\mathbb{T} \backslash K_i}} \in \left\{ f \in \Cu_E \; \Big\vert \; 
\begin{matrix}
\supp_0(f) \subseteq \langle \mathcal{D}_I \rangle \text{ and for $i \in \{1, 2, \dots, n \}$,} \\
\supp(f(\check{v}_i)) \subseteq \mathbb{T} \backslash L_i
\end{matrix}  \right\}
\end{equation*}
and subsequently, $L_i \subseteq K_i$. By interchanging the roles of $K_i$ and $L_i$, we have $K_i = L_i$. Finally by \eqref{equation: Sandwich}, 
\begin{IEEEeqnarray*}{rCl}
I & = & I_H \cap {\bigcap}_{i = 1}^n \mathcal{R}_{v_i, K_i}
\\
& = & I_H \cap \Big({\bigcap}_{i = 1}^n I_{(E^0 \backslash R(v_i)) \sqcup \{ v_i \}} \cap {\bigcap}_{i = 1}^n \mathcal{R}_{v_i, K_i}\Big)
\\
& = & I_{H \cap \bigcap_{i = 1}^n ((E^0 \backslash R(v_i)) \sqcup \{ v_i \})} \cap {\bigcap}_{i = 1}^n \mathcal{R}_{v_i, K_i}.
\end{IEEEeqnarray*}
Using property \ref{property: I3},
\begin{equation*}
H \cap {\bigcap}_{i = 1}^n \big((E^0 \backslash R(v_i)) \sqcup \{ v_i \} \big) = \Big( H \cap {\bigcap}_{i = 1}^n (E^0 \backslash R(v_i)) \Big) \sqcup \{ v_1, \dots, v_n \} = \mathcal{D}_I.
\end{equation*}
Combining this with the equality $\mathcal{D}_I = \mathcal{D}_J$, we find that
\begin{IEEEeqnarray*}{rCl}
I & = & I_{H \cap \bigcap_{i = 1}^n ((E^0 \backslash R(v_i)) \sqcup \{ v_i \})} \cap {\bigcap}_{i = 1}^n \mathcal{R}_{v_i, K_i}
\\
& = & I_{\mathcal{D}_I} \cap {\bigcap}_{i = 1}^n \mathcal{R}_{v_i, K_i}
\\
& = & I_{\mathcal{D}_J} \cap {\bigcap}_{j = 1}^q \mathcal{R}_{w_j, L_j}
\\
& = & I_{H' \cap \bigcap_{j = 1}^q ((E^0 \backslash R(w_j)) \sqcup \{ w_j \} )} \cap {\bigcap}_{j = 1}^q \mathcal{R}_{w_j, L_j} = J.
\end{IEEEeqnarray*} 
This completes the proof.
\end{proof}

By Lemma \ref{lemma: Non Gauge Invariant Ideal}, \eqref{equation: Ideal Lattice 2} and \eqref{equation: Ideal Lattice 1} define an injective map $\mathcal{I}: \Lat(C^*(E)) \rightarrow \Lat(\Cu_E)$. We will now show that $\mathcal{I}$ is a bijection. If $f \in \Cu_E$ then define $f_\partial$ to be the set
\begin{equation} \label{equation: Function Partial}
f_\partial = \big\{ \check{v} \in \supp_0(f)_t \; \vert \; \text{$c(\check{v}) = 1$ and $\supp(f(\check{v})) \not \in \{\emptyset, \mathbb{T} \}$} \big\}. 
\end{equation}
If $f_{\partial}$ is non-empty, then the set $\{ w \in E^0 \; \vert \; \check{w} \in f_\partial \}$ can be thought of as the vertices which induce the non-gauge-invariant part of a non-gauge-invariant ideal in \eqref{equation: Non Gauge Invariant Form} --- this is essentially the content of Lemma \ref{lemma: Largest GII Below} below. We remind the reader that a hereditary subset $H$ of $S(E)^0$ induces a hereditary subset $H^\dagger$ of $E^0$, defined in \eqref{equation: Dagger}. 

\begin{lemma} \label{lemma: Largest GII Below}
Let $f \in \Cu_E$. As subsets of $E^0$, 
\begin{equation*}
\supp_0(f)^\dagger \backslash \{ w \in E^0 \; \vert \; \check{w} \in f_\partial \} = \supp_0(f)^\dagger \cap {\bigcap}_{\check{w} \in f_\partial} (E^0 \backslash R(w)). 
\end{equation*}
\end{lemma}

\begin{proof}
First, assume that $x \in \supp_0(f)^\dagger \cap {\bigcap}_{\check{w} \in f_\partial} (E^0 \backslash R(w))$. Since $x \in R(x)$ then $\check{x} \not \in f_\partial$ and
\begin{equation*}
\supp_0(f)^\dagger \cap {\bigcap}_{\check{w} \in f_\partial} (E^0 \backslash R(w)) \subseteq \supp_0(f)^\dagger \backslash \{ w \in E^0 \; \vert \; \check{w} \in f_\partial \}. 
\end{equation*}
To see that the reverse inclusion holds, assume that $$x \not \in \supp_0(f)^\dagger \cap {\bigcap}_{\check{w} \in f_\partial} (E^0 \backslash R(w)).$$ Then either $x \not \in \supp_0(f)^\dagger$ or there exists $w \in E^0$ such that $\check{w} \in f_\partial$ and $x \in R(w)$. Observe that if $x \in \supp_0(f)^\dagger$ then there exists $y \in E^0$ such that $x \geq y$ and $\check{y} \in \supp_0(f)$. Since $x \in R(w)$ then $\check{w} \geq \check{x} \geq \check{y}$ in $S(E)$. But by definition of $f_\partial$, $\check{w}$ is a terminal vertex for $\supp_0(f)$. So $\check{w} = \check{y}$ and $\check{x} = \check{w} \in f_\partial$. Consequently, we have proved that either $x \not \in \supp_0(f)^\dagger$ or $x \in \{ w \in E^0 \; \vert \; \check{w} \in f_\partial \}$.
\end{proof}

\begin{theorem} \label{theorem: Ideal Lattice Iso}
The map $\mathcal{I}: \Lat(C^*(E)) \rightarrow \Lat(\Cu_E)$ defined by \eqref{equation: Ideal Lattice 2} is bijective. 
\end{theorem}

\begin{proof}
By Lemma \ref{lemma: Non Gauge Invariant Ideal}, $\mathcal{I}$ is injective. To see that $\mathcal{I}$ is surjective, assume that $S$ is an ideal of $\Cu_E$. First, observe that $\Cu_E$ is a countably based Cu-semigroup (see Definition \ref{definition: Countably Based}). Indeed, if $\check{v} \in S(E)^0$ is such that $c(\check{v}) \neq \infty$ then we can use countable sup-dense subsets of $\overline{\mathbb{N}}$ or $\Lsc(\mathbb{T}, \overline{\mathbb{N}})$ to write the function $\theta_{\check{v}, f(\check{v})}$ as the supremum of a rapidly increasing sequence of functions formed from these sup-dense subsets. The claim that $\Cu_E$ is countably based then follows from Lemma \ref{lemma: Singleton Decomp} and \ref{axiom: O4}. \newline
\phantom{} \newline
Now by \cite[Proposition 5.1.7]{APT18}, there exists $f \in \Cu_E$ such that $S = \Idl(f)$. If $f_\partial = \emptyset$ then by \eqref{equation: Hereditary Skeleton} and \eqref{equation: Ideal Lattice 1}, 
\begin{equation*}
\mathcal{I}(I_{\supp_0(f)^\dagger}) = \big\{ g \in \Cu_E \; \vert \; \supp_0(g) \subseteq \supp_0(f) \big\} = \Idl(f). 
\end{equation*}
If $f_\partial \neq \emptyset$ then define the non-gauge-invariant ideal
\begin{equation*}
I_f = I_{\supp_0(f)^\dagger} \cap {\bigcap}_{\check{v} \in f_\partial} \mathcal{R}_{v, \mathbb{T} \backslash \supp(f(\check{v}))}. 
\end{equation*}
Since $f_\partial$ was assumed to be non-empty, $I_f$ satisfies property \ref{property: I2}. By definition, $I_f$ automatically satisfies property \ref{property: I1}. To see that $I_f$ satisfies property \ref{property: I3}, we need to show that if $v \in E^0$ is such that $\check{v} \in f_\partial$ then 
\begin{equation*}
v \in \supp_0(f)^\dagger \cap {\bigcap}_{\check{w} \in f_\partial, \; w \neq v } (E^0 \backslash R(w)). 
\end{equation*}
By definition of $f_\partial$, $v \in \supp_0(f)^{\dagger}$. Furthermore, if $w \in E^0$ is such that $\check{w} \in f_{\partial}$ and $v \in R(w)$ then $\check{w} \geq \check{v}$ and $\check{w} = \check{v}$ in $S(E)$ because $\check{w}$ is a terminal vertex for $\supp_0(f)$. Since $w$ supports a single loop, we in fact have $v = w$ in $E$. So by the contrapositive, if $w \neq v$ then $v \in E^0 \backslash R(w)$. Consequently, $I_f$ satisfies property \ref{property: I3}.\newline
\phantom{} \newline
Recalling the definition of $\mathcal{D}_{I_f}$ from \eqref{equation: Hereditary DJ}, we have by Lemma \ref{lemma: Largest GII Below} that 
\begin{equation*}
\mathcal{D}_{I_f} = \Big( \supp_0(f)^\dagger \cap {\bigcap}_{\check{w} \in f_\partial} (E^0 \backslash R(w)) \Big) \sqcup \{ w \in E^0 \; \vert \; \check{w} \in f_\partial \} = \supp_0(f)^\dagger. 
\end{equation*}
So by \eqref{equation: Hereditary Skeleton} and \eqref{equation: Ideal Lattice 2}, $\langle \mathcal{D}_{I_f} \rangle = \supp_0(f)$ as hereditary subsets of $S(E)^0$ and
{ \small \begin{equation*}
\mathcal{I}(I_f) = 
\left\{ g \in \Cu_E \; \Big\vert \; 
\begin{matrix}
\supp_0(g) \subseteq \supp_0(f) \text{ and for $i \in \{1, 2, \dots, n \}$,} \\
\supp(g(\check{v}_i)) \subseteq \supp(f(\check{v}_i))
\end{matrix}
\right\} = \Idl(f).
\end{equation*}}
Hence, $\mathcal{I}: \Lat(C^*(E)) \rightarrow \Lat(\Cu_E)$ is bijective. 
\end{proof}

\subsection{An exact sequence involving $\Cu_E$}
\label{subsec: Exact Sequence}

Let $E$ be a finite graph satisfying the standing assumptions and $H$ be a hereditary subset of $E^0$. Recall from Theorem \ref{theorem: CRS} that the extension
\begin{equation*} 
\begin{tikzcd}
0 \arrow[r] & I_H \arrow[r, "\iota"] & C^*(E) \arrow[r, "\pi"] & C^*(E \backslash H) \arrow[r] & 0
\end{tikzcd}
\end{equation*}
induces a short exact sequence of positively ordered monoids { \small
\begin{equation} \label{equation: Cu(C*(E)) SES}
\begin{tikzcd}
0 \arrow[r] & \Cu(C^*(H)) \arrow[r, "\Cu(\iota)"] & \Cu(C^*(E)) \arrow[r, "\Cu(\pi)"] & \Cu(C^*(E \backslash H)) \arrow[r] & 0
\end{tikzcd}
\end{equation}}
such that as Cu-semigroups, $\Cu(C^*(E \backslash H))$ is isomorphic to the quotient Cu-semigroup $\Cu(C^*(E))/\Cu(C^*(H))$. Note that by the isomorphism in \eqref{equation: Full Corner Iso}, $\Cu(I_H) \cong \Cu(C^*(H))$. In this subsection, we will demonstrate that $\Cu_E$ fits into a similar short exact sequence. Firstly, we have the inclusion map
\begin{equation} \label{equation: Inclusion}
\begin{matrix}
i_H: & \Cu_H & \hookrightarrow & \Cu_E \\
& f & \mapsto & f
\end{matrix}
\end{equation}
which is easily seen to be additive and an order-embedding. It is also readily checked that the image of $i_H$ is the ideal 
\begin{equation} \label{equation: Ideal In Terms Of Support}
\mathcal{I}(I_H) = \Idl \Big({\sum}_{\check{v} \in \langle H \rangle_t} \theta_{\check{v}, \delta_{\check{v}}} \Big) = \big\{ f \in \Cu_E \; \vert \; \supp_0(f) \subseteq \langle H \rangle \big\}. 
\end{equation}
Therefore, $i_H$ implements the Cu-isomorphism
\begin{equation*}
\Cu_H \cong \big\{ f \in \Cu_E \; \vert \; \supp_0(f) \subseteq \langle H \rangle \big\}.
\end{equation*}

\begin{remark} \label{remark: POM Iso Between Cu Semigroups}
Let $S$ and $T$ be Cu-semigroups and $\phi: S \rightarrow T$ be an isomorphism of positively ordered monoids. Since suprema of increasing sequences and the way-below relation are defined in terms of the order, $\phi$ automatically preserves both way-below and suprema of increasing sequences. So $\phi$ is a Cu-isomorphism. This observation was noted in the proof of \cite[Lemma 7.1.12]{APT18}, and justifies why $i_H$ is a Cu-isomorphism.  
\end{remark}

Our next task is to produce a Cu-morphism from $\Cu_E$ to $\Cu_{E \backslash H}$.

\begin{lemma} \label{lemma: Cu E Minus H}
There exists a surjective Cu-morphism $\rho_H: \Cu_E \rightarrow \Cu_{E \backslash H}$ such that $$\ker \rho_H = \big\{ f \in \Cu_E \; \vert \; \supp_0(f) \subseteq \langle H \rangle \big\} \cong \Cu_H.$$
\end{lemma}

\begin{proof}
Assume that $f \in \Cu_E$. By Lemma \ref{lemma: Singleton Decomp}, 
\begin{equation*}
f = {\sum}_{\check{v} \in \supp_0(f)_t} \theta_{\check{v}, f(\check{v})}. 
\end{equation*}
Recall the definition of $\langle E \backslash H \rangle$ from \eqref{equation: Hereditary Skeleton}. Define $\rho_H(f)$ by
\begin{equation} \label{equation: Quotient}
\rho_H(f) = {\sum}_{\check{v} \in (\supp_0(f) \cap \langle E \backslash H \rangle)_t} \theta_{\check{v}, f(\check{v})}. 
\end{equation}
By construction, $\rho_H(f) \in \Cu_{E \backslash H}$. To see that $\rho_H$ is additive, assume that $f, g \in \Cu_E$. By Lemma \ref{lemma: Hereditary Terminal},
\begin{equation*}
\supp_0(\rho_H(f + g)) = (\supp_0(f) \cup \supp_0(g)) \cap \langle E \backslash H \rangle = \supp_0(\rho_H(f) + \rho_H(g)).
\end{equation*}
Now suppose that $\check{v} \in ((\supp_0(f) \cup \supp_0(g)) \cap \langle E \backslash H \rangle)_t$ and $c(\check{v}) \neq \infty$. Observe that if $\check{v} \in \supp_0(f) \cap \langle E \backslash H \rangle$ then by definition of a terminal vertex, $\check{v} \in (\supp_0(f) \cap \langle E \backslash H \rangle)_t$. The same statement holds when $\supp_0(f)$ is replaced with $\supp_0(g)$. In the resulting three cases, we have by \eqref{equation: Quotient} that
\begin{equation*}
\rho_H(f + g)(\check{v}) = f(\check{v}) + g(\check{v}) = \rho_H(f)(\check{v}) + \rho_H(g)(\check{v}).
\end{equation*}
So by Lemma \ref{lemma: Order Alternative}, $\rho_H(f + g) = \rho_H(f) + \rho_H(g)$. Using Lemma \ref{lemma: Order Alternative} once more, it is straightforward to check that $\rho_H$ is order preserving. The map $\rho_H$ is surjective because if $g \in \Cu_{E \backslash H}$ then 
\begin{equation*}
\rho_H \Big( {\sum}_{\check{v} \in \supp_0(g)_t} \theta_{\check{v}, g(\check{v})} \Big) = g. 
\end{equation*}
Moreover, $\ker \rho_H$ is the set of all functions in $\Cu_E$ with vertex support contained in $\langle H \rangle$ because $\langle H \rangle \cap \langle E \backslash H \rangle = \emptyset$. \newline
\phantom{} \newline
It remains to show that $\rho_H$ is a Cu-morphism. The claim that $\rho_H$ is way-below preserving is proved in a similar fashion to showing that $\rho_H$ is order preserving, using Lemma \ref{lemma: Way Below} in lieu of Lemma \ref{lemma: Order Alternative}. To see that $\rho_H$ preserves suprema of increasing sequences, let $\{ f_n \}_{n > 0}$ be an increasing sequence in $\Cu_E$ with supremum $f$. Since $\rho_H$ is order preserving, $\{ \rho_H(f_n) \}_{n > 0}$ is an increasing sequence in $\Cu_{E \backslash H}$. \newline
\phantom{} \newline
We claim that $\rho_H(f) = \sup_{n > 0} \rho_H(f_n)$. To this end, assume that $g \in \Cu_{E \backslash H}$ is such that if $n \in \{1, 2, \dots \}$ then $\rho_H(f_n) \leq g$. By viewing $\supp_0(g)_t$ as a subset of $S(E)^0$, we define $g'$ be the function
\begin{equation*}
g' = {\sum}_{\check{v} \in \supp_0(g)_t} \theta_{\check{v}, g(\check{v})} + \infty \cdot {\sum}_{\check{w} \in \langle H \rangle_t} \theta_{\check{w}, \delta_{\check{w}}} \in \Cu_E.
\end{equation*}
By Lemma \ref{lemma: Order Alternative} and \eqref{equation: Quotient}, $f_n \leq g'$ for $n \in \{1, 2, \dots \}$ and $\rho_H(g') = g$. So $f \leq g'$ and $\rho_H(f) \leq \rho_H(g') = g$. Hence, $\rho_H$ preserves suprema of increasing sequences, completing the proof. 
\end{proof}

By Lemma \ref{lemma: Cu E Minus H}, we obtain a short exact sequence of positively ordered monoids
\begin{equation} \label{equation: Cu E SES}
\begin{tikzcd}
0 \arrow[r] & \Cu_H \arrow[r, "i_H"] & \Cu_E \arrow[r, "\rho_H"] & \Cu_{E \backslash H} \arrow[r] & 0.
\end{tikzcd}
\end{equation}
Analogously to the short exact sequence in \eqref{equation: Cu(C*(E)) SES}, we will now show that the Cu-semigroups $\Cu_E / \Cu_H$ and $\Cu_{E \backslash H}$ are isomorphic. 

\begin{proposition} \label{proposition: First Iso Thm Cu E}
The Cu-semigroups $\Cu_E / \Cu_H$ and $\Cu_{E \backslash H}$ are isomorphic. 
\end{proposition}

\begin{proof}
By Remark \ref{remark: POM Iso Between Cu Semigroups}, it suffices to show that $\Cu_E / \Cu_H$ and $\Cu_{E \backslash H}$ are isomorphic as positively ordered monoids. By Lemma \ref{lemma: Cu E Minus H}, $\rho_H: \Cu_E \rightarrow \Cu_{E \backslash H}$ is a surjective Cu-morphism with $\ker \rho_H = \Cu_H$. Define $\widetilde{\rho_H}$ to be the map
\begin{equation*}
\begin{matrix}
\widetilde{\rho_H}: & \Cu_E/\Cu_H & \rightarrow & \Cu_{E \backslash H} \\
& [f]_{\Cu_H} & \mapsto & \rho_H(f). 
\end{matrix}
\end{equation*}
The map $\widetilde{\rho_H}$ is well-defined because $\ker \rho_H = \Cu_H$. It is also additive and order preserving because $\rho_H$ is. Now to see that $\widetilde{\rho_H}$ is an order-embedding, assume that $[f]_{\Cu_H}, [g]_{\Cu_H} \in \Cu_E / \Cu_H$ are such that $\rho_H(f) \leq \rho_H(g)$ in $\Cu_{E \backslash H}$. Define
\begin{equation*}
g' = g + \infty \cdot {\sum}_{\check{v} \in \langle H \rangle_t} \theta_{\check{v}, \delta_{\check{v}}}. 
\end{equation*}
By combining Lemma \ref{lemma: Order Alternative} with the inequality $\rho_H(f) \leq \rho_H(g)$, $f \leq g'$ in $\Cu_E$. Recalling the partial order on $\Cu_E/\Cu_H$ from Definition \ref{definition: Quotient Cu Semigroup}, we have that $[f]_{\Cu_H} \leq [g]_{\Cu_H}$ and thus, $\widetilde{\rho_H}$ is an order-embedding.
\end{proof}

\section{Reframing the Cuntz-Krieger relations using pure largeness}
\label{sec: General Cuntz Krieger}

The main goal of this section is to prove Theorem \ref{introtheorem: General Cuntz Krieger}. We remind the reader that Theorem \ref{introtheorem: General Cuntz Krieger} is an important special case of Theorem \ref{introtheorem: Cu E Iso}. For the graphs it applies to, Theorem \ref{introtheorem: General Cuntz Krieger} generalises the behaviour of Cuntz classes of vertex projections imposed by the Cuntz-Krieger relations to Cuntz classes of positive elements in the stabilisation of the graph C*-algebra. By combining Theorem \ref{introtheorem: General Cuntz Krieger} with the isomorphism in \eqref{equation: Full Corner Iso}, the various relations one obtains from the Cuntz semigroup in Theorem \ref{introtheorem: General Cuntz Krieger} can be transported to the Cuntz semigroups of graph C*-algebras coming from arbitrary finite graphs satisfying the standing assumptions. This is especially important to showing that the map $\Delta_E: \Cu_E \rightarrow \Cu(C^*(E))$ defined in Definition \ref{definition: Delta E} is a Cu-morphism; see also Lemma \ref{lemma: Same Open Subset Upstairs} and the beginning of \Cref{sec: Delta E} for more details. \newline
\phantom{} \newline
To make the statement of Theorem \ref{introtheorem: General Cuntz Krieger} more precise, our first task is to formalise the notion of the stack of two Cu-semigroups. This is done in \cref{subsec: Stack}. The proof of Theorem \ref{introtheorem: General Cuntz Krieger} is divided into two cases. The easier of the two cases is when the distinguished vertex $v$ in Theorem \ref{introtheorem: General Cuntz Krieger} supports multiple cycles. The proof in this case takes advantage of the simple purely infinite quotient, and is provided in \cref{subsec: Kirchberg Quotient}. The more difficult case where the distinguished vertex supports a single loop is done in \cref{subsec: Circle Quotient}.  

\subsection{The stack of two Cu-semigroups}
\label{subsec: Stack}

Given two Cu-semigroups $S$ and $T$, one can construct from it a Cu-semigroup $S \stack T$ whose order behaves in exactly the same way as the order of the Cuntz semigroup of a purely large extension (see Corollary \ref{corollary: Purely Large Order Stabilised}). The stack of two Cu-semigroups is the key construction we need to prove Theorem \ref{introtheorem: General Cuntz Krieger}.

\begin{definition} \label{definition: Cuntz Stack}
Let $(S, \leq_S, +_S)$ and $(T, \leq_T, +_T)$ be Cu-semigroups and suppose that $S$ has largest element $\infty_S$. The \textbf{stack} of $S$ below $T$, denoted by $S \stack T$, is the Cu-semigroup defined by the following data:
\begin{enumerate}
\item The underlying set is the disjoint union $S \sqcup (T \backslash \{ 0_T \})$ where $0_T$ is the zero element of $T$. 
\item Let $s_1, s_2 \in S$ and $t_1, t_2 \in T \backslash \{ 0_T \}$. The order on $S \stack T$ is given by $s_1 \leq s_2$ if and only if $s_1 \leq_S s_2$, $t_1 \leq t_2$ if and only if $t_1 \leq_T t_2$ and $s_1 \leq t_1$. 
\item Let $s_1, s_2 \in S$ and $t_1, t_2 \in T \backslash \{ 0_T \}$. Addition on $S \stack T$ is given by $s_1 + s_2 = s_1 +_S s_2$, $t_1 + t_2 = t_1 +_T t_2$ and $s_1 + t_1 = t_1$. 
\item The way-below relation of $S \stack T$ behaves similarly to the order.
\end{enumerate}
\end{definition}

It is straightforward to check that the stack of two Cu-semigroups is again a Cu-semigroup and that $S$ is an ideal of $S \stack T$. 

\begin{proposition} \label{proposition: Stack Characterisation}
Let
\begin{equation*}
\begin{tikzcd}
0 \arrow[r] & I \arrow[r] & A \arrow[r] & B \arrow[r] & 0
\end{tikzcd}
\end{equation*} 
be an extension of C*-algebras. Suppose that $I$ is separable so that $\Cu(I)$ has largest element $\infty_I$. Then $\Cu(A) \cong \Cu(I) \stack \Cu(B)$ if and only if for $x \in \Cu(A) \backslash \Cu(I)$, $x + \infty_I = x$. 
\end{proposition}

\begin{proof}
The ``only if" direction follows from Definition \ref{definition: Cuntz Stack}. Assume that $x \in \Cu(A) \backslash \Cu(I)$ and let $\pi: A \rightarrow B$ be the quotient map. Observe that if $y \in \Cu(A)$ and $\Cu(\pi)(y) \leq \Cu(\pi)(x)$ then by Theorem \ref{theorem: CRS}, $y \leq x + \infty_I = x$ in $\Cu(A)$. Hence, any non-zero element of $\Cu(B)$ lifts uniquely to an element of $\Cu(A) \backslash \Cu(I)$ and the map $\phi: \Cu(A) \rightarrow \Cu(I) \stack \Cu(B)$ defined by
\begin{equation*}
\phi(x) = 
\begin{cases}
x, \text{ viewed as an element of $\Cu(I)$ if $x \leq \infty_I$,} \\
\Cu(\pi)(x), \text{ viewed as an element of $\Cu(B)$ otherwise.}
\end{cases}
\end{equation*}
is an isomorphism of positively ordered monoids and thus, a Cu-isomorphism by Remark \ref{remark: POM Iso Between Cu Semigroups}.
\end{proof}

\begin{remark} \label{remark: Purely Large Absorption}
By Corollary \ref{corollary: Purely Large Order Stabilised}, if the Cuntz semigroup of an extension decomposes as the stack, then the extension must be purely large. However, Cuntz semigroups of C*-algebras which are purely large extensions do not always decompose as the stack of the Cuntz semigroup of the ideal below the Cuntz semigroup of the quotient. \newline
\phantom{} \newline
To illustrate this point, let $H$ be an infinite-dimensional separable Hilbert space. The C*-algebra $B(H)$ is a purely large extension of the Calkin algebra $Q(H)$ by $K(H)$ by Theorem \ref{theorem: Kucerovsky Ng}. Moreover by \cite[Theorem 6.4]{ABP11}, $$\Cu(B(H)) \cong \Cu(K(H)) \stack \Cu(Q(H)) \cong \overline{\mathbb{N}} \stack \{0, \infty \}.$$ Now let $\tilde{K}$ denote the minimal unitisation of $K(H)$, which is a purely large extension of $\mathbb{C}$ by $K$. However, $$\Cu(\tilde{K}) \not \cong \overline{\mathbb{N}} \stack \overline{\mathbb{N}}.$$ To see this, note that $\tilde{K}$ has stable rank one and thus, $\Cu(\tilde{K})$ is weakly cancellative by \cite[Theorem 4.3]{RW10}. On the other hand, by Definition \ref{definition: Cuntz Stack}, $\overline{\mathbb{N}} \stack \overline{\mathbb{N}}$ is not weakly cancellative. 
\end{remark}

By Remark \ref{remark: Purely Large Absorption}, being weakly cancellative is generally an obstruction for the Cuntz semigroup of a purely large extension with a stably finite quotient to decompose as the stack. The following proposition, communicated to the author by Vilalta, shows via an application of axiom (O5) the fundamental issue --- that Cuntz classes of the extension which descend to Cuntz classes of projections in the stably finite quotient may not have the absorption behaviour required to avoid being weakly cancellative and hence, for the Cuntz semigroup to decompose as the stack. The author is grateful to Eduard Vilalta for allowing Proposition \ref{proposition: Purely Large Axiom (O5)} to be included and used in the paper. 

\begin{proposition}[Vilalta] \label{proposition: Purely Large Axiom (O5)}
Let
\begin{equation*}
\begin{tikzcd}
0 \arrow[r] & I \arrow[r] & A \arrow[r, "\pi"] & B \arrow[r] & 0
\end{tikzcd}
\end{equation*}
be a purely large extension such that $I$ is separable and $B$ is stably finite. Let $\infty_I$ be the largest element in $\Cu(I)$ and $z \in \Cu(A) \backslash \Cu(I)$. Either $z + \infty_I = z$ or there exists a projection $p \in B \otimes K$ such that $\Cu(\pi)(z) = [p]$. 
\end{proposition}

\begin{proof}
Assume that $x \in (A \otimes K)_+ \backslash (I \otimes K)$ and $\epsilon > 0$. Then
\begin{equation*}
[(x - \epsilon)_+] \ll \Big[ \Big(x - \frac{\epsilon}{2} \Big)_+ \Big] \ll [x]. 
\end{equation*}
Applying (O5), there exists $y_\epsilon \in \Cu(A)$ such that 
\begin{equation} \label{equation: Axiom (O5) Application}
[(x - \epsilon)_+] + y_\epsilon  \leq [x] \leq \Big[ \Big(x - \frac{\epsilon}{2} \Big)_+ \Big] + y_\epsilon. 
\end{equation}
Firstly, if there exists $\epsilon > 0$ such that $y_\epsilon \in \Cu(I)$ then by the RHS inequality of \eqref{equation: Axiom (O5) Application},
\begin{equation*}
[\pi(x)] \leq \Big[ \Big(\pi(x) - \frac{\epsilon}{2} \Big)_+ \Big] \ll [\pi(x)]. 
\end{equation*}
So $[\pi(x)] \in \Cu(B)$ is compact. Since $B$ is stably finite then by \cite[Theorem 3.5]{BC09}, there exists a projection $p \in (B \otimes K)_+$ such that $\Cu(\pi)([x]) = [\pi(x)] = [p]$. On the other hand, suppose that for $\epsilon > 0$, $y_\epsilon \in \Cu(A) \backslash \Cu(I)$. By the LHS inequality of \eqref{equation: Axiom (O5) Application} and Corollary \ref{corollary: Purely Large Order Stabilised},
\begin{equation*}
[(x - \epsilon)_+] + \infty_I \leq [(x - \epsilon)_+] + y_\epsilon \leq [x]. 
\end{equation*}
Taking the supremum over all $\epsilon > 0$, we have $[x] + \infty_I = [x]$. 
\end{proof}

\subsection{Proof of Theorem \ref{introtheorem: General Cuntz Krieger} in the Kirchberg quotient case}
\label{subsec: Kirchberg Quotient}

The next two subsections are dedicated to proving Theorem \ref{introtheorem: General Cuntz Krieger}. In this subsection, we will first tackle the case where the distinguished vertex $v$ in Theorem \ref{introtheorem: General Cuntz Krieger} supports multiple cycles; this corresponds to $C^*(E)$ having a Kirchberg quotient. Stepping outside the confines of graph C*-algebras, we will prove a more general result than what we need; namely that the Cuntz semigroup of a C*-algebra which is a purely large extension with a purely infinite quotient decomposes as the stack. The key results needed for this are Lemma \ref{lemma: KR Lemma} and the characterisation of pure largeness in Corollary \ref{corollary: Purely Large Order Stabilised} because we are dealing with arbitrary Cuntz classes, not just Cuntz classes of positive elements in the C*-algebra. 

\begin{lemma} \label{lemma: Not Ideal Properly Infinite}
Let
\begin{equation*}
\begin{tikzcd}
0 \arrow[r] & I \arrow[r] & A \arrow[r] & B \arrow[r] & 0 
\end{tikzcd}
\end{equation*}
be a purely large extension of C*-algebras. Suppose that $B$ is purely infinite. If $z \in \Cu(A) \backslash \Cu(I)$ then $2z = z$ in $\Cu(A)$. 
\end{lemma}

\begin{proof}
Assume that $A$ is a purely large extension of $B$ by $I$ and that $B$ is purely infinite. Assume that $[x] \in \Cu(A) \backslash \Cu(I)$. To show that $x$ is properly infinite, it suffices by \cite[Proposition 3.3 (iii)]{KR00} to show that if $\epsilon > 0$ then there exist mutually orthogonal positive elements $a_1, a_2 \in \overline{x (A \otimes K) x}$ such that for $j \in \{1, 2 \}$, $(x - \epsilon)_+ \precsim a_j$. \newline
\phantom{} \newline
Assume that $\epsilon > 0$ is small enough so that $(x - \epsilon)_+ \not \in (I \otimes K)_+$. By Theorem \ref{theorem: Purely Infinite Properly Infinite} and Lemma \ref{lemma: KR Lemma}, there exist mutually orthogonal positive elements $a_1, a_2, a_3, a_4 \in \overline{x (A \otimes K) x}$ and a positive element $c \in \overline{x (I \otimes K) x}$ such that $a_1, a_2, a_3$ and $a_4$ are Cuntz equivalent to each other and for $j \in \{1, 2, 3, 4 \}$, 
\begin{equation*}
[(x - \epsilon)_+] \leq [a_j \oplus c] = [a_j] + [c]. 
\end{equation*}
Now observe that if $j \in \{1, 2, 3, 4 \}$ then $[a_j] \not \in \Cu(I)$ because $[(x - \epsilon)_+] \not \in \Cu(I)$ and $[c] \in \Cu(I)$. Since $a_j$ and $c$ are positive elements of $A \otimes K$ then by Corollary \ref{corollary: Purely Large Order Stabilised}, 
\begin{equation*}
[(x - \epsilon)_+] \leq [a_1] + [c] \leq [a_1] + [a_2] = [a_1 + a_2].
\end{equation*}
Similarly, $[(x - \epsilon)_+] \leq [a_3 + a_4]$. So $a_1 + a_2$ and $a_3 + a_4$ are mutually orthogonal positive elements in $\overline{x (A \otimes K) x}$ which Cuntz dominate $(x - \epsilon)_+$ and hence, $(x - \delta)_+$ for $\delta \geq \epsilon$. If $\mu < \epsilon$ then $[(x - \mu)_+] \not \in \Cu(I)$ and we may repeat the above argument with $\mu$ in place of $\epsilon$ to obtain two mutually orthogonal positive elements in $\overline{x (A \otimes K) x}$ which Cuntz dominate $(x - \mu)_+$. 
\end{proof}

\begin{theorem} \label{theorem: Purely Large Purely Infinite Quotient}
Let
\begin{equation*}
\begin{tikzcd}
0 \arrow[r] & I \arrow[r] & A \arrow[r] & B \arrow[r] & 0 
\end{tikzcd}
\end{equation*}
be a purely large extension of C*-algebras. Suppose that $I$ is separable and $B$ is purely infinite. Then $$\Cu(A) \cong \Cu(I) \stack \Cu(B).$$
\end{theorem}

\begin{proof}
Assume that $x \in \Cu(A) \backslash \Cu(I)$. By Lemma \ref{lemma: Not Ideal Properly Infinite}, we have that $2x = x$ in $\Cu(A)$. Thus, if $\infty_I$ is the largest element of $\Cu(I)$ then by Corollary \ref{corollary: Purely Large Order Stabilised}, $x + \infty_I \leq 2x = x$. So $x + \infty_I = x$ and we are done by Proposition \ref{proposition: Stack Characterisation}. 
\end{proof}

Using Theorem \ref{theorem: Purely Large Purely Infinite Quotient}, we now prove Theorem \ref{introtheorem: General Cuntz Krieger} in the case where the distinguished vertex $v$ supports multiple cycles.

\begin{proof}[Proof of Theorem \ref{introtheorem: General Cuntz Krieger} when $v$ supports multiple cycles]
Assume that $v \in E^0$ is as in Theorem \ref{introtheorem: General Cuntz Krieger} and supports multiple cycles. It is easy to check that $H$ is hereditary and saturated. By \cite[Theorem 4.1]{BPRS00}, we have the extension
\begin{equation} \label{equation: Typical Graph Extension}
\begin{tikzcd}
0 \arrow[r] & I_H  \arrow[r] & C^*(E) \arrow[r] & C^*(E \backslash H) \arrow[r] & 0.
\end{tikzcd}
\end{equation}
Since every vertex in $E^0$ admits a path to $v$, 
\begin{equation*}
(E \backslash H)^0 = E^0 \backslash H = \{ x \in E^0 \; \vert \; \text{$x \geq v$ and $v \geq x$} \}. 
\end{equation*}
By assumption, $v \in E^0 \backslash H$ supports more than one cycle. So $E \backslash H$ cannot consist of just a single cycle. Subsequently by \cite[Proposition 4.2]{Rae05} and \cite[Remark 5.6]{BPRS00}, $C^*(E \backslash H)$ is simple purely infinite.  \newline
\phantom{} \newline
Now since $C^*(E)$ is a unital extension with a simple quotient, it is full. To see that it is purely large, first note that by \cite[Lemma 1.6]{DHS03}, there exists a row-finite graph $E_H$ such that $I_H \cong C^*(E_H)$. By \cite[Theorem 3.2]{Tom04} and the structure of $E_H$, $I_H \cong C^*(E_H)$ is stable. So by \eqref{equation: Full Corner Iso}, $I_H \cong C^*(H) \otimes K$ and thus, $I_H$ is $\sigma$-unital and has the corona factorisation property by Proposition \ref{proposition: Graph CFP}. By Theorem \ref{theorem: Kucerovsky Ng}, $C^*(E)$ is a purely large extension. Theorem \ref{theorem: Purely Large Purely Infinite Quotient} finishes the proof. 
\end{proof}

\subsection{Proof of Theorem \ref{introtheorem: General Cuntz Krieger} in the $C(\mathbb{T})$ quotient case}
\label{subsec: Circle Quotient}

We now turn to proving Theorem \ref{introtheorem: General Cuntz Krieger} in the case where $v$ supports a single loop; this corresponds to $C^*(E)$ having a $C(\mathbb{T})$ quotient. Broadly speaking, the proof divides into two steps. First, we prove Lemma \ref{lemma: Purely Large Circle Quotient}, which says that in the situation of Theorem \ref{introtheorem: General Cuntz Krieger} with $v$ supporting a single loop, $\Cu(C^*(E))$ decomposes as the stack if and only if $C^*(E)$ is a purely large extension with $C(\mathbb{T})$ quotient. 

\begin{lemma} \label{lemma: Purely Large Circle Quotient}
Let $E$ be a finite graph and suppose that there exists $v \in E^0$ such that $v$ supports a single loop $g$ and if $w \in E^0$ then $w \geq v$. Let $H = \{ x \in E^0 \; \vert \; v \not \geq x \} = E^0 \backslash \{ v \}$. Then $C^*(E)$ is a purely large extension of $C(\mathbb{T})$ by $I_H \cong C^*(H) \otimes K$ if and only if $$\Cu(C^*(E)) \cong \Cu(C^*(H)) \stack \Lsc(\mathbb{T}, \overline{\mathbb{N}}).$$ 
\end{lemma}
The second step is to remove the condition that $C^*(E)$ is a purely large extension from Lemma \ref{lemma: Purely Large Circle Quotient}. By Lemma \ref{lemma: Purely Large Circle Quotient} and Proposition \ref{proposition: QDS Forcing}, it suffices to prove that the graph C*-algebra of a quantum double suspension graph (see Definition \ref{definition: QDS}) is a purely large extension of $C(\mathbb{T})$. \newline
\phantom{} \newline
To prove Lemma \ref{lemma: Purely Large Circle Quotient}, we need to show that if $x \in \Cu(C^*(E)) \backslash \Cu(I_H)$ and $y \in \Cu(I_H)$ then $x$ absorbs $y$. Fortunately by Proposition \ref{proposition: Purely Large Axiom (O5)}, we only need to prove this for elements in $\Cu(C^*(E)) \backslash \Cu(I_H)$ which quotient to a non-zero, finite-valued constant function in $\Lsc(\mathbb{T}, \overline{\mathbb{N}})$. In this simplified situation, we will see that the Cuntz-Krieger relations supply the missing piece needed for these Cuntz classes to have the desired absorption behaviour. \newline
\phantom{} \newline
To use the Cuntz-Krieger relations, we first consider positive elements of $C^*(E) \otimes M_n(\mathbb{C})$ not lying in $I_H \otimes M_n(\mathbb{C})$ whose entries lie in the set 
\begin{equation} \label{equation: Leavitt Path}
\mathcal{L}(E) = \Span \{ s_\mu s_\nu^* \; \vert \; \mu, \nu \in E^*, \; s(\mu) = s(\nu) \} \subseteq C^*(E). 
\end{equation}
Recall from the construction of a graph C*-algebra (see for instance, \cite[Proposition 1.20]{Rae05}) that the *-subalgebra $\mathcal{L}(E)$ is norm-dense in $C^*(E)$. If $\pi: C^*(E) \rightarrow C(\mathbb{T})$ is the quotient map and $x \in \mathcal{L}(E)$ then $\pi(s_g)$ is unitary and by the Cuntz-Krieger relations,
\begin{equation} \label{equation: Quotient Form}
\pi(x) \in \Span \big\{ \pi(p_v), \; \pi(s_g)^k, \; (\pi(s_g)^*)^{\ell} \; \big\vert \; k, \ell > 0 \big\}.
\end{equation}
Hence, we define $c_x \in C(\mathbb{T})$ to be the Laurent polynomial satisfying $\pi(x) = c_x(\pi(s_g))$ and $c_x(s_g)$ to be the element of $C^*(E)$ obtained by replacing every instance of $z$ and $z^{-1}$ in $c_x(z)$ with $s_g$ and $s_g^*$ respectively. \newline
\phantom{} \newline
Moreover, the Cuntz-Krieger relations also say that if $j, \ell \geq 0$ and $k \geq \max \{j, \ell \}$ then
\begin{equation} \label{equation: Paths In W Cutdown}
(s_g^*)^k \big( s_g^j (s_g^*)^\ell \big) s_g^k = 
\begin{cases}
(s_g^*)^{\ell - j}, \text{ if $\ell > j$,} \\
s_g^{j - \ell}, \text{ if $j > \ell$,} \\
p_v, \text{ if $\ell = j$.}
\end{cases}
\end{equation}
Lemma \ref{lemma: Positive I Arbitrarily Small} and Proposition \ref{proposition: Sg Cutdown} detail two specific consequences of the Cuntz-Krieger relations we will need in the proof of Lemma \ref{lemma: Purely Large Circle Quotient}. 

\begin{lemma} \label{lemma: Positive I Arbitrarily Small}
Let $y \in (I_H)_+$ and $\epsilon > 0$. Then there exists $k \geq 0$ such that $\lVert (s_g^*)^k y s_g^k \rVert \leq \epsilon$.
\end{lemma}

\begin{proof}
By density, there exists a positive element
\begin{equation*}
y' \in \Span \big\{ s_\mu s_\nu^* \; \vert \; \mu, \nu \in E^*, \; s(\mu) = s(\nu) \in H \big\}
\end{equation*}
such that $\lVert y - y' \rVert < \epsilon$. By the Cuntz-Krieger relations for $E$, there exists $k \geq 0$ large enough such that $(s_g^*)^k y' s_g^k = 0$. So
\begin{equation*}
\big\lVert (s_g^*)^k y s_g^k \big\rVert = \big\Vert (s_g^*)^k y s_g^k - (s_g^*)^k y' s_g^k \big\rVert \leq \lVert s_g^k \rVert^2 \lVert y - y' \rVert < \epsilon. \qedhere
\end{equation*}
\end{proof}

\begin{proposition} \label{proposition: Sg Cutdown}
Let $x \in \Span \{ s_\mu s_\nu^* \; \vert \; \mu, \nu \in E^*, \; s(\mu) = s(\nu) = v \}$. Then there exists $k \geq 0$ such that $(s_g^*)^k x s_g^k = c_x(s_g)$.
\end{proposition}

\begin{proof}
Let 
\begin{equation*}
x = {\sum}_{i = 1}^n \lambda_i s_{\mu_i} s_{\nu_i}^* \in \Span \big\{ s_\mu s_\nu^* \; \big\vert \; \mu, \nu \in E^*, \; s(\mu) = s(\nu) = v \big\}. 
\end{equation*}
If $k = \max \{ \ell(\mu_i), \ell(\nu_i) \; \vert \; i \in \{1, 2, \dots, n \} \}$ then by \eqref{equation: Paths In W Cutdown}, 
\begin{equation*}
(s_g^*)^k x s_g^k \in \Span \big\{ p_v, s_g^m, (s_g^*)^n \; \big\vert \; m, n > 0 \big\}. 
\end{equation*}
Moreover by \eqref{equation: Quotient Form},
\begin{equation*}
\pi \big((s_g^*)^k x s_g^k \big) = (\pi(s_g)^*)^k \pi(x) \pi(s_g)^k = \pi(x) = c_x(\pi(s_g)) = \pi(c_x(s_g)). 
\end{equation*}
So there exists $i \in I_H$ such that $(s_g^*)^k x s_g^k - c_x(s_g) = i$. Now if $\epsilon > 0$ then by Lemma \ref{lemma: Positive I Arbitrarily Small}, there exists $p \in \{1, 2, \dots \}$ such that $\lVert (s_g^*)^p i s_g^p \rVert < \epsilon$. Since $$(s_g^*)^k x s_g^k - c_x(s_g)  \in \Span \{ p_v, s_g^m, (s_g^*)^n \; \vert \; m, n > 0 \}$$ then by \eqref{equation: Paths In W Cutdown},
\begin{equation*}
(s_g^*)^p \big( (s_g^*)^k x s_g^k - c_x(s_g) \big) s_g^p = (s_g^*)^k x s_g^k - c_x(s_g)
\end{equation*}
and $\lVert (s_g^*)^k x s_g^k - c_x(s_g) \rVert = \lVert (s_g^*)^p i s_g^p \rVert < \epsilon$. Since $\epsilon > 0$ was arbitrary then $(s_g^*)^k x s_g^k - c_x(s_g) = 0$ as required. 
\end{proof}

With the preliminary results Lemma \ref{lemma: Positive I Arbitrarily Small} and Proposition \ref{proposition: Sg Cutdown} ready, we will now prove Lemma \ref{lemma: Span Matrix Cutdown Quotients To Projection}, the technical lemma needed to prove Lemma \ref{lemma: Purely Large Circle Quotient} by induction. 

\begin{lemma} \label{lemma: Span Matrix Cutdown Quotients To Projection}
Let $n \geq 1$ and $x = (x_{ij})_{i, j = 1}^n \in (C^*(E) \otimes M_n(\mathbb{C}))_+$ be such that $x_{ij} \in \mathcal{L}(E)$ for $i, j \in \{1, 2, \dots, n \}$. Let $m > 0$ and suppose that $[\pi((x - \frac{1}{m})_+)] = n \cdot [\pi(p_v)]$. Then $$\Big[ \Big(x - \frac{1}{m} \Big)_+ \Big] = n \cdot [p_v]$$ in $\Cu(C^*(E))$ and $[(x - \frac{1}{m})_+] + \infty \cdot [1_{C^*(H)}] = [(x - \frac{1}{m})_+]$. 
\end{lemma}

\begin{proof}
Let $1_n \in C(\mathbb{T}) \otimes M_n(\mathbb{C})$ be the identity matrix. Since the entries of $x$ lie in $\mathcal{L}(E)$, we have $\pi(x) = \big(c_{x_{ij}}(\pi(s_g)) \big)_{i, j = 1}^n$ by \eqref{equation: Quotient Form}. Diagonalising the matrix $(c_{x_{ij}})_{i, j = 1}^n \in M_n(C(\mathbb{T}))_+$ pointwise, we see that if $z \in \mathbb{T}$ then there exist a unitary $u_z \in M_n(\mathbb{C})$ and eigenvalues $\lambda_1(z), \lambda_2(z), \dots, \lambda_n(z) \geq 0$ such that
\begin{equation*}
\big(c_{x_{ij}}(z) \big)_{i, j = 1}^n = u_z^* 
\begin{pmatrix}
\lambda_1(z) & & & \\
& \lambda_2(z) & & \\
& & \ddots & \\
& & & \lambda_n(z)
\end{pmatrix} u_z. 
\end{equation*}
If we cut down the positive matrix $\big(c_{x_{ij}}(z) \big)_{i, j}$ by $\frac{1}{m}$, the resulting matrix still has full rank because $[(\pi(x) - \frac{1}{m})_+] = n \cdot [\pi(p_v)]$. So, if $z \in \mathbb{T}$ and $i \in \{1, 2, \dots, n \}$ then $\lambda_i(z) > \frac{1}{m}$ and consequently, there exists $\mu > 0$ small enough such that for $z \in \mathbb{T}$,
\begin{equation*}
\big(c_{x_{ij}}(z) \big)_{i, j = 1}^n \geq u_z^* \Big(\Big(\mu + \frac{1}{m} \Big) 1_n(z) \Big) u_z = \Big(\mu + \frac{1}{m} \Big) 1_n(z). 
\end{equation*}
Since $z \in \mathbb{T}$ was arbitrary then as positive elements in $M_n(C(\mathbb{T}))$,
\begin{equation*}
\pi(x) = \big(c_{x_{ij}}(\pi(s_g)) \big)_{i, j = 1}^n \geq \Big(\mu + \frac{1}{m} \Big) \cdot (\pi(p_v) \otimes 1_n).
\end{equation*}
Since $\pi(x) = \pi \big( (c_{x_{ij}}(s_g))_{i, j} \big)$ then there exist $y \in M_n(C^*(E))$ and $i \in I_H \otimes M_n(\mathbb{C})$ self-adjoint such that 
\begin{equation*}
\big(c_{x_{ij}}(s_g) \big)_{i, j = 1}^n - \Big(\mu + \frac{1}{m} \Big) \cdot (p_v \otimes 1_n) = y^* y + i. 
\end{equation*}
Now write $i = i_+ - i_-$ where $i_+$ and $i_-$ are the positive and negative parts of $i$ respectively. If $\epsilon > 0$ then by Lemma \ref{lemma: Positive I Arbitrarily Small}, there exists $\ell \geq 0$ large enough such that $$\lVert ((s_g^*)^\ell \otimes 1_n) i_- (s_g^\ell \otimes 1_n) \rVert < \epsilon.$$ So by \eqref{equation: Paths In W Cutdown},
\begin{IEEEeqnarray*}{rCl}
\mu \cdot (p_v \otimes 1_n) & = & ((s_g^*)^\ell \otimes 1_n) \big( \mu \cdot (p_v \otimes 1_n) \big) (s_g^\ell \otimes 1_n)
\\
& \leq & ((s_g^*)^\ell \otimes 1_n) \big( \mu \cdot (p_v \otimes 1_n) + y^* y + i_+ \big) (s_g^\ell \otimes 1_n)
\\
& = & ((s_g^*)^\ell \otimes 1_n) \Big( \Big(c_{x_{ij}}(s_g) - \frac{1}{m} \delta_{i, j} \; p_v \Big)_{i, j = 1}^n + i_- \Big) (s_g^\ell \otimes 1_n)
\\
& \leq & \Big(c_{x_{ij}}(s_g) - \frac{1}{m} \delta_{i, j} \; p_v \Big)_{i, j = 1}^n + \epsilon \cdot (1_{C^*(E)} \otimes 1_n)
\end{IEEEeqnarray*} 
where $\delta_{i, j}$ is the Kronecker delta. Since this holds for arbitrary $\epsilon > 0$ and the subset of positive elements of a C*-algebra is always norm-closed, we find that
\begin{equation*}
\mu \cdot (p_v \otimes 1_n) \leq \Big(c_{x_{ij}}(s_g) -  \frac{1}{m} \delta_{i, j} \; p_v \Big)_{i, j = 1}^n
\end{equation*}
as positive elements in $M_n(C^*(E))$. Now for $i, j \in \{1, 2, \dots, n \}$, write $x_{ij} = a_{ij} + b_{ij}$ where
\begin{equation*}
a_{ij} \in \Span \{ s_\mu s_\nu^* \; \vert \; \mu, \nu \in E^*, \; s(\mu) = s(\nu) = v \}
\end{equation*}
and
\begin{equation*}
b_{ij} \in \Span \{ s_\mu s_\nu^* \; \vert \; \mu, \nu \in E^*, \; s(\mu) = s(\nu) \in H \}.
\end{equation*}
Applying Proposition \ref{proposition: Sg Cutdown} to the $a_{ij}$ and noting that by the Cuntz-Krieger relations, $(s_g^*)^{\ell} b_{ij} s_g^{\ell} = 0$ for large enough $\ell > 0$, there exists $k > 0$ large enough such that
\begin{equation*}
\big(c_{x_{ij}}(s_g) \big)_{i, j = 1}^n = ((s_g^*)^k \otimes 1_n) x (s_g^k \otimes 1_n). 
\end{equation*}
So
\begin{IEEEeqnarray*}{rCl}
\mu \cdot (p_v \otimes 1_n) & \leq & ((s_g^*)^k \otimes 1_n) \Big( x - \frac{1}{m} \cdot (p_v \otimes 1_n) \Big) (s_g^k \otimes 1_n)
\\
& = & ((s_g^*)^k \otimes 1_n) \Big( x - \frac{1}{m} \cdot ((1_{C^*(H)} + p_v) \otimes 1_n) \Big) (s_g^k \otimes 1_n)
\\
& = & ((s_g^*)^k \otimes 1_n) \Big( x - \frac{1}{m} \cdot (1_{C^*(E)} \otimes 1_n) \Big) (s_g^k \otimes 1_n)
\\
& \leq & ((s_g^*)^k \otimes 1_n) \Big( x - \frac{1}{m} \Big)_+ (s_g^k \otimes 1_n). 
\end{IEEEeqnarray*} 
Taking Cuntz classes, we obtain $n \cdot [p_v] \leq [(x - \frac{1}{m})_+]$. The reverse inequality follows by applying Theorem \ref{theorem: CRS} to the inequality
\begin{equation*}
\Cu(\pi)\Big(\Big[ \Big( x - \frac{1}{m} \Big)_+ \Big] \Big) = \Big[\pi \Big( \Big( x - \frac{1}{m} \Big)_+ \Big) \Big] \leq n \cdot [\pi(p_v)] = \Cu(\pi)(n \cdot [p_v])
\end{equation*}
and then using that $n \cdot [p_v] + \infty \cdot [1_{C^*(H)}] = n \cdot [p_v]$ from the Cuntz-Krieger relations. 
\end{proof}

We will now combine Lemma \ref{lemma: Span Matrix Cutdown Quotients To Projection}, Lemma \ref{lemma: Sandwich} and the Cuntz-Krieger relations to produce the inductive argument required to overcome the major difficulty in proving Lemma \ref{lemma: Purely Large Circle Quotient}; namely that $C^*(E)$ being a purely large extension implies that its Cuntz semigroup decomposes as the stack. 

\begin{proof}[Proof of Lemma \ref{lemma: Purely Large Circle Quotient}]
If $\Cu(C^*(E)) \cong \Cu(C^*(H)) \stack \Lsc(\mathbb{T}, \overline{\mathbb{N}})$ then by definition of the stack and Corollary \ref{corollary: Purely Large Order Stabilised}, $C^*(E)$ is a purely large extension of $C(\mathbb{T})$ by $I_H \cong C^*(H) \otimes K$. Conversely, assume that $C^*(E)$ is a purely large extension of $C(\mathbb{T})$ by $I_H$. We will first prove by induction that if $n > 0$ and $x \in (C^*(E) \otimes M_n(\mathbb{C}))_+$  does not lie in $I_H \otimes M_n(\mathbb{C})$ then $[x] + \infty \cdot [1_{C^*(H)}] = [x]$. \newline
\phantom{} \newline
For the base case, assume that $n = 1$ so that $x \in C^*(E)_+ \backslash I_H$. Let $\epsilon > 0$ be small enough so that $(x - \epsilon)_+ \in C^*(E)_+ \backslash I_H$. Since $\mathcal{L}(E)$ is a dense *-subalgebra of $C^*(E)$, $\mathcal{L}(E)_+$ is dense in $C^*(E)_+$ and so by Lemma \ref{lemma: Countably Based}, there exist $y \in \mathcal{L}(E)_+$ and $m > 0$  such that 
\begin{equation} \label{equation: Sup Dense}
[(x - \epsilon)_+] \leq \Big[\Big(y - \frac{1}{m} \Big)_+ \Big] \leq [x].
\end{equation}
Since $[(x - \epsilon)_+] \not \in \Cu(I_H)$ by assumption then $[(y - \frac{1}{m})_+] \not \in \Cu(I_H)$. By Proposition \ref{proposition: Purely Large Axiom (O5)} and Lemma \ref{lemma: Span Matrix Cutdown Quotients To Projection}, 
\begin{equation*}
\Big[\Big(y - \frac{1}{m} \Big)_+ \Big] + \infty \cdot [1_{C^*(H)}] = \Big[\Big(y - \frac{1}{m} \Big)_+ \Big]. 
\end{equation*}
Hence by the LHS inequality of \eqref{equation: Sup Dense},
\begin{equation*}
[(x - \epsilon)_+] + \infty \cdot [1_{C^*(H)}]  \leq \Big[\Big(y - \frac{1}{m} \Big)_+ \Big] \leq [x].
\end{equation*}
Taking the supremum over $\epsilon > 0$, we obtain $[x] + \infty \cdot [1_{C^*(H)}] = [x]$. This proves the base case. \newline
\phantom{} \newline
For the inductive hypothesis, suppose that there exists $k > 0$ such that if $x \in (C^*(E) \otimes M_k(\mathbb{C}))_+$ does not lie in $I_H \otimes M_k(\mathbb{C})$ then $[x] + \infty \cdot [1_{C^*(H)}] = [x]$. Now recall the definition of $\Sigma_A$ from \eqref{equation: Scale}. We claim that if $[y] \in \Sigma_{M_k(C^*(E))} \backslash \Cu(I_H)$ then $[y] + \infty \cdot [1_{C^*(H)}] = [y]$. To this end, assume that $[y] \in \Sigma_{M_k(C^*(E))} \backslash \Cu(I_H)$. Pick $\epsilon > 0$ small enough so that $[(y - \epsilon)_+] \not \in \Cu(I_H)$. By Lemma \ref{lemma: Sandwich} applied to $M_k(C^*(E))$, there exists $b \in M_k(C^*(E))_+$ such that $[(y - \epsilon)_+] \ll [b] \ll [y]$. Since $[(y - \epsilon)_+] \not \in \Cu(I_H)$ by construction then $[b] \not \in \Cu(I_H)$ and by the inductive hypothesis, 
\begin{equation*}
[(y - \epsilon)_+] + \infty \cdot [1_{C^*(H)}] \leq [b] + \infty \cdot [1_{C^*(H)}] = [b] \leq [y]. 
\end{equation*}
Taking the supremum over all $\epsilon > 0$, $[y] + \infty \cdot [1_{C^*(H)}] = [y]$. \newline
\phantom{} \newline
Now assume that $a \in (C^*(E) \otimes M_{k+1}(\mathbb{C}))_+$ does not lie in $I_H \otimes M_{k+1}(\mathbb{C})$. By Lemma \ref{lemma: Countably Based}, there exist $z = (z_{ij})_{i, j = 1}^{k+1} \in M_{k+1}(C^*(E))_+$ and $m > 0$ such that $z_{ij} \in \mathcal{L}(E)$ for $i, j \in \{1, 2, \dots, k+1 \}$ and 
\begin{equation*}
[(a - \epsilon)_+] \leq \Big[ \Big( z - \frac{1}{m} \Big)_+ \Big] \leq [a]. 
\end{equation*}
Since $[(a - \epsilon)_+] \in \Cu(C^*(E)) \backslash \Cu(I_H)$ then $[(z - \frac{1}{m})_+] \not \in \Cu(I_H)$. To see that $[(z - \frac{1}{m})_+]$ absorbs $\infty \cdot [1_{C^*(H)}]$, we split into three cases. First suppose that there does not exist a projection $p \in C(\mathbb{T}) \otimes K$ such that $[\pi((z - \frac{1}{m})_+)] = [p]$. Since $C^*(E)$ is a purely large extension of $C(\mathbb{T})$ by $I_H$ then by Proposition \ref{proposition: Purely Large Axiom (O5)},  $[(z - \frac{1}{m})_+] + \infty \cdot [1_{C^*(H)}] = [(z - \frac{1}{m})_+]$. \newline
\phantom{} \newline
Otherwise, if $[\pi((z - \frac{1}{m})_+)] = (k+1) \cdot [\pi(p_v)]$ then absorption  of $\infty \cdot [1_{C^*(H)}]$ follows from Lemma \ref{lemma: Span Matrix Cutdown Quotients To Projection}. Finally, if there exists $\ell \in \{1, 2, \dots, k \}$ such that $[\pi((z - \frac{1}{m})_+)] = \ell \cdot [\pi(p_v)]$ then by Theorem \ref{theorem: CRS} and the Cuntz-Krieger relations,
\begin{equation*}
\Big[ \Big( z - \frac{1}{m} \Big)_+ \Big] \leq \ell \cdot [p_v] + \infty \cdot [1_{C^*(H)}] = \ell \cdot [p_v] \leq k \cdot [p_v].
\end{equation*}
So $[(z - \frac{1}{m})_+] \in \Sigma_{M_k(C^*(E))} \backslash \Cu(I_H)$ and by the previous claim, it absorbs $\infty \cdot [1_{C^*(H)}]$. \newline
\phantom{} \newline
Combining all three cases, we have $[(z - \frac{1}{m})_+] + \infty \cdot [1_{C^*(H)}] = [(z - \frac{1}{m})_+]$ and
\begin{equation*}
[(a - \epsilon)_+] + \infty \cdot [1_{C^*(H)}] \leq \Big[ \Big( z - \frac{1}{m} \Big)_+ \Big] + \infty \cdot [1_{C^*(H)}] = \Big[ \Big( z - \frac{1}{m} \Big)_+ \Big] \leq [a].
\end{equation*}
By taking the supremum over $\epsilon > 0$, $[a] + \infty \cdot [1_{C^*(H)}] \leq [a]$. This completes the induction. \newline
\phantom{} \newline
Finally assume that $x \in (C^*(E) \otimes K)_+ \backslash (I_H \otimes K)$. Once again, let $\epsilon > 0$ be small enough so that $[(x - \epsilon)_+] \in \Cu(C^*(E)) \backslash \Cu(I_H)$. By Rørdam's lemma (see \cite[Proposition 2.4]{Ror92}), there exist $n > 0$ large enough and $a \in (C^*(E) \otimes M_n(\mathbb{C}))_+$ such that $[(x - \epsilon)_+] = [a]$. By the previous result, $$[ (x - \epsilon)_+] + \infty \cdot [1_{C^*(H)}] = [a] + \infty \cdot [1_{C^*(H)}] = [a] \leq [x].$$ Taking the supremum over all $\epsilon > 0$, we have $[x] + \infty \cdot [1_{C^*(H)}] = [x]$. Proposition \ref{proposition: Stack Characterisation} finishes the proof. 
\end{proof}

Now recall from Definition \ref{definition: QDS} that if $E$ is a finite graph then $E^0$ is a hereditary subset of $(S^2 E)^0$ and we obtain the extension
\begin{equation} \label{equation: QDS Extension}
\begin{tikzcd}
0 \arrow[r] & I_E \arrow[r] & C^*(S^2 E) \arrow[r, "\pi"] & C(\mathbb{T}) \arrow[r] & 0.
\end{tikzcd}
\end{equation}
By \cite[Lemma 1.1]{HS03}, the extension in \eqref{equation: QDS Extension} is essential. Now if $m \in \{1, 2, \dots \}$ then define
\begin{equation} \label{equation: QDS Matrix Unit}
a_{m1} = 
\begin{cases}
\sum_{w \in E^0} p_{w}, \text{ if $m = 1$,} \\
\sum_{w \in E^0} s_{f_w}, \text{ if $m = 2$,} \\
s_g^{m-2} \sum_{w \in E^0} s_{f_w}, \text{ if $m > 2$,}
\end{cases}
\end{equation}
where for $w \in E^0$, the edges $f_w$ and $g$ are as in Definition \ref{definition: QDS}. We also define $a_{1m} = a_{m1}^*$ and $a_{mn} = a_{m1} a_{1n}$ for $m, n \in \{1, 2, \dots \}$. By the proof of \cite[Proposition 6.2]{HS02}, the family $\{ a_{ij} \}_{i, j = 1}^{\infty}$ defines a matrix unit in $K$ and the map
\begin{equation} \label{equation: QDS Ideal Iso}
\begin{matrix}
C^*(E) \otimes K & \rightarrow & I_E \\
x \otimes a_{ij} & \mapsto & a_{i1} x a_{1j}
\end{matrix}
\end{equation}
extends to an isomorphism $C^*(E) \otimes K \cong I_E$. Moreover, if $\mu: \mathcal{M}(C^*(E) \otimes K) \rightarrow \mathcal{Q}(C^*(E) \otimes K)$ is the quotient map then the Busby invariant of the extension in \eqref{equation: QDS Extension} is given explicitly by sending $\pi(s_g)$ to $\mu(1_{C^*(E)} \otimes s_g)$ where $1_{C^*(E)}$ is the unit of $C^*(E)$. 

\begin{lemma} \label{lemma: QDS Purely Large}
Let $E$ be a finite graph. The extension in \eqref{equation: QDS Extension} is purely large. 
\end{lemma}

\begin{proof}
Since the extension in \eqref{equation: QDS Extension} is essential, its associated Busby invariant $\beta: C(\mathbb{T}) \rightarrow \mathcal{Q}(C^*(E) \otimes K)$ is an injective *-homomorphism. We may view the multiplier algebra $\mathcal{M}(K)$ as a subalgebra of $\mathcal{M}(C^*(E) \otimes K)$ via the following chain of inclusions:
\begin{equation*}
\mathcal{M}(K) \cong \{ 1_{C^*(E)} \} \otimes \mathcal{M}(K) \subseteq C^*(E) \otimes \mathcal{M}(K) \subseteq \mathcal{M}(C^*(E) \otimes K). 
\end{equation*}
In particular, the Calkin algebra $\mathcal{Q}(K)$ is a unital subalgebra of $\mathcal{Q}(C^*(E) \otimes K)$ because by \eqref{equation: QDS Matrix Unit} and \eqref{equation: QDS Ideal Iso},
\begin{equation*}
1_{\mathcal{M}(C^*(E) \otimes K)} = {\sum}_{i = 1}^{\infty} a_{ii} \in \mathcal{M}(K). 
\end{equation*}  
Now by \cite[Proposition 6.2]{HS02}, $\beta(\pi(s_g)) = \mu(1_{C^*(E)} \otimes s_g) \in \mathcal{Q}(K)$. Since $\pi(s_g)$ generates the quotient $C(\mathbb{T})$ then $\beta$ factors through $\mathcal{Q}(K)$. Consequently, if $x \in C(\mathbb{T})$ is non-zero then the ideal generated by $\beta(x) \neq 0$ in $\mathcal{Q}(K)$ contains the unit of $\mathcal{Q}(C^*(E) \otimes K)$. Therefore, $\beta$ is a full *-homomorphism and the extension in \eqref{equation: QDS Extension} is full. Moreover, the ideal $I_E \cong C^*(E) \otimes K$ is stable, $\sigma$-unital and satisfies the corona factorisation property by Proposition \ref{proposition: Graph CFP}. So by Theorem \ref{theorem: Kucerovsky Ng}, $C^*(S^2 E)$ is a purely large extension of $C(\mathbb{T})$ by $C^*(E) \otimes K$. 
\end{proof}

We are now in a position to finish the proof of Theorem \ref{introtheorem: General Cuntz Krieger}.

\begin{proof}[Proof of Theorem \ref{introtheorem: General Cuntz Krieger} when $v$ supports a single loop]
Let $E$ and $H$ be as in Theorem \ref{introtheorem: General Cuntz Krieger} and suppose $v$ supports a single loop. By Proposition \ref{proposition: QDS Forcing}, there exists a finite graph $F$ such that $C^*(E)$ is stably isomorphic to $C^*(S^2 F)$ and $C^*(H)$ is stably isomorphic to $C^*(F)$. Thus, by Lemma \ref{lemma: QDS Purely Large} and Lemma \ref{lemma: Purely Large Circle Quotient},
\begin{IEEEeqnarray*}{rCl}
\Cu(C^*(E)) & \cong & \Cu(C^*(S^2 F))
\\
& \cong & \Cu(C^*(F)) \stack \Lsc(\mathbb{T}, \overline{\mathbb{N}})
\\
& \cong & \Cu(C^*(H)) \stack \Lsc(\mathbb{T}, \overline{\mathbb{N}}). 
\end{IEEEeqnarray*} 
This completes the proof.
\end{proof}

In the situation of Theorem \ref{introtheorem: General Cuntz Krieger} when $v$ supports a single loop, we can concretely realise the characteristic functions of open subsets of $\mathbb{T}$ lying in $\Cu(C^*(H)) \stack \Lsc(\mathbb{T}, \overline{\mathbb{N}})$ as Cuntz classes of strictly positive elements of ideals of $C^*(E)$ which properly contain $I_H$. This is the content of Lemma \ref{lemma: Same Open Subset Upstairs}, which allows the various relations coming from $\Cu(C^*(H)) \stack \Lsc(\mathbb{T}, \overline{\mathbb{N}})$ to be transferred to $\Cu(C^*(F))$ where $F$ is a finite graph which contains $H$ as a subgraph and satisfies the standing assumptions.

\begin{lemma} \label{lemma: Same Open Subset Upstairs}
Let $E$ be a finite graph and suppose that there exists $v \in E^0$ such that $v$ supports a single loop $g$ and if $w \in E^0$ then $w \geq v$. Let $H = \{ w \in E^0 \; \vert \; v \not \geq w \} = E^0 \backslash \{ v \}$ and $U \subseteq \mathbb{T}$ be non-empty and open. As in Definition \ref{definition: Strictly Positive}, let $h_{v, U}$ be a strictly positive element of $$p_{L(v)} \big( I_{L(v)} \cap \mathcal{R}_{v, \mathbb{T} \backslash U} \big) p_{L(v)} = \mathcal{R}_{v, \mathbb{T} \backslash U}$$ and $\chi_U \in \Lsc(\mathbb{T}, \overline{\mathbb{N}})$ denote the characteristic function of $U$. Then the isomorphism $\Cu(C^*(E)) \cong \Cu(C^*(H)) \stack \Lsc(\mathbb{T}, \overline{\mathbb{N}})$ from Theorem \ref{introtheorem: General Cuntz Krieger} sends $[h_{v, U}]$ to $\chi_U$. 
\end{lemma}

\begin{proof}
Let $\phi$ denote the Cu-isomorphism between $\Cu(C^*(H)) \stack \Lsc(\mathbb{T}, \overline{\mathbb{N}})$ and $\Cu(C^*(E))$ constructed in Proposition \ref{proposition: Stack Characterisation}, $\pi$ be the quotient map from $C^*(E)$ to $C(\mathbb{T})$ and $\rho$ be the map from $\Cu(C^*(H)) \stack \Lsc(\mathbb{T}, \overline{\mathbb{N}})$ to $\Lsc(\mathbb{T}, \overline{\mathbb{N}})$ which is zero on $\Cu(C^*(H))$ and restricts to the identity on $\Lsc(\mathbb{T}, \overline{\mathbb{N}})$. By Proposition \ref{proposition: Stack Characterisation}, the following diagram commutes:
\begin{equation} \label{equation: Stack Commutative Diagram}
\begin{tikzcd}
\Cu(C^*(E)) \arrow[rr, "\phi"] \arrow[ddrr, swap, "\Cu(\pi)"] & & \Cu(C^*(H)) \stack \Lsc(\mathbb{T}, \overline{\mathbb{N}}) \arrow[dd, "\rho"] \\
& & \\
& & \Lsc(\mathbb{T}, \overline{\mathbb{N}})
\end{tikzcd}
\end{equation} 
Let $U$ be a non-empty, open subset of $\mathbb{T}$. If $U = \mathbb{T}$ then $h_{v, \mathbb{T}}$ is a strictly positive element of the hereditary subalgebra $p_{L(v)} I_{L(v)} p_{L(v)} = C^*(E)$. By the Cuntz-Krieger relations on $E$, $$[h_{v, \mathbb{T}}] = {\sum}_{x \in E^0} [p_x] = [p_v]$$ in $\Cu(C^*(E))$. By commutativity of diagram \eqref{equation: Stack Commutative Diagram}, 
\begin{equation*}
(\rho \circ \phi)([h_{v, \mathbb{T}}]) = \Cu(\pi)([p_v]) = [\pi(p_v)] = \chi_{\mathbb{T}} = \rho(\chi_{\mathbb{T}}). 
\end{equation*}
Now observe that as Cu-semigroups, 
\begin{equation} \label{equation: Stack Quotient}
\Lsc(\mathbb{T}, \overline{\mathbb{N}}) \cong \big(\Cu(C^*(H)) \stack \Lsc(\mathbb{T}, \overline{\mathbb{N}}) \big) / \Cu(C^*(H)). 
\end{equation}
Hence by definition of a quotient Cu-semigroup (see Definition \ref{definition: Quotient Cu Semigroup}), $$\phi([h_{v, \mathbb{T}}]) \leq \chi_{\mathbb{T}} + \infty \cdot {\sum}_{x \in H} [p_x] = \chi_{\mathbb{T}}$$ in $\Cu(C^*(H)) \stack \Lsc(\mathbb{T}, \overline{\mathbb{N}})$ because $\chi_{\mathbb{T}} \not \in \Cu(C^*(H))$. The reverse inequality follows from an analogous argument, where we note that $[h_{v, \mathbb{T}}] = [p_v] \not \in \Cu(C^*(H))$. \newline
\phantom{} \newline
Next assume that $U$ is a proper subset of $\mathbb{T}$. Recall from Definition \ref{definition: Primitive Non Gauge Invariant} that if $t \in \mathbb{T}$ then $\mathcal{R}_{v, t}$ is the primitive, non-gauge-invariant ideal $\langle I_H, s_g - t p_v \rangle$. So if $t \in \mathbb{T} \backslash U$ then $\pi(\mathcal{R}_{v, t}) = \langle \pi(s_g) - t \pi(p_v) \rangle$ and 
\begin{equation} \label{equation: Explicit Continuous Function}
(\pi(s_g) - t \pi(p_v) )^* (\pi(s_g) - t \pi(p_v) ) = 2 \pi(p_v) - t \pi(s_g)^* - t^* \pi(s_g). 
\end{equation}
The element on the RHS of \eqref{equation: Explicit Continuous Function} is the continuous function $z \mapsto 2 - t z^* - t^* z$, which is zero if and only if $z = t$. So
\begin{equation*}
\rank \big(z \mapsto 2 - t z^* - t^* z \big) = 
\begin{cases}
0, \text{ if $z = t$,} \\
1, \text{ otherwise.}
\end{cases}
\end{equation*}
Now since the isomorphism $\Cu(C(\mathbb{T})) \cong \Lsc(\mathbb{T}, \overline{\mathbb{N}})$ is implemented by the pointwise rank map, $[2 \pi(p_v) - t \pi(s_g)^* - t^* \pi(s_g)] = \chi_{\mathbb{T} \backslash \{ t \}}$ in $\Lsc(\mathbb{T}, \overline{\mathbb{N}})$. Moreover, $2 \pi(p_v) - t \pi(s_g)^* - t^* \pi(s_g)$ is a full, positive element of the ideal $\langle \pi(s_g) - t \pi(p_v) \rangle$ and thus, as ideals of $\Lsc(\mathbb{T}, \overline{\mathbb{N}})$, 
\begin{equation} \label{equation: Primitive Ideal Cuntz Semigroup}
\Cu(\langle \pi(s_g) - t \pi(p_v) \rangle) = \big\{ f \in \Lsc(\mathbb{T}, \overline{\mathbb{N}}) \; \vert \; f \leq \infty \cdot \chi_{\mathbb{T} \backslash \{ t \}} \big\}.
\end{equation}
Since \eqref{equation: Primitive Ideal Cuntz Semigroup} holds for $t \in \mathbb{T} \backslash U$, we have
\begin{equation*}
\Cu(\pi(\mathcal{R}_{v, \mathbb{T} \backslash U})) = \big\{ f \in \Lsc(\mathbb{T}, \overline{\mathbb{N}}) \; \vert \; f \leq \infty \cdot \chi_{U} \big\}. 
\end{equation*}
Now since $h_{v, U}$ is a strictly positive element of the ideal $\mathcal{R}_{v, \mathbb{T} \backslash U}$, $\pi(h_{v, U})$ is a strictly positive element of $\pi(\mathcal{R}_{v, \mathbb{T} \backslash U})$ and 
\begin{equation*}
\Cu(\pi(\mathcal{R}_{v, \mathbb{T} \backslash U})) = \big\{ f \in \Lsc(\mathbb{T}, \overline{\mathbb{N}}) \; \vert \; f \leq \infty \cdot [\pi(h_{v, U})] \big\}.
\end{equation*}
So $\infty \cdot \chi_{U} = \infty \cdot [\pi(h_{v, U})]$. By taking infima and noting that $[\pi(h_{v, U})]$ and $\chi_{U}$ are both below $[\pi(p_v)]$ in $\Lsc(\mathbb{T}, \overline{\mathbb{N}})$, we have
\begin{IEEEeqnarray*}{rCl}
[\pi(h_{v, U})] & = & [\pi(h_{v, U})] \wedge [\pi(p_v)]
\\
& = & (\infty \cdot [\pi(h_{v, U})]) \wedge [\pi(p_v)]
\\
& = & (\infty \cdot \chi_{U} ) \wedge [\pi(p_v)]
\\
& = & \chi_{U} \wedge [\pi(p_v)] = \chi_{U}. 
\end{IEEEeqnarray*} 
So by commutativity of diagram \eqref{equation: Stack Commutative Diagram} once more, 
\begin{equation*}
(\rho \circ \phi)([h_{v, U}]) = \Cu(\pi)([h_{v, U}]) = [\pi(h_{v, U})] = \chi_{U} = \rho(\chi_{U}).
\end{equation*}
Once again by \eqref{equation: Stack Quotient}, the definition of the stack and the fact that $\chi_{U} \not \in \Cu(C^*(H))$, $\phi([h_{v, U}]) \leq \chi_{U}$. The reverse inequality follows from \eqref{equation: Stack Quotient}, together with the fact that $\phi([h_{v, U}]) \not \in \Cu(C^*(H))$, as $(\rho \circ \phi)([h_{v, U}]) = \chi_U \neq 0$. 
\end{proof}

\section{The map $\Delta_E$ from $\Cu_E$ to $\Cu(C^*(E))$}
\label{sec: Delta E}

The main goal of this section is to prove Theorem \ref{introtheorem: Cu E Iso}. To do this, we will prove that for a finite graph $E$ satisfying the standing assumptions, the map $\Delta_E: \Cu_E \rightarrow \Cu(C^*(E))$ defined in Definition \ref{definition: Delta E} is a Cu-isomorphism. As discussed in Remark \ref{remark: Delta E From Singleton}, the construction of $\Delta_E$ is motivated by the bijection $\mathcal{I}: \Lat(C^*(E)) \rightarrow \Lat(\Cu_E)$ from Theorem \ref{theorem: Ideal Lattice Iso}. To summarise, the functions from Definition \ref{definition: Singletons} are mapped via $\Delta_E$ to linear combinations of strictly positive elements coming from hereditary subalgebras of $C^*(E)$ that are determined by $\mathcal{I}$; see also Definition \ref{definition: Strictly Positive}. \newline
\phantom{} \newline
Observe that if $v \in E^0$ is not a source vertex, then the subgraph $L(v)$ satisfies the assumptions of Theorem \ref{introtheorem: General Cuntz Krieger} because $E$ satisfies the standing assumptions. As emphasised previously, we will use the isomorphism $C^*(L(v)) \cong p_{L(v)} I_{L(v)} p_{L(v)}$ to transfer the relations amongst elements in $\Cu(C^*(L(v)))$ --- which is isomorphic to a stack of two Cu-semigroups --- to $\Cu(C^*(E))$. In particular, when $v$ supports a single loop, we find from Lemma \ref{lemma: Same Open Subset Upstairs} that the family of Cuntz classes $$\{ [h_{v, U}] \}_{U \subseteq \mathbb{T} \text{ open}} \subseteq \Cu(C^*(E))$$ behaves exactly like the family of characteristic functions $\{ \chi_U \}_{U \subseteq \mathbb{T} \text{ open}}$ generating $\Lsc(\mathbb{T}, \overline{\mathbb{N}})$ as a CU-semigroup. Moreover, the absorption relations one obtains from applying Theorem \ref{introtheorem: General Cuntz Krieger} to $L(v)$ and then transferring to $\Cu(C^*(E))$ precisely mimic the absorption relations satisfies by functions in $\Cu_E$; see the third property of Definition \ref{definition: Cu E}. These are the critical reasons why $\Delta_E$ is a Cu-morphism. See the proof of Proposition \ref{proposition: Delta E Cu Morphism} for further details. \newline
\phantom{} \newline
To show that $\Delta_E$ is a Cu-isomorphism, it suffices to prove that $\Delta_E$ is an order-embedding and surjective. We will adopt a similar proof strategy to \cite[Theorem 7.1]{AMP07} and proceed by induction over the (finite) gauge-invariant ideal lattice of $C^*(E)$, taking advantage of the fact that the number of vertices of the skeleton $S(E)$ of $E$ is exactly the gauge length of $C^*(E)$; see Lemma \ref{lemma: Gauge Length}. The commutative squares in Lemma \ref{lemma: Ideal Square} and Lemma \ref{lemma: Quotient Square} allow the inductive hypothesis to be used. Combining these commutative squares with specific properties of $\Cu_E$ is enough to conclude that $\Delta_E$ is an order embedding. \newline
\phantom{} \newline
The proof that $\Delta_E$ is surjective eventually spilts into two cases, based on whether $C^*(E)$ has a simple purely infinite ideal or a gauge-simple, gauge-invariant ideal. Lemma \ref{lemma: Polar Part Absorbed} and Theorem \ref{theorem: SR1 Corner Positive Cancellation} will be used to deal with these respective cases. It is interesting that the proof of surjectivity uses Theorem \ref{theorem: SR1 Corner Positive Cancellation}, as the analogous result for the Murray-von Neumann semigroup originating from \cite[Theorem 2.1]{War80} is used by Ara, Moreno and Pardo in the proof of \cite[Theorem 7.1]{AMP07} to prove that their candidate monoid morphism is injective. We also note that the stronger cancellation result Theorem \ref{theorem: General Weak Cancellation} is in fact, not required for the proof that $\Delta_E$ is surjective. This is because every gauge-invariant ideal $I$ of $C^*(E)$ is generated by the vertex projections coming from a hereditary subset $H$ of $E^0$ (see \eqref{equation: Hereditary Saturated} and Corollary \ref{corollary: Auto Saturated}), and hence, the largest element of $\Cu(I)$ is $\infty \cdot [{\sum}_{v \in H} p_v]$. 

\subsection{Definition and proof that it is a Cu-morphism}

Let $E$ be a finite graph satisfying the standing assumptions. Observe that if $v \in E^0$ is not a source vertex, then $L(v)$ is a subgraph of $E$ which satisfies the assumptions of Theorem \ref{introtheorem: General Cuntz Krieger}. Hence, if $H_v = \{ x \in L(v) \; \vert \; v \not \geq x \}$ then by Theorem \ref{introtheorem: General Cuntz Krieger},
\begin{equation} \label{equation: Subgraph Cuntz Semigroup}
\Cu(I_{L(v)}) \cong \Cu(I_{H_v}) \stack \Cu(C^*(L(v) \backslash H_v))
\end{equation}
where $\Cu(C^*(L(v) \backslash H_v))$ is isomorphic to either $\{0, \infty \}$ or $\Lsc(\mathbb{T}, \overline{\mathbb{N}})$. If $f \in \Cu_E$, $\check{v} \in \supp_0(f)_t$, $c(\check{v}) = 1$ and $i > 0$ then define
\begin{equation*}
\{ f(\check{v}) \geq i \} = \{ x \in \mathbb{T} \; \vert \; f(\check{v})(x) \geq i \}.
\end{equation*} 
By definition of a lower semicontinuous function, this is an open subset of $\mathbb{T}$.

\begin{definition} \label{definition: Delta E}
Let $E$ be a finite graph satisfying the standing assumptions. Recalling the definition of the strictly positive element $h_{v, U}$ from Definition \ref{definition: Strictly Positive}, we define $\Delta_E(f)$ to be the following element of $\Cu(C^*(E))$:
\begin{equation*}
\begin{split}
\Delta_E(f) & = {\sum}_{\check{v} \in \supp_0(f)_t, \; c(\check{v}) \neq 1} \Big( f(\check{v}) \cdot [p_v] + {\sum}_{w \in L(v), \; v \not \geq w} \infty \cdot [p_w] \Big) \\
& + {\sum}_{\check{v} \in \supp_0(f)_t, \; c(\check{v}) = 1} \Big( {\sum}_{i = 1}^{\infty} \big[h_{v, \; \{ f(\check{v}) \geq i \}} \big] + {\sum}_{w \in L(v), \; v \not \geq w} \infty \cdot [p_w] \Big).
\end{split}
\end{equation*}
Here, ${\sum}_{i = 1}^{\infty} [h_{v, \; \{ f(\check{v}) \geq i \}} ]$ is the supremum of the increasing sequence $$\Big\{ {\sum}_{i = 1}^{n} \big[h_{v, \; \{ f(\check{v}) \geq i \}} \big] \Big\}_n$$ and we set $h_{v, \emptyset} = 0$.  
\end{definition}

Once again, we emphasise that $\Delta_E$ does not depend on any choices of strictly positive elements, as any two strictly positive elements of a C*-algebra are Cuntz equivalent. 

\begin{remark} \label{remark: Delta E From Singleton}
Recall from Lemma \ref{lemma: Singleton Decomp} that if $f \in \Cu_E$ then $$f = {\sum}_{\check{v} \in \supp_0(f)_t} \theta_{\check{v}, f(\check{v})}$$ where the functions $\theta_{\check{v}, f(\check{v})}$ are defined in Definition \ref{definition: Singletons}. In turn, the functions in Definition \ref{definition: Singletons} are either finite sums of functions in the set
\begin{equation*}
\{ \theta_{\check{v}, 1} \; \vert \; c(\check{v}) = 0 \} \sqcup \{ \theta_{\check{w}, \infty} \; \vert \; c(\check{w}) = \infty \} \sqcup \left\{ \theta_{\check{x}, \chi_U} \; \Big\vert \; 
\begin{matrix}
c(\check{x}) = 1 \text{ and } U \subseteq \mathbb{T} \\
\text{ is non-empty and open}
\end{matrix} \right\}
\end{equation*}
or the supremum of an increasing sequence of such finite sums. The essential idea behind Definition \ref{definition: Delta E} is that the functions in the above set are mapped to Cuntz classes of strictly positive elements lying in specific hereditary subalgebras of $C^*(E)$. \newline
\phantom{} \newline
To illustrate this, assume that $\check{v} \in S(E)^0$ satisfies $c(\check{v}) = 0$. Let $\mathcal{I}$ be the map defined in \eqref{equation: Ideal Lattice 2}, which is bijective by Theorem \ref{theorem: Ideal Lattice Iso}. To see where $\Delta_E$ sends the function $\theta_{\check{v}, 1}$, we first consider the ideal $\Idl(\theta_{\check{v}, 1})$ generated by $\theta_{\check{v}, 1}$ in $\Cu_E$ (see \eqref{equation: Singly Gen Ideal}). Use $\mathcal{I}$ to obtain the corresponding gauge-invariant ideal
\begin{equation*}
\mathcal{I}^{-1}\big( \Idl(\theta_{\check{v}, 1}) \big) = I_{L_{S(E)}(\check{v})^{\dagger}} = I_{L(v)} =  I_{\{ v \}}
\end{equation*}
of $C^*(E)$. The map $\Delta_E$ then maps $\theta_{\check{v}, 1}$ to the Cuntz class of a strictly positive element of the corner $p_v I_{\{ v \}} p_v$, which is $[p_v]$.  The cases where $c(\check{v}) \neq 0$ follow a similar principle. Note that we take strictly positive elements of corners of ideals in order to avoid choosing a strictly positive element of a stable ideal of $C^*(E)$, which may not be Cuntz dominated by a sum of vertex projections; see Example \ref{example: Why Corner} and the preceding discussion. \newline
\phantom{} \newline
Finally, we extend $\Delta_E$ to all of $\Cu_E$ by taking the appropriate linear combinations and suprema of increasing sequences so that
\begin{equation} \label{equation: Delta E Singleton}
\Delta_E(f) = {\sum}_{\check{v} \in \supp_0(f)_t}  \Delta_E(\theta_{\check{v}, f(\check{v})}).
\end{equation}
\end{remark}

\begin{remark} \label{remark: Delta E Ideal}
Suppose that $f \in \Cu_E$ and $\supp_0(f)$ is a proper hereditary subset of $S(E)^0$. By \eqref{equation: Hereditary Skeleton}, $\supp_0(f) = \langle \supp_0(f)^{\dagger} \rangle$ as hereditary subsets of $S(E)^0$, $$f \in \big\{ g \in \Cu_E \; \big\vert \; \supp_0(g) \subseteq \langle \supp_0(f)^{\dagger} \rangle \big\}$$ and by Definition \ref{definition: Delta E}, $\Delta_E(f) \in \Cu(I_{\supp_0(f)^{\dagger}})$. So by the contrapositive, if there does not exist a proper hereditary subset $H$ of $E^0$ such that $\Delta_E(f) \in \Cu(I_H)$ then $\supp_0(f)$ is not a proper hereditary subset of $S(E)^0$ and subsequently, $\supp_0(f) = S(E)^0$. This observation will play a major role in the proof of Theorem \ref{introtheorem: Cu E Iso}. 
\end{remark} 

By using Lemma \ref{lemma: Same Open Subset Upstairs} to handle the strictly positive elements coming from corners of non-gauge-invariant ideals of $C^*(E)$, we will now prove that $\Delta_E$ is a Cu-morphism. 

\begin{proposition} \label{proposition: Delta E Cu Morphism}
Let $E$ be a finite graph satisfying the standing assumptions. The map $\Delta_E: \Cu_E \rightarrow \Cu(C^*(E))$ is a Cu-morphism. 
\end{proposition}

\begin{proof}
We first prove that $\Delta_E$ is additive. Suppose that $f, g \in \Cu_E$ and let $\check{v}$ be a terminal vertex for $\supp_0(f + g) = \supp_0(f) \cup \supp_0(g)$. Similarly to the proof of Lemma \ref{lemma: Cu E Minus H}, observe that if $\check{v} \in \supp_0(f)$ then it must necessarily be a terminal vertex for $\supp_0(f)$. A similar statement holds for $\supp_0(g)$. Without loss of generality, suppose that $\check{v} \in \supp_0(f)$. If $\check{v} \not \in \supp_0(g)$ then there is nothing to show. Otherwise, $\check{v} \in \supp_0(g)$ and there are two cases to consider. Firstly, if $c(\check{v}) \neq 1$ then
{\small \begin{equation*}
\begin{split}
& \Big( f(\check{v}) \cdot [p_v] + {\sum}_{w \in L(v), \; v \not \geq w} \infty \cdot [p_w] \Big) + \Big( g(\check{v}) \cdot [p_v] + {\sum}_{w \in L(v), \; v \not \geq w} \infty \cdot [p_w] \Big) \\
= & \Big( (f + g)(\check{v}) \cdot [p_v] + {\sum}_{w \in L(v), \; v \not \geq w} \infty \cdot [p_w] \Big).
\end{split}
\end{equation*}}
Secondly, if $c(\check{v}) = 1$ then
\begin{IEEEeqnarray*}{rCl}
{\sum}_{i = 1}^{\infty} \chi_{\{ (f + g)(\check{v}) \geq i \} } & = & (f + g)(\check{v}) = f(\check{v}) + g(\check{v}) 
\\
& = & {\sum}_{i = 1}^{\infty} \chi_{\{ f(\check{v}) \geq i \} } + {\sum}_{i = 1}^{\infty} \chi_{\{ g(\check{v}) \geq i \} }
\end{IEEEeqnarray*} 
in $\Lsc(\mathbb{T}, \overline{\mathbb{N}})$. By Lemma \ref{lemma: Same Open Subset Upstairs} applied to the subgraph $L(v)$,  
\begin{equation} \label{equation: Local Application}
{\sum}_{i = 1}^{\infty} \big[h_{v, \; \{ f(\check{v}) \geq i \}} \big] + {\sum}_{i = 1}^{\infty} \big[h_{v, \; \{ g(\check{v}) \geq i \}} \big] = {\sum}_{i = 1}^{\infty} \big[h_{v, \; \{ (f + g)(\check{v}) \geq i \}} \big]
\end{equation}
in $\Cu(C^*(L(v)))$. We carefully observe that in \eqref{equation: Local Application}, if $\{ f(\check{v}) \geq i \}$ is non-empty then $h_{v, \; \{ f(\check{v}) \geq i \}}$ is a strictly positive element of 
\begin{equation*}
p_{L(v)} \big( C^*(L(v)) \cap \mathcal{R}_{v, \; \mathbb{T} \backslash \{ f(\check{v}) \geq i \}} \big) p_{L(v)} = \mathcal{R}_{v, \mathbb{T} \backslash \{ f(\check{v}) \geq i \}},
\end{equation*}
viewed as a hereditary subalgebra of $C^*(L(v))$. By \eqref{equation: Full Corner Iso}, we have the isomorphism $C^*(L(v)) \cong p_{L(v)} I_{L(v)} p_{L(v)}$. Under this isomorphism, $h_{v, \; \{ f(\check{v}) \geq i \}}$ is mapped to a strictly positive element $h'$ of
\begin{equation*}
p_{L(v)} \big( I_{L(v)} \cap \mathcal{R}_{v, \; \mathbb{T} \backslash \{ f(\check{v}) \geq i \}} \big) p_{L(v)}
\end{equation*}
which is now viewed as a hereditary subalgebra of $C^*(E)$. By Definition \ref{definition: Strictly Positive}, $h'$ is Cuntz equivalent to $h_{v, \; \{ f(\check{v}) \geq i \}}$. Applying the isomorphism $C^*(L(v)) \cong p_{L(v)} I_{L(v)} p_{L(v)}$ to the rest of the strictly positive elements appearing in \eqref{equation: Local Application}, we find that
\begin{equation*}
{\sum}_{i = 1}^{\infty} \big[h_{v, \; \{ f(\check{v}) \geq i \}} \big] + {\sum}_{i = 1}^{\infty} \big[h_{v, \; \{ g(\check{v}) \geq i \}} \big] = {\sum}_{i = 1}^{\infty} \big[h_{v, \; \{ (f + g)(\check{v}) \geq i \}} \big]
\end{equation*}
in $\Cu(C^*(E))$. Consequently, by summing over all terminal vertices of $\supp_0(f + g)$ and using \eqref{equation: Delta E Singleton},
\begin{equation} \label{equation: Delta E Additivity Step}
\Delta_E(f + g) = {\sum}_{\check{v} \in \supp_0(f+g)_t}  \Delta_E(\theta_{\check{v}, f(\check{v})}) + {\sum}_{\check{v} \in \supp_0(f+g)_t}  \Delta_E(\theta_{\check{v}, g(\check{v})}).
\end{equation}
To see that the RHS of \eqref{equation: Delta E Additivity Step} is equal to $\Delta_E(f) + \Delta_E(g)$, we need to add the contributions coming from vertices lying in $\supp_0(f)_t$ or $\supp_0(g)_t$ which are not terminal vertices for $\supp_0(f + g)$. So suppose that $\check{v} \in \supp_0(f)_t$ is not a terminal vertex for $\supp_0(f + g)$. By Lemma \ref{lemma: Hereditary Terminal}, there exists $\check{x} \in \supp_0(f + g)_t$ such that $\check{v}$ admits a finite path to $\check{x}$. \newline
\phantom{} \newline
Now by \eqref{equation: Subgraph Cuntz Semigroup}, $\Cu(I_{L(x)}) \cong \Cu(I_{H_x}) \stack \Cu(C^*(L(x) \backslash H_x))$. Since $\check{v}$ admits a finite path to $\check{x}$ and $\check{x}$ does not admit a finite path to $\check{v}$ (because $\check{v}$ is not a terminal vertex for $\supp_0(f + g)$), then $\Delta_E(\theta_{\check{v}, f(\check{v})}) \in \Cu(I_{H_x})$. Hence, by definition of the stack, 
\begin{equation} \label{equation: Delta E Additivity Absorption}
\Delta_E(\theta_{\check{x}, f(\check{x})}) + \Delta_E(\theta_{\check{v}, f(\check{v})}) = \Delta_E(\theta_{\check{x}, f(\check{x})})
\end{equation}
in $\Cu(I_{L(x)})$ and thus, in $\Cu(C^*(E))$ as well. The analogous equation holds with $g$ in place of $f$. Combining \eqref{equation: Delta E Additivity Absorption}, its analogue for $g$ and \eqref{equation: Delta E Additivity Step}, we find that $\Delta_E(f + g) = \Delta_E(f) + \Delta_E(g)$. \newline
\phantom{} \newline
Next, we prove that $\Delta_E$ is order-preserving. Once again, assume that $f, g \in \Cu_E$ and $f \leq g$. By Lemma \ref{lemma: Order Alternative}, $\supp_0(f) \subseteq \supp_0(g)$. If $\check{v} \in \supp_0(f)_t$ is not a terminal vertex for $\supp_0(g)$ then by Lemma \ref{lemma: Hereditary Terminal}, there exists $\check{x} \in \supp_0(g)_t$ such that $\check{v}$ admits a finite path to $\check{x}$. So for $\check{x} \in \supp_0(g)_t$, define
\begin{equation*}
O_{\check{x}} = \{ \check{v} \in \supp_0(f)_t \; \vert \; \text{$\check{v} \not \in \supp_0(g)_t$ and $\check{v} \geq \check{x}$} \}. 
\end{equation*}
We now divide the argument into two cases. First, suppose that every vertex in $\supp_0(f)_t$ is not a terminal vertex for $\supp_0(g)$ so that
\begin{equation} \label{equation: Orbit Union}
{\bigcup}_{\check{x} \in \supp_0(g)_t} O_{\check{x}} = \supp_0(f)_t. 
\end{equation}
From here, we now employ a similar argument to the one used in the proof of additivity. If $\check{x} \in \supp_0(g)_t$ then by \eqref{equation: Subgraph Cuntz Semigroup}, $$\Cu(I_{L(x)}) \cong \Cu(I_{H_x}) \stack \Cu(C^*(L(x) \backslash H_x)).$$ Observe that if $\check{v} \in O_{\check{x}}$ then $\Delta_E(\theta_{\check{v}, f(\check{v})}) \in \Cu(I_{H_x})$. Hence by definition of the stack, 
\begin{equation} \label{equation: Not Terminal Absorption}
\Delta_E(\theta_{\check{x}, g(\check{x})}) + {\sum}_{\check{v} \in O_{\check{x}}} \Delta_E(\theta_{\check{v}, f(\check{v})}) = \Delta_E(\theta_{\check{x}, g(\check{x})})
\end{equation}
in $\Cu(I_{L(x)})$ and thus, in $\Cu(C^*(E))$ as well. Since any Cuntz class in $\Cu(I_{H_x})$ is way-below any Cuntz class in $\Cu(I_{L(x)}) \backslash \Cu(I_{H_x})$, we also have 
\begin{equation} \label{equation: Not Terminal Inequality}
{\sum}_{\check{v} \in O_{\check{x}}} \Delta_E(\theta_{\check{v}, f(\check{v})}) \ll \Delta_E(\theta_{\check{x}, g(\check{x})})
\end{equation}
in $\Cu(C^*(E))$. Summing over all $\check{x} \in \supp_0(g)_t$ in \eqref{equation: Not Terminal Inequality}, we obtain $\Delta_E(f) \leq \Delta_E(g)$ in this case by \eqref{equation: Delta E Singleton} and \eqref{equation: Orbit Union}. \newline
\phantom{} \newline
Second, assume that there exists a vertex $\check{v} \in \supp_0(f)_t$ which is also a terminal vertex for $\supp_0(g)$. If $c(\check{v}) \neq 1$ then by the assumption that $f \leq g$,
\begin{equation} \label{equation: Delta E Order No Loop}
f(\check{v}) \cdot [p_v] + {\sum}_{w \in L(v), \; v \not \geq w} \infty \cdot [p_w] \leq g(\check{v}) \cdot [p_v] + {\sum}_{w \in L(v), \; v \not \geq w} \infty \cdot [p_w].
\end{equation} 
Note that if $c(\check{v}) = \infty$ then \eqref{equation: Delta E Order No Loop} follows from the fact that $f(\check{v}) = g(\check{v}) = \infty$. In the case where $c(\check{v}) = 0$, \eqref{equation: Delta E Order No Loop} follows from Lemma \ref{lemma: Order Alternative}. So by definition of $\Delta_E$, $\Delta_E(\theta_{\check{v}, f(\check{v})}) \leq \Delta_E(\theta_{\check{v}, g(\check{v})})$ in $\Cu(C^*(E))$. \newline
\phantom{} \newline
Next, if $c(\check{v}) = 1$ then by Lemma \ref{lemma: Order Alternative} again, 
\begin{equation*}
{\sum}_{i = 1}^{\infty} \chi_{\{ f(\check{v}) \geq i \}} = f(\check{v}) \leq g(\check{v}) = {\sum}_{i = 1}^{\infty} \chi_{\{ g(\check{v}) \geq i \}}
\end{equation*}
in $\Lsc(\mathbb{T}, \overline{\mathbb{N}})$. Subsequently by Lemma \ref{lemma: Same Open Subset Upstairs} and the isomorphism $C^*(L(v)) \cong p_{L(v)} I_{L(v)} p_{L(v)}$,
\begin{equation*}
{\sum}_{i = 1}^{\infty} \big[h_{v, \; \{ f(\check{v}) \geq i \}} \big] \leq {\sum}_{i = 1}^{\infty} \big[h_{v, \; \{ g(\check{v}) \geq i \}} \big]
\end{equation*}  
in $\Cu(C^*(E))$ and $\Delta_E(\theta_{\check{v}, f(\check{v})}) \leq \Delta_E(\theta_{\check{v}, g(\check{v})})$ once more. Hence, by \eqref{equation: Delta E Singleton} and \eqref{equation: Not Terminal Absorption},
\begin{IEEEeqnarray*}{rCl}
& & \Delta_E(f) = {\sum}_{\check{w} \in \supp_0(f)_t}  \Delta_E(\theta_{\check{w}, f(\check{w})})
\\
& \leq & {\sum}_{\check{w} \in \supp_0(f)_t \cap \supp_0(g)_t} \Delta_E(\theta_{\check{w}, f(\check{w})}) + {\sum}_{\check{x} \in \supp_0(g)_t} {\sum}_{\check{v} \in O_{\check{x}}} \Delta_E(\theta_{\check{v}, f(\check{v})})
\\
& \leq & {\sum}_{\check{w} \in \supp_0(f)_t \cap \supp_0(g)_t} \Delta_E(\theta_{\check{w}, g(\check{w})}) + {\sum}_{\check{x} \in \supp_0(g)_t} {\sum}_{\check{v} \in O_{\check{x}}} \Delta_E(\theta_{\check{v}, f(\check{v})})
\\
& \leq & {\sum}_{\check{w} \in \supp_0(g)_t} \Delta_E(\theta_{\check{w}, g(\check{w})}) + {\sum}_{\check{x} \in \supp_0(g)_t} {\sum}_{\check{v} \in O_{\check{x}}} \Delta_E(\theta_{\check{v}, f(\check{v})})
\\
& = & {\sum}_{\check{w} \in \supp_0(g)_t} \Delta_E(\theta_{\check{w}, g(\check{w})}) = \Delta_E(g). 
\end{IEEEeqnarray*} 
Therefore, $\Delta_E$ is order-preserving. The proof that $\Delta_E$ preserves way-below is done similarly. Establishing this in the first case is done by summing over all $\check{x} \in \supp_0(g)_t$ in \eqref{equation: Not Terminal Inequality} and using \ref{axiom: O3}. In the second case, the argument is nearly identical to the argument for order preservation, using Lemma \ref{lemma: Way Below} in lieu of Lemma \ref{lemma: Order Alternative}. The key difference between the two arguments is that the way-below version of \eqref{equation: Delta E Order No Loop} is not immediate if $c(\check{v}) = \infty$. However, since $v$ supports multiple cycles then by the Cuntz-Krieger relations, $\infty \cdot [p_v] = [p_v]$ in $\Cu(C^*(E))$ and 
\begin{IEEEeqnarray*}{rCl}
f(\check{v}) \cdot [p_v] + {\sum}_{w \in L(v), \; v \not \geq w} \infty \cdot [p_w] & = & \infty \cdot [p_v] + {\sum}_{w \in L(v), \; v \not \geq w} \infty \cdot [p_w]
\\
& = & [p_v] + {\sum}_{w \in L(v), \; v \not \geq w} \infty \cdot [p_w] 
\\
& = & [p_v] \ll [p_v] 
\\
& \leq & g(\check{v}) \cdot [p_v] + {\sum}_{w \in L(v), \; v \not \geq w} \infty \cdot [p_w]. 
\end{IEEEeqnarray*} 
Finally, to see that $\Delta_E$ preserves suprema of increasing sequences, assume that $\{ f_n \}_{n > 0}$ is an increasing sequence in $\Cu_E$ with supremum $f$. Since $S(E)$ only has finitely many hereditary subsets, there exists $k > 0$ such that if $\ell \geq 0$ then $$\supp_0(f_k) = \supp_0(f_{k+\ell}) = \supp_0(f).$$ Now if $\check{v} \in \supp_0(f)_t$ then $f(\check{v}) = \sup_{n \geq k} f_n(\check{v})$. In particular, in the case where $c(\check{v}) = 1$, this claim follows from the fact that by Lemma \ref{lemma: Same Open Subset Upstairs} and the isomorphism $C^*(L(v)) \cong p_{L(v)} I_{L(v)} p_{L(v)}$,
\begin{equation*}
{\sum}_{i = 1}^{\infty} \big[h_{v, \; \{ f(\check{v}) \geq i \}} \big] = {\sup}_{n \geq k} {\sum}_{i = 1}^{\infty} \big[h_{v, \; \{ f_n(\check{v}) \geq i \}} \big]
\end{equation*}
in $\Cu(C^*(E))$. Since $\Cu(C^*(E))$ satisfies \ref{axiom: O4} and $\Delta_E$ is order-preserving, $\Delta_E(f) = \sup_{n \geq k} \Delta_E(f_n) = \sup_{n > 0} \Delta_E(f_n)$. So $\Delta_E$ preserves suprema of increasing sequences and is therefore a Cu-morphism. 
\end{proof}

\subsection{Proof of Theorem \ref{introtheorem: Cu E Iso}} 
\label{subsec: Proof Main Theorem}

Recall that if $E$ is a finite graph satisfying the standing assumptions and $H$ is a hereditary subset of $E^0$, then we have short exact sequences of positively ordered monoids in \eqref{equation: Cu(C*(E)) SES} and \eqref{equation: Cu E SES} which involve $\Cu(C^*(E))$ and $\Cu_E$ respectively. We now show that the Cu-morphism $\Delta_E$ defined in Definition \ref{definition: Delta E} intertwines both short exact sequences. The resulting commutative squares will be extensively used in the proof of Theorem \ref{introtheorem: Cu E Iso}. 

\begin{lemma} \label{lemma: Ideal Square}
Let $E$ be a finite graph satisfying the standing assumptions and $H$ be a hereditary subset of $E^0$. The following diagram commutes:
\begin{equation*}
\begin{tikzcd}
\Cu_H \arrow[r, "i_H"] \arrow[d, swap, "\Delta_H"] & \Cu_E \arrow[d, "\Delta_E"] \\
\Cu(C^*(H)) \arrow[r, swap, "\Cu(\iota)"] & \Cu(C^*(E))
\end{tikzcd}
\end{equation*}
where $i_H$ is the Cu-morphism in \eqref{equation: Inclusion} and $\iota$ is the inclusion map $I_H \hookrightarrow C^*(E)$. 
\end{lemma}

\begin{proof}
Assume that $\check{v} \in \langle H \rangle$ and $U \subseteq \mathbb{T}$ be non-empty and open. From the definitions, it follows from direct computation that if $c(\check{v}) = 0$ then $$(\Cu(\iota) \circ \Delta_H)(\theta_{\check{v}, 1}) = (\Delta_E \circ i_H)(\theta_{\check{v}, 1}) = [p_v],$$ if $c(\check{v}) = \infty$ then $$(\Cu(\iota) \circ \Delta_H)(\theta_{\check{v}, \infty}) = (\Delta_E \circ i_H)(\theta_{\check{v}, \infty}) = \infty \cdot [p_v] + {\sum}_{w \in L(v), \; v \not \geq w} \infty \cdot [p_w]$$ and if $c(\check{v}) = 1$ then $$(\Cu(\iota) \circ \Delta_H)(\theta_{\check{v}, \chi_U}) = (\Delta_E \circ i_H)(\theta_{\check{v}, \chi_U}) = [h_{v, U}] + {\sum}_{w \in L(v), \; v \not \geq w} \infty \cdot [p_w].$$ By Lemma \ref{lemma: Singleton Decomp} and Proposition \ref{proposition: Delta E Cu Morphism}, the square commutes. 
\end{proof}

\begin{lemma} \label{lemma: Quotient Square}
Let $E$ be a finite graph satisfying the standing assumptions and $H$ be a hereditary subset of $E^0$. The following diagram commutes:
\begin{equation*}
\begin{tikzcd}
\Cu_E \arrow[r, "\rho_H"] \arrow[d, swap, "\Delta_E"] & \Cu_{E \backslash H} \arrow[d, "\Delta_{E \backslash H}"] \\
\Cu(C^*(E)) \arrow[r, swap, "\Cu(\pi)"] & \Cu(C^*(E \backslash H))
\end{tikzcd}
\end{equation*}
where $\rho_H$ is the Cu-morphism constructed in Lemma \ref{lemma: Cu E Minus H} and $\pi: C^*(E) \rightarrow C^*(E \backslash H)$ is the quotient map.
\end{lemma}

\begin{proof}
Assume that $\check{v} \in \langle E \rangle$ and $U \subseteq \mathbb{T}$ is non-empty and open. We may assume without loss of generality that $\check{v} \not \in \langle H \rangle$, as otherwise, there is nothing to show. Once again from the definitions, it follows by direct computation that if $c(\check{v}) = 0$ then $$(\Cu(\pi) \circ \Delta_E)(\theta_{\check{v}, 1}) = (\Delta_{E \backslash H} \circ \rho_H)(\theta_{\check{v}, 1}) = [\pi(p_v)]$$ and if $c(\check{v}) = \infty$ then
\begin{IEEEeqnarray*}{rCl}
& & (\Cu(\pi) \circ \Delta_E)(\theta_{\check{v}, \infty}) = (\Delta_{E \backslash H} \circ \rho_H)(\theta_{\check{v}, \infty}) 
\\
& = & \infty \cdot [\pi(p_v)] + {\sum}_{w \in L_{E \backslash H}(v), \; v \not \geq w} \infty \cdot [\pi(p_w)]
\end{IEEEeqnarray*} 
where $L_{E \backslash H}(v) = \{ x \in (E \backslash H)^0 \; \vert \; x \geq v \}$. Now if $c(\check{v}) = 1$ and $g$ is the single loop based at $v$ then 
\begin{equation*}
\Delta_E(\theta_{\check{v}, \chi_U}) = [h_{v, U}] + {\sum}_{w \in L_E(v) \backslash \{ v \}} \infty \cdot [p_w]
\end{equation*}
where $L_E(v) = \{ x \in E^0 \; \vert \; x \geq v \}$ and $h_{v, U}$ is a strictly positive element of the hereditary subalgebra $$p_{L_E(v)} \big( I_{L_E(v)} \cap \mathcal{R}_{E, v, \mathbb{T} \backslash U}  \big) p_{L_E(v)} .$$ For clarity, we have modified the notation established in Definition \ref{definition: Primitive Non Gauge Invariant} slightly to include the graph, emphasising that the above subalgebra is a subalgebra of $C^*(E)$. Now by combining \cite[II.5.1.4 (ii)]{Bla06} with surjectivity of $\pi$, we have that
\begin{IEEEeqnarray*}{rCl}
& & \pi \big( p_{L_E(v)}  \big( I_{L_E(v)} \cap \mathcal{R}_{E, v, \mathbb{T} \backslash U} \big) p_{L_E(v)}  \big)
\\
& = & \pi(p_{L_E(v)}) \pi \big( I_{L_E(v)} \cap \mathcal{R}_{E, v, \mathbb{T} \backslash U} \big)  \pi(p_{L_E(v)})
\\
& = & \pi(p_{L_E(v)}) \big( \pi(I_{L_E(v)}) \cap \pi(\mathcal{R}_{E, v, \mathbb{T} \backslash U}) \big) \pi(p_{L_E(v)})
\\
& = & p_{L_{E \backslash H}(v)} \big( I_{L_{E \backslash H}(v)} \cap \mathcal{R}_{E \backslash H, v, \mathbb{T} \backslash U} \big) p_{L_{E \backslash H}(v)}. 
\end{IEEEeqnarray*} 
Consequently, $\pi(h_{v, U})$ is a strictly positive element of $$p_{L_{E \backslash H}(v)} \big( I_{L_{E \backslash H}(v)} \cap \mathcal{R}_{E \backslash H, v, \mathbb{T} \backslash U} \big) p_{L_{E \backslash H}(v)}$$ and by definition of $\Delta_{E \backslash H}$,
\begin{IEEEeqnarray*}{rCl}
(\Cu(\pi) \circ \Delta_E)(\theta_{\check{v}, \chi_U})  & = & \Cu(\pi)\Big( [h_{v, U}] + {\sum}_{w \in L_E(v) \backslash \{ v \}} \infty \cdot [p_w] \Big)
\\
& = & [\pi(h_{v, U})] + {\sum}_{w \in L_{E \backslash H}(v) \backslash \{ v \}} \infty \cdot [\pi(p_w)]
\\
& = & (\Delta_{E \backslash H} \circ \rho_H)(\theta_{\check{v}, \chi_U}). 
\end{IEEEeqnarray*} 
Again by Lemma \ref{lemma: Singleton Decomp} and Proposition \ref{proposition: Delta E Cu Morphism}, the square commutes. 
\end{proof}

We are now in a position to prove Theorem \ref{introtheorem: Cu E Iso} by induction over the gauge-invariant ideal lattice. Our goal is to prove that for a finite graph $E$ satisfying the standing assumptions, $\Delta_E: \Cu_E \rightarrow \Cu(C^*(E))$ is a Cu-isomorphism. We start by proving Theorem \ref{introtheorem: Cu E Iso} in the base case where $C^*(E)$ has gauge length one.  

\begin{lemma} \label{lemma: Base Case}
Let $E$ be a finite graph satisfying the standing assumptions. Suppose that $C^*(E)$ has gauge length one. Then the Cu-morphism $\Delta_E: \Cu_E \rightarrow \Cu(C^*(E))$ is an isomorphism.
\end{lemma}

\begin{proof}
By Lemma \ref{lemma: Gauge Length}, the skeleton $S(E)$ has exactly one vertex, which we denote by $\check{v}$. There are now three cases to consider. If $c(\check{v}) = 0$ then $E$ consists of a single source vertex $v$. By definition of $\Delta_E$ (see Definition \ref{definition: Delta E}) and the fact that $\Cu(C^*(E)) \cong \overline{\mathbb{N}}$, it is clear that $\Delta_E$ is a Cu-isomorphism in this case. \newline
\phantom{} \newline
Next, suppose that $c(\check{v}) = \infty$. By the proof of Lemma \ref{lemma: Gauge Simple}, $C^*(E)$ is simple purely infinite and the map $\Delta_E$ is explicitly given by
\begin{equation*}
\begin{matrix}
\Delta_E: & \Cu_E & \rightarrow & \Cu(C^*(E)) \\
& \big(\check{v} \mapsto \infty \big) & \mapsto & \infty \cdot [p_v]. 
\end{matrix}
\end{equation*}
Since the only non-zero element of $\Cu_E$ is the function $\check{v} \mapsto \infty$ and $\Cu(C^*(E))$ is isomorphic to $\{ 0, \infty \}$, $\Delta_E$ is a Cu-isomorphism in this case. \newline
\phantom{} \newline
Finally, suppose that $c(\check{v}) = 1$. By standing assumption \ref{property: S3}, $E$ consists of a single vertex $v$ which supports a single loop and so $C^*(E) \cong C(\mathbb{T})$. If $U \subseteq \mathbb{T}$ is non-empty and open then $\Delta_E$ is explicitly the map
\begin{equation*}
\begin{matrix}
\Delta_E: & \Cu_E & \rightarrow & \Cu(C^*(E)) \\
& \theta_{\check{v}, \chi_U} & \mapsto & [h_{v, U}]
\end{matrix}
\end{equation*}
where $h_{v, U}$ is a strictly positive element of the ideal $\mathcal{R}_{v, \mathbb{T} \backslash U}$ (see Definition \ref{definition: Strictly Positive}). The isomorphism $\Cu(C^*(E)) \cong \Lsc(\mathbb{T}, \overline{\mathbb{N}})$ sends $[h_{v, U}]$ to $\chi_U$. Consequently, $\Delta_E$ is a Cu-isomorphism.
\end{proof}

Before we embark on the proof of Theorem \ref{introtheorem: Cu E Iso}, we first describe its overall architecture. Similarly to the proof of \cite[Theorem 7.1]{AMP07}, we induct over the gauge-invariant ideal lattice of $C^*(E)$. The arguments that $\Delta_E$ is an order-embedding and surjective divide into two cases, depending on whether there exists a proper hereditary subset $H$ of $E^0$ such that a certain element of $\Cu(C^*(E))$ lies in $\Cu(I_H)$. If there exists such a proper hereditary subset $H$, then the desired conclusions of both arguments follow from the inductive hypothesis and Lemma \ref{lemma: Ideal Square}. \newline
\phantom{} \newline
On the other hand, if there is no such proper hereditary subset $H$, then we apply Lemma \ref{lemma: Quotient Square} in the case where $\pi$ is the quotient map by a gauge-simple, gauge-invariant ideal of $C^*(E)$. This allows us to invoke the inductive hypothesis for the quotient graph C*-algebra. To conclude that $\Delta_E$ is an order-embedding in this case, it is enough to use the specific properties of $\Cu_E$. \newline
\phantom{} \newline
Proving that $\Delta_E$ is surjective in this situation is a more delicate affair: one splits into two further cases, depending on whether the gauge-simple, gauge-invariant ideal of $C^*(E)$ is purely infinite or has stable rank one. In the former case, one applies Lemma \ref{lemma: Polar Part Absorbed} to obtain the desirable absorption behaviour needed to prove surjectivity. In the latter case, one uses Theorem \ref{theorem: SR1 Corner Positive Cancellation} to obtain the key inequality needed for surjectivity. 

\begin{proof}[Proof of Theorem \ref{introtheorem: Cu E Iso}]
First observe that it suffices to prove Theorem \ref{introtheorem: Cu E Iso} in the special case where $E$ consists of a single connected component. Indeed, suppose that $E = E_1 \sqcup E_2 \sqcup \dots \sqcup E_m$ where $m > 1$ and for $j \in \{1, 2, \dots, m \}$, $E_j$ is a connected subgraph. It is straightforward to check that 
\begin{equation*}
\Cu_E \cong {\bigoplus}_{j = 1}^m \Cu_{E_j} \qquad \text{and} \qquad \Delta_E = \Delta_{E_1} \oplus \dots \oplus \Delta_{E_m}. 
\end{equation*}
So if Theorem \ref{introtheorem: Cu E Iso} is true for graphs consisting of a single connected component then $\Delta_{E_j}$ is a Cu-isomorphism from $\Cu_{E_j}$ to $\Cu(C^*(E_j))$ for $j \in \{1, 2, \dots, m \}$. Therefore, $\Delta_E$ is an isomorphism from $\Cu_E$ to $${\bigoplus}_{j = 1}^m \Cu(C^*(E_j)) \cong \Cu(C^*(E))$$ which gives Theorem \ref{introtheorem: Cu E Iso} in general. \newline
\phantom{} \newline
Without loss of generality, assume that $E$ consists of a single connected component. We will prove that $\Delta_E$ is a Cu-isomorphism. The base case where $C^*(E)$ has gauge length one is done by Lemma \ref{lemma: Base Case}. For the inductive hypothesis, assume that there exists $k > 0$ such that for a finite graph $F$ which satisfies the standing assumptions, has gauge length at most $k$ and consists of a single connected component, $\Delta_F$ is a Cu-isomorphism. Suppose that $C^*(E)$ has gauge length $k + 1$. Then there exists an ascending chain of gauge-invariant ideals
\begin{equation*}
\{ 0 \} = I_0 \subseteq I_1 \subseteq \dots \subseteq I_k \subseteq I_{k+1} = C^*(E)
\end{equation*}
such that if $j \in \{0, 1, \dots, k \}$ and $H_j$ is the hereditary subset of $E^0$ corresponding to $I_j$ then $C^*(H_{j+1} \backslash H_j)$ is gauge-simple. In particular, $C^*(H_1)$ is gauge-simple. Let $\pi: C^*(E) \rightarrow C^*(E \backslash H_1)$ be the quotient map. \newline
\phantom{} \newline
$\Delta_E$ is an \textbf{order-embedding:} \newline
\phantom{} \newline
Let $f, g \in \Cu_E$ be such that $\Delta_E(f) \leq \Delta_E(g)$. First, suppose that there exists a proper hereditary subset $H$ of $E^0$ such that $\Delta_E(g) \in \Cu(I_H)$. Then $\Delta_E(f) \in \Cu(I_H)$. By definition of $\Delta_E$ (see Definition \ref{definition: Delta E}), if $\check{v} \in \supp_0(f)_t$ and $c(\check{v}) \neq 1$ then $[p_v] \in \Cu(I_H)$ and $v \in H$. If $c(\check{v}) = 1$ and $U \subseteq \mathbb{T}$ is non-empty and open then by Definition \ref{definition: Strictly Positive}, $$h_{v, U} \in p_{L(v)} (I_{L(v)} \cap \mathcal{R}_{v, U}) p_{L(v)} \subseteq p_{L(v)} I_{L(v)} p_{L(v)} \subseteq I_H.$$ The last inclusion follows from the fact that $I_{L(v)}$ is the smallest gauge-invariant ideal of $C^*(E)$ containing $I_{L(v)} \cap \mathcal{R}_{v, U}$; see \eqref{equation: Sandwich}. Once again, $$p_v \in p_{L(v)} I_{L(v)} p_{L(v)}  \subseteq I_H$$ and $v \in H$. Consequently, $\supp_0(f) \subseteq \langle H \rangle$. Similarly, $\supp_0(g) \subseteq \langle H \rangle$. Since $H$ is a proper hereditary subset of $E^0$ then by Lemma \ref{lemma: Gauge Length}, $C^*(H)$ has gauge length $\lvert S(H)^0 \rvert < \lvert S(E)^0 \rvert = k+1$. So if $H$ consists of a single connected component then by the inductive hypothesis, $\Delta_H$ is a Cu-isomorphism from $\Cu_H$ to $\Cu(C^*(H)) \cong \Cu(I_H)$. On the other hand, if $H$ consists of multiple disjoint connected subgraphs then the inductive hypothesis tells that on each connected subgraph $G$ of $H$, $\Delta_G$ is a Cu-isomorphism from $\Cu_G$ to $\Cu(C^*(G))$. Then by taking direct sums, $\Delta_H$ is still a Cu-isomorphism from $\Cu_H$ to $\Cu(C^*(H))$. \newline
\phantom{} \newline
Using \eqref{equation: Ideal In Terms Of Support}, we may view $f$ and $g$ as elements of $\Cu_H$. So $(\Delta_E \circ i_H)(f) \leq (\Delta_E \circ i_H)(g)$ and by Lemma \ref{lemma: Ideal Square}, $(\Cu(\iota) \circ \Delta_H)(f) \leq (\Cu(\iota) \circ \Delta_H)(g)$. In particular, $\Delta_H(f) \leq \Delta_H(g)$ in $\Cu(C^*(E))$ and hence, in $\Cu(I_H) \cong \Cu(C^*(H))$ as well. Since $\Delta_H$ is a Cu-isomorphism then $f \leq g$ in $\Cu_H$ and by applying the map $i_H$, $f \leq g$ in $\Cu_E$. \newline
\phantom{} \newline
Now assume that there is no proper hereditary subset $H$ of $E^0$ such that $\Delta_E(g) \in \Cu(I_H)$. Then $(\Cu(\pi) \circ \Delta_E)(f) \leq (\Cu(\pi) \circ \Delta_E)(g)$ and by Lemma \ref{lemma: Quotient Square}, $(\Delta_{E \backslash H_1} \circ \rho_{H_1})(f) \leq (\Delta_{E \backslash H_1} \circ \rho_{H_1})(g)$. By the inductive hypothesis, $\Delta_{E \backslash H_1}$ is a Cu-isomorphism. So $\rho_{H_1}(f) \leq \rho_{H_1}(g)$ in $\Cu_{E \backslash H_1}$. \newline
\phantom{} \newline
Observe that by Remark \ref{remark: Delta E Ideal} and Lemma \ref{lemma: Gauge Length}, $\supp_0(g) = S(E)^0$ and $S(H_1)$ consists of a single vertex $\check{w}$ respectively. If $\check{w}$ is a terminal vertex for $\supp_0(g) = S(E)^0$ and $v \in E^0$ satisfies $w \geq v$ then $v \geq w$ and $v \in H_1$. Coupled with the fact that $H_1$ is hereditary, we see that $H_1$ is an isolated subgraph of $E$. Since $E$ consists of exactly one connected component then by the contrapositive, $\check{w}$ cannot be a terminal vertex for $\supp_0(g)$. So by Definition \ref{definition: Cu E}, $g(\check{w}) = \infty$. Together with the Cu-isomorphism $\Cu_{E \backslash H_1} \cong \Cu_E /\Cu_{H_1}$ from Proposition \ref{proposition: First Iso Thm Cu E} and the inequality $\rho_{H_1}(f) \leq \rho_{H_1}(g)$, there exists $h \in \Cu_E$ such that $\supp_0(h) \subseteq \langle H_1 \rangle = \{ \check{w} \}$ and $f \leq g + h = g$ in $\Cu_E$. Therefore, $\Delta_E$ is an order-embedding. \newline
\phantom{} \newline
$\Delta_E$ is \textbf{surjective:} \newline
\phantom{} \newline
Assume that $[x] \in \Cu(C^*(E))$. Firstly, if there exists a proper hereditary subset $H$ of $E^0$ such that $[x] \in \Cu(I_H) \cong \Cu(C^*(H))$ then by the inductive hypothesis, $\Delta_H$ is a Cu-isomorphism and there exists $f \in \Cu_H$ such that $\Delta_H(f) = [x]$. By Lemma \ref{lemma: Ideal Square}, $i_H(f) = f \in \Cu_E$ satisfies
\begin{equation*}
\Delta_E(f) = (\Delta_E \circ i_H)(f) = (\Cu(\iota) \circ \Delta_H)(f) = [\iota(x)] = [x]. 
\end{equation*}
On the other hand, suppose that there is no such proper hereditary subset $H$. If there exists a proper hereditary subset $H'$ of $(E \backslash H_1)^0 = E^0 \backslash H_1$ such that $\Cu(\pi)([x]) = [\pi(x)] \in \Cu(I_{H'})$ then $[\pi(x)] \leq \infty \cdot {\sum}_{v \in H'} [\pi(p_v)]$ and by Theorem \ref{theorem: CRS}, $$[x] \leq \infty \cdot \Big[{\sum}_{v \in H' \sqcup H_1} p_v \Big]$$ in $\Cu(C^*(E))$. Thus, $H' \sqcup H_1$ is a proper hereditary subset of $E^0$ such that $[x] \in \Cu(I_{H' \sqcup H_1})$. \newline
\phantom{} \newline
By the contrapositive, there is no proper hereditary subset $H'$ of $(E \backslash H_1)^0$ such that $[\pi(x)] \in \Cu(I_{H'})$. Since $\Delta_{E \backslash H_1}$ is a Cu-isomorphism by the inductive hypothesis, there exists $g \in \Cu_{E \backslash H_1}$ such that $\Delta_{E \backslash H_1}(g) = [\pi(x)]$. So by Remark \ref{remark: Delta E Ideal} and Lemma \ref{lemma: Cu E Minus H}, $\supp_0(g) = S(E \backslash H_1)^0$ and $$g' = {\sum}_{\check{v} \in \supp_0(g)_t} \theta_{\check{v}, g(\check{v})} + \infty \cdot \theta_{\check{w}, \delta_{\check{w}}}$$ is an element of $\Cu_E$ such that $\supp_0(g') = S(E)^0$ and $\rho_{H_1}(g') = g$ (recall the definition of $\delta_{\check{w}}$ from \eqref{equation: Delta V}). By Lemma \ref{lemma: Quotient Square}, 
\begin{equation} \label{equation: Need To Lift}
(\Cu(\pi) \circ \Delta_E)(g') = (\Delta_{E \backslash H_1} \circ \rho_{H_1})(g') = \Delta_{E \backslash H_1}(g) = \Cu(\pi)([x]). 
\end{equation}
We claim that $\Delta_E(g') = [x]$. Since $\Cu(\pi)([x]) \leq \Cu(\pi)(\Delta_E(g'))$ then by Theorem \ref{theorem: CRS}, 
\begin{IEEEeqnarray*}{rCl}
[x] & \leq & \Delta_E(g') + \infty \cdot [p_w]
\\
& = & \Delta_E\Big({\sum}_{\check{v} \in \supp_0(g)_t} \theta_{\check{v}, g(\check{v})} \Big) + \Delta_E(\infty \cdot \theta_{\check{w}, \delta_{\check{w}}}) + \infty \cdot [p_w]
\\
& = & \Delta_E\Big({\sum}_{\check{v} \in \supp_0(g)_t} \theta_{\check{v}, g(\check{v})} \Big) + \Delta_E(\infty \cdot \theta_{\check{w}, \delta_{\check{w}}}) + \Delta_E(\infty \cdot \theta_{\check{w}, \delta_{\check{w}}})
\\
& = & \Delta_E\Big({\sum}_{\check{v} \in \supp_0(g)_t} \theta_{\check{v}, g(\check{v})} \Big) + \Delta_E(\infty \cdot \theta_{\check{w}, \delta_{\check{w}}}) = \Delta_E(g').
\end{IEEEeqnarray*} 
To show that the reverse inequality holds, we will take advantage of the fact that $I_1$ is gauge-simple. There are two possible cases to consider. First, assume that $I_1$ is simple purely infinite. Then $C^*(H_1)$ is simple purely infinite and subsequently, $[p_w]$ is a properly infinite element of $\Cu(C^*(E))$ by Theorem \ref{theorem: Purely Infinite Properly Infinite}. \newline
\phantom{} \newline
We will now show that $[x] + [p_w] = [x]$. For simplicity, first assume that there exists $n > 0$ such that $x \in (C^*(E) \otimes M_n(\mathbb{C}))_+$. Since the lone vertex $\check{w}$ of $S(H_1)$ is not a terminal vertex for $S(E)^0$ then by Lemma \ref{lemma: Hereditary Terminal}, there exists $\check{v} \in S(E)^0$ such that $\check{v}$ is terminal and $\check{w} \geq \check{v}$. Let $1_n$ denote the unit of $M_n(\mathbb{C})$. Then 
\begin{equation} \label{equation: Not In Ideal Decomp}
\begin{split}
x & = (p_{R(v)} \otimes 1_n) x (p_{R(v)} \otimes 1_n) + (p_{E^0 \backslash R(v)} \otimes 1_n) x (p_{R(v)} \otimes 1_n)   \\
& + (p_{R(v)} \otimes 1_n)  x (p_{E^0 \backslash R(v)} \otimes 1_n) + (p_{E^0 \backslash R(v)} \otimes 1_n) x (p_{E^0 \backslash R(v)} \otimes 1_n).
\end{split}
\end{equation}
Since $R(v) = \{ a \in E^0 \; \vert \; v \geq a \}$ and $\check{v}$ is terminal for $S(E)^0$ then $R(v) \subseteq L(v)$ and subsequently, $$(p_{R(v)} \otimes 1_n) x (p_{R(v)} \otimes 1_n) \in (I_{L(v)} \otimes M_n(\mathbb{C}))_+.$$  Now $v$ cannot be a source vertex because $w \geq v$ and $w \neq v$. So by standing assumptions \ref{property: S1} and \ref{property: S2}, $v$ supports at least one cycle. By Theorem \ref{introtheorem: General Cuntz Krieger}, the set $G = \{ x \in L(v) \; \vert \; v \not \geq x \}$ is a hereditary subset of $L(v)$ and
\begin{equation} \label{equation: Cuntz Stack Computation}
\Cu(I_{L(v)}) \cong \Cu(I_G) \stack \Cu(C^*(L(v) \backslash G))
\end{equation}
where $\Cu(C^*(L(v) \backslash G))$ is either $\{0, \infty \}$ or $\Lsc(\mathbb{T}, \overline{\mathbb{N}})$. Since $G \subseteq E^0 \backslash R(v)$ then $$(p_{R(v)} \otimes 1_n) x (p_{R(v)} \otimes 1_n) \not \in I_{G} \otimes M_n(\mathbb{C})$$ as otherwise by \eqref{equation: Not In Ideal Decomp}, $x \in I_{E^0 \backslash R(v)} \otimes M_n(\mathbb{C})$, which would contradict our assumption that there is no proper hereditary subset $H$ of $E^0$ such that $[x] \in \Cu(I_H)$. So $$[(p_{R(v)} \otimes 1_n) x (p_{R(v)} \otimes 1_n)] \in \Cu(I_{L(v)}) \backslash \Cu(I_G).$$ Now let $e_{11} \in M_n(\mathbb{C})$ be the matrix with a 1 in the top left corner and zeros elsewhere. By \eqref{equation: Cuntz Stack Computation} and the fact that $[p_w] \in \Cu(I_G)$,
\begin{equation*}
[p_w] = [p_w \otimes e_{11}] \leq [(p_{R(v)} \otimes 1_n) x (p_{R(v)} \otimes 1_n)]
\end{equation*}
in the ideal $\Cu(I_{L(v)})$ and hence, in $\Cu(C^*(E))$ as well. Since $[p_w]$ is properly infinite and $[p_w] \leq [x]$ then by Lemma \ref{lemma: Polar Part Absorbed} with $A = C^*(E) \otimes K$, $$x \oplus (p_w \otimes e_{11}) \precsim x.$$ 
So we have proved that if $n > 0$ and $x \in (C^*(E) \otimes M_n(\mathbb{C}))_+$ then $[x] + [p_w] = [x]$ in $\Cu(C^*(E))$. This is extended to the case where $x \in (C^*(E) \otimes K)_+$ by first choosing $\epsilon > 0$ small enough so that there is no proper hereditary subset $H$ of $E^0$ such that $[(x - \epsilon)_+] \in \Cu(I_H)$ and then using Rørdam's lemma (see \cite[Proposition 2.4]{Ror92}) to pick $n > 0$ and $a \in (C^*(E) \otimes M_n(\mathbb{C}))_+$ such that $[(x - \epsilon)_+] = [a]$. \newline
\phantom{} \newline
Now to see that $\Delta_E(g') \leq [x]$ in $\Cu(C^*(E))$, recall from \eqref{equation: Need To Lift} that $$\Cu(\pi)(\Delta_E(g')) \leq \Cu(\pi)([x])$$ in $\Cu(C^*(E \backslash H_1))$. By Theorem \ref{theorem: CRS} and the fact that $[x]$ absorbs $[p_w]$, 
\begin{equation*}
\Delta_E(g') \leq [x] + \infty \cdot [p_w] = [x]. 
\end{equation*}
Therefore, we have $\Delta_E(g') = [x]$ in this case. \newline
\phantom{} \newline
For the second case, assume that $I_1$ has stable rank one. By the proof of Lemma \ref{lemma: Axiom O2}, there exists a rapidly increasing sequence $\{ h_n \}_{n > 0}$ in $\Cu_E$ such that for $n > 0$, $\supp_0(h_n) = \supp_0(g') = S(E)^0$ and $\sup_{n > 0} h_n = g'$. Now if $k > 0$ then by Theorem \ref{theorem: CRS} and the fact that $\Delta_E$ is a Cu-morphism (see Proposition \ref{proposition: Delta E Cu Morphism}),
\begin{equation*}
\Delta_E(h_k) \ll \Delta_E(g') \leq [x] + \infty \cdot [p_w]. 
\end{equation*}
So there exists $\ell > 0$ such that $\Delta_E(h_k) \leq [x] + \ell \cdot [p_w]$. Now since $\check{w}$ is not a terminal vertex for $\supp_0(h_k) = S(E)^0$ then $h_k(\check{w}) = \infty$ and 
\begin{equation*}
\Delta_E(h_k) + \ell \cdot [p_w] = \Delta_E(h_k + \ell \cdot \theta_{\check{w}, \delta_{\check{w}}}) = \Delta_E(h_k) \leq  [x] + \ell \cdot [p_w]. 
\end{equation*}
By the proof of Lemma \ref{lemma: Base Case}, $H_1$ is either a single vertex or a single vertex supporting exactly one loop. Therefore by \eqref{equation: Full Corner Iso}, the corner $$p_w C^*(E) p_w = p_w I_1 p_w \cong C^*(H_1)$$ has stable rank one and by Theorem \ref{theorem: SR1 Corner Positive Cancellation}, $\Delta_E(h_k) \leq [x]$. Since $\Delta_E$ preserves suprema of increasing sequences, $$\Delta_E(g') = {\sup}_{n > 0} \Delta_E(h_n) \leq [x].$$ Therefore, $\Delta_E(g') \leq [x]$ and $[x] = \Delta_E(g')$ in general. Hence, $\Delta_E$ is surjective, which completes the proof of Theorem \ref{introtheorem: Cu E Iso}. 
\end{proof}

\section{Consequences}
\label{sec: Consequences}

The final section of the paper consists of an assortment of consequences of Theorem \ref{introtheorem: Cu E Iso}. In \cref{subsec: Covering Dimension}, we prove Corollary \ref{introcorollary: Cu E One Comparison} and Theorem \ref{introtheorem: Covering Dimension}. With Theorem \ref{introtheorem: Covering Dimension}, we show that at the level of the Cuntz semigroup, there is no obstruction to the conjecture that the nuclear dimension of any graph C*-algebra is at most one. Next, in \cref{subsec: Relatively Purely Large}, we provide a different picture of the Cuntz semigroup of a graph C*-algebra coming from a finite graph, which arises directly from the extensions used to build such a graph C*-algebra. This is given in Corollary \ref{corollary: Extension Picture}. A significant consequence of Corollary \ref{corollary: Extension Picture} is Corollary \ref{corollary: Extension Dichotomy}, which says that any extension used to build a graph C*-algebra coming from a finite graph is either a direct sum or satisfies a weaker version of pure largeness first defined in \cite[Definition 3.1]{HW26}. We end the paper with \cref{subsec: Inf}, where it is proved that the Cuntz semigroup of a graph C*-algebra associated to a finite graph is inf-semilattice ordered. Subsequently, such Cuntz semigroups can, in certain cases, be computed as a suitable pullback of Cu-semigroups; see Corollary \ref{corollary: Pullback}. 

\subsection{Almost unperforation and covering dimension}
\label{subsec: Covering Dimension}

We begin by recalling the definition of an almost unperforated Cu-semigroup. 

\begin{definition}[{\cite[Page 260]{Ror92}}]
Let $M$ be a positively ordered monoid and $x, y \in M$. Define $x <_s y$ if and only if there exists $k > 0$ such that $(k+1) x \leq k y$. $M$ is said to be \textbf{almost unperforated} if for $x, y \in M$ satisfying $x <_s y$, one has $x \leq y$. 
\end{definition}

For simple C*-algebras, almost unperforation of the Cuntz semigroup is equivalent to \emph{strict comparison} (see \cite[Proposition 3.2]{Ror04b}), a property for C*-algebras which allows one to recover Cuntz subequivalence of certain pairs of positive elements by testing them on dimension functions. Recently for simple, separable, nuclear and non-elementary C*-algebras, almost unperforation of the Cuntz semigroup was demonstrated to be equivalent to $\mathcal{Z}$-stability, finite nuclear dimension and being pure --- important regularity conditions which together with the UCT, when added to simple, separable, nuclear and non-elementary C*-algebras, permit their classification via the Elliott invariant; see \cite[Corollary E]{TWW17}, \cite{Thi26} and \cite{EHSSTW26}. \newline
\phantom{} \newline
It is known that the Cuntz semigroups of the three fundamental building blocks of graph C*-algebras --- AF-algebras, Kirchberg algebras and those stably isomorphic to $C(\mathbb{T})$ --- are all almost unperforated. Below, we prove Corollary \ref{introcorollary: Cu E One Comparison} which says that the Cuntz semigroups of all graph C*-algebras associated to finite graphs are also almost unperforated. 

\begin{proof}[Proof of Corollary \ref{introcorollary: Cu E One Comparison}]
Assume that $E$ is a finite graph satisfying the standing assumptions. Let $f, g \in \Cu_E$ be such that $f <_s g$. There exist $k > 0$ such that $(k + 1) \cdot f \leq k \cdot g$. Hence, $\supp_0(f) \subseteq \supp_0(g)$. Now let $\check{w} \in \supp_0(g)$ be terminal and satisfy $c(\check{w}) \neq \infty$. By Lemma \ref{lemma: Order Alternative}, $(k + 1) \cdot f(\check{w}) \leq k \cdot g(\check{w})$. Since $\overline{\mathbb{N}}$ and $\Lsc(\mathbb{T}, \overline{\mathbb{N}})$ are both almost unperforated, $f(\check{w}) \leq g(\check{w})$. By Lemma \ref{lemma: Order Alternative} once more, $f \leq g$ and so $\Cu_E$ is almost unperforated. The second statement follows from Proposition \ref{proposition: Collapse Graph}. 
\end{proof}

\begin{remark} \label{remark: Graph CFP 2}
In \cite[Definition 2.11]{OPR12}, a property called $\omega$-comparison was defined for positively ordered monoids which contain suprema of all increasing sequences. Notably, the definition of the $\omega$-comparison property applies to Cu-semigroups, and the $n$-comparison property for some $n \geq 0$, defined in \cite[Definition 2.8]{OPR12}, implies the $\omega$-comparison property. By definition, almost unperforation is equivalent to having the $0$-comparison property. So if $E$ is an arbitrary finite graph, then by Corollary \ref{introcorollary: Cu E One Comparison}, $\Cu(C^*(E))$ has the $\omega$-comparison property. By \cite[Proposition 2.17]{OPR12}, $\Cu(C^*(E))$ has the corona factorisation property for semigroups and consequently, $C^*(E)$ has the corona factorisation property by \cite[Theorem 5.11]{OPR12}. This is consistent with Proposition \ref{proposition: Graph CFP}. \newline
\phantom{} \newline
Furthermore, by using the characterisation of an inductive limit of Cu-semigroups from \cite[Lemma 3.8]{TV22}, it is straightforward and slightly tedious to check that almost unperforation is preserved by taking inductive limits of Cu-semigroups. By \cite[Lemma 3.3]{AMP07} and \cite[Theorem 2]{CEI08}, the Cuntz semigroup of any graph C*-algebra from a row-finite graph is an inductive limit of Cuntz semigroups of graph C*-algebras coming from finite graphs. Hence, such a graph C*-algebra must also have the corona factorisation property. This is extended to any graph C*-algebra by Drinen-Tomforde desingularisation (see \cite[Theorem 2.11]{DT05}). The fact that any graph C*-algebras has the corona factorisation property is known to experts; see the remark preceding \cite[Proposition 6.1]{ERR13}.
\end{remark}

The second property of Cu-semigroups we address in this section bears close ties to nuclear dimension. It is the notion of covering dimension for Cu-semigroups first defined in \cite[Definition 3.1]{TV22}. 

\begin{definition}[{\cite[Definition 3.1]{TV22}}] \label{definition: Cu Covering Dimension}
Let $S$ be a Cu-semigroup and $n \geq 0$. We say that $\dim S \leq n$ if for $r > 0$ and $x', x, y_1, \dots, y_r \in S$ satisfying $$x' \ll x \ll y_1 + \dots + y_r,$$ there exist $z_{j, k} \in S$ for $j \in \{1, 2, \dots, r \}$ and $k \in \{0, 1, \dots, n \}$ such that the following statements hold:
\begin{enumerate}
\item If $j \in \{1, 2, \dots, r \}$ and $k \in \{0, 1, \dots, n \}$ then $z_{j, k} \ll y_j$.
\item $x' \ll \sum_{j = 1}^r \sum_{k = 0}^n z_{j, k}$.
\item If $k \in \{0, 1, \dots, n \}$ then $\sum_{j = 1}^r z_{j, k} \ll x$. 
\end{enumerate}
If there is no such $n \geq 0$ such that $\dim S \leq n$ then we set $\dim S = \infty$. The \textbf{covering dimension} of $S$, denoted by $\dim S$, is the smallest $n \geq 0$ such that $\dim S \leq n$. 
\end{definition}

In \cite[Theorem 4.1]{TV22}, it was proved that if $A$ is a C*-algebra then $\dim \Cu(A)$ is bounded above by the nuclear dimension of $A$. Hence, the covering dimension of $\Cu(A)$ provides lower bounds for the nuclear dimension of $A$, which could potentially aid in determining the nuclear dimension of certain extensions of C*-algebras. By \cite[Proposition 2.9]{WZ10}, we have readily available upper bounds for the nuclear dimension of an extension, but lower bounds are far more scarce. \newline
\phantom{} \newline
Next, we will compute the covering dimension of $\Cu_E$ for a finite graph $E$ satisfying the standing assumptions and as a consequence, obtain Theorem \ref{introtheorem: Covering Dimension}. We expect the covering dimension to be at most one because graph C*-algebras are ``low-dimensional" C*-algebras; one piece of evidence supporting this claim is that the nuclear dimensions of the three fundamental building blocks of graph C*-algebras do not exceed one. To do this, we will use the following third characterisation of the way-below relation, which follows from a straightforward application of Lemma \ref{lemma: Way Below}. 

\begin{lemma} \label{lemma: Way Below 2}
Let $(S, c)$ be a skeleton and $f, g \in \Cu_S$. Then $f \ll g$ if and only if the following statements are satisfied:
\begin{enumerate}
\item $\supp_0(f) \subseteq \supp_0(g)$. 
\item Suppose that $w\in \supp_0(f)$ is a terminal vertex for $\supp_0(g)$ and $c(w) \neq \infty$. Then $f(w) \ll g(w)$. 
\end{enumerate}
\end{lemma}

Recall that a graph $E$ satisfies \textbf{Condition (K)} if every vertex $v \in E^0$ either does not support a cycle or supports more than one cycle. This property was first defined in \cite[Section 6]{KPRR97}.

\begin{proposition} \label{proposition: Cu E Covering Dimension}
Let $E$ be a finite graph satisfying the standing assumptions. Then
\begin{equation*}
\dim \Cu_E = 
\begin{cases}
0, \text{ if $E$ satisfies Condition (K),} \\
1, \text{ otherwise.}
\end{cases}
\end{equation*}
\end{proposition}

\begin{proof}
First suppose that $E$ satisfies Condition (K). By \cite[Theorem 4.6]{JPS01}, $C^*(E)$ has real rank zero and subsequently by \cite[Theorem 5.7]{TV22} and Theorem \ref{introtheorem: Cu E Iso}, $\dim \Cu_E = 0$. Now suppose that $E$ does not satisfy Condition (K) so that there exist vertices in $E$ which support a single loop. To see that $\dim \Cu_E \leq 1$, assume that $r > 0$ and $f', f, g_1, \dots, g_r \in \Cu_E$ satisfy $$f' \ll f \ll g_1 + \dots + g_r.$$ By the proof of Lemma \ref{lemma: Axiom O2}, there exists a rapidly increasing sequence $\{ h_n \}_{n > 0}$ in $\Cu_E$ such that for $n > 0$, $\supp_0(h_n) = \supp_0(f)$ and $\sup_{n > 0} h_n = f$. Since $f' \ll f$, there exists $h \in \Cu_E$ such that $f' \leq h \ll f$ and $\supp_0(h) = \supp_0(f)$. \newline
\phantom{} \newline
We claim that if $\check{v} \in \supp_0(f)_t$ then there exist elements $z_{j, k, v}$ for $j \in \{1, 2, \dots, r \}$ and $k \in \{0, 1 \}$ in the Cu-semigroups $\{ 0, \infty \}$, $\overline{\mathbb{N}}$ or $\Lsc(\mathbb{T}, \overline{\mathbb{N}})$ such that 
\begin{enumerate}
\item if $j \in \{1, 2, \dots, r \}$ and $k \in \{0, 1 \}$ then $z_{j, k, v} \ll g_j(\check{v})$, 
\item $h(\check{v}) \ll \sum_{j = 1}^r \sum_{k = 0}^1 z_{j, k, v}$, 
\item if $k \in \{0, 1 \}$ then $\sum_{j = 1}^r z_{j, k, v} \ll f(\check{v})$. 
\end{enumerate}
To this end, first assume that $\check{v} \in \supp_0(f)_t$ is such that $c(\check{v}) \neq \infty$. If $\check{v}$ is also a terminal vertex for $\bigcup_{i = 1}^r \supp_0(g_i)$ then by Lemma \ref{lemma: Way Below}, we have in either $\overline{\mathbb{N}}$ or $\Lsc(\mathbb{T}, \overline{\mathbb{N}})$
\begin{equation*}
h(\check{v}) \ll f(\check{v}) \ll g_1(\check{v}) + \dots + g_r(\check{v}). 
\end{equation*}
Now recall that $\overline{\mathbb{N}}$ and $\Lsc(\mathbb{T}, \overline{\mathbb{N}})$ both have covering dimension at most one by \cite[Theorem 5.7]{TV22} and \cite[Corollary 4.8]{TV22} respectively. By Definition \ref{definition: Cu Covering Dimension}, there exist $z_{j, k, v}$ for $j \in \{1, 2, \dots, r \}$ and $k \in \{0, 1 \}$ which collectively satisfy the three conditions outlined above. \newline
\phantom{} \newline
Next assume that $\check{v}$ is not a terminal vertex for $\bigcup_{i = 1}^r \supp_0(g_i)$. Then $h(\check{v}) \ll f(\check{v})$ and $g_1(\check{v}) + \dots + g_r(\check{v}) = \infty$. So let $\ell \in \{1, 2, \dots, r \}$ be such that $g_\ell(\check{v}) = \infty$. There exists $z_{\ell, 0, v}$ such that $h(\check{v}) \ll z_{\ell, 0, v} \ll f(\check{v})$ in either $\overline{\mathbb{N}}$ or $\Lsc(\mathbb{T}, \overline{\mathbb{N}})$. Now if $j \in \{1, 2, \dots, r \}$, $k \in \{0, 1 \}$ and $(j, k) \neq (\ell, 0)$ then define $z_{j, k, v} = 0$. In this way, the family of elements $\{ z_{j, k, v} \}_{j, k}$ satisfies the three properties we are after. \newline
\phantom{} \newline
Finally, assume that $c(\check{v}) = \infty$. Then $f(\check{v}) = \infty$ and there exists $\ell \in \{1, 2, \dots, r \}$ such that $g_\ell(\check{v}) = \infty$. If $j \in \{1, 2, \dots, r \}$ and $k \in \{0, 1 \}$ then similarly to the previous case, define
\begin{equation*}
z_{j, k, v} = 
\begin{cases}
\infty, \text{ if $(j, k) = (\ell, 0)$,} \\
0, \text{ otherwise.}
\end{cases}
\end{equation*}
Using the fact that $\infty \ll \infty$ in the Cu-semigroup $\{0, \infty \}$, the family $\{ z_{j, k, v} \}_{j, k}$ once again satisfies the desired three properties. This proves the claim. \newline
\phantom{} \newline
Now if $j \in \{1, 2, \dots, r \}$ and $k \in \{0, 1 \}$ then define $S_j = \supp_0(f)_t \cap \supp_0(g_j)$ and
\begin{equation*}
\zeta_{j, k} = {\sum}_{\check{v} \in S_j} \theta_{\check{v}, z_{j, k, v}}
\end{equation*}
with the convention that $\theta_{\check{v}, z_{j, k, v}} = 0$ if $z_{j, k, v} = 0$. By Lemma \ref{lemma: Hereditary Terminal}, 
\begin{IEEEeqnarray*}{rCl}
\supp_0(\zeta_{j, k}) & \subseteq & {\bigcup}_{\check{v} \in S_j} L_{S(E)}(\check{v})
\\
& \subseteq & \supp_0(g_j) \cap {\bigcup}_{\check{v} \in \supp_0(f)_t} L_{S(E)}(\check{v})
\\
& = & \supp_0(g_j) \cap \supp_0(f).
\end{IEEEeqnarray*} 
Furthermore, if $\check{v}, \check{w} \in S_j$ are distinct then $\check{w} \not \in L_{S(E)}(\check{v})$ because $\check{v}$ and $\check{w}$ are both terminal vertices for $\supp_0(f)$. Since the vertex support of $\theta_{\check{v}, z_{j, k, v}}$ is $L_{S(E)}(\check{v})$ then 
\begin{equation} \label{equation: Zeta Evaluation}
\zeta_{j, k}(\check{w}) = {\sum}_{\check{v} \in S_j} \theta_{\check{v}, z_{j, k, v}}(\check{w}) = z_{j, k, w}. 
\end{equation}
Next, observe that
\begin{equation*}
\supp_0 \Big( {\sum}_{j, k} \zeta_{j, k} \Big) = {\bigcup}_{j, k} \supp_0(\zeta_{j, k}) \subseteq  \supp_0(f). 
\end{equation*}
To see that the reverse inclusion holds, assume that $\check{v} \in \supp_0(f)$. By Lemma \ref{lemma: Hereditary Terminal}, there exists $\check{w} \in \supp_0(f)_t$ such that $\check{v} \geq \check{w}$. By construction of the elements $z_{j, k, w}$, we know that $h(\check{w}) \ll \sum_{j, k} z_{j, k, w}$ in either $\overline{\mathbb{N}}$ or $\Lsc(\mathbb{T}, \overline{\mathbb{N}})$. Now since $\supp_0(h) = \supp_0(f)$ then $h(\check{w}) \neq 0$ and so, there exist $j' \in \{1, 2, \dots, r \}$ and $k' \in \{0, 1 \}$ such that $0 \neq z_{j', k', w} = \zeta_{j', k'}(\check{w})$. Since $\check{v} \in L_S(\check{w})$ then $\zeta_{j', k'}(\check{v}) \neq 0$. This gives the reverse inclusion and so
\begin{equation} \label{equation: Zeta Support}
\supp_0 \Big( {\sum}_{j, k} \zeta_{j, k} \Big) = \supp_0(f) = \supp_0(h).
\end{equation}
We will now show that the first property in Definition \ref{definition: Cu Covering Dimension} is satisfied. Assume that $j \in \{1, 2, \dots, r \}$ and $k \in \{0, 1 \}$. We know that $\supp_0(\zeta_{j, k}) \subseteq \supp_0(g_j)$. If $\check{w} \in \supp_0(\zeta_{j, k})$ is a terminal vertex for $\supp_0(g_j)$ such that $c(\check{w}) \neq \infty$ then by definition of $\zeta_{j, k}$, there exists $\check{x} \in S_j$ such that $\check{w} \geq \check{x}$. But since $\check{x} \in \supp_0(g_j)$ then $\check{w} = \check{x} \in S_j$ and $\zeta_{j, k}(\check{w}) = z_{j, k, w} \ll g_j(\check{w})$ by \eqref{equation: Zeta Evaluation}. So by Lemma \ref{lemma: Way Below 2}, we have $\zeta_{j, k} \ll g_j$ for $j \in \{1, 2, \dots, r \}$ and $k \in \{0, 1 \}$. \newline
\phantom{} \newline
To see that the second property in Definition \ref{definition: Cu Covering Dimension} is satisfied, observe that if $\check{a} \in \supp_0(h)_t = \supp_0(f)_t$ and $c(\check{a}) \neq \infty$ then
{\small \begin{equation*}
h(\check{a}) \ll {\sum}_{j, k} z_{j, k, a} = {\sum}_{j, k} \theta_{\check{a}, z_{j, k, a}}(\check{a}) \leq {\sum}_{j, k} {\sum}_{\check{v} \in S_j} \theta_{\check{v}, z_{j, k, v}}(\check{a}) = {\sum}_{j, k} \zeta_{j, k}(\check{a}). 
\end{equation*}}
By \eqref{equation: Zeta Support} and Lemma \ref{lemma: Way Below 2}, $h \ll \sum_{j, k} \zeta_{j, k}$ in $\Cu_E$. Since $f' \leq h$ by construction then $f' \ll \sum_{j, k} \zeta_{j, k}$. \newline
\phantom{} \newline
Now we verify that the third property in Definition \ref{definition: Cu Covering Dimension} holds. First observe that if $k \in \{0, 1 \}$ then
\begin{equation*}
\supp_0 \Big( {\sum}_{j = 1}^r \zeta_{j, k} \Big) \subseteq {\bigcup}_{j = 1}^r \supp_0(\zeta_{j, k}) \subseteq \supp_0(f).
\end{equation*}
Assume that $\check{x} \in \supp_0 \Big( {\sum}_{j = 1}^r \zeta_{j, k} \Big)$ is a terminal vertex for $\supp_0(f)$ such that $c(\check{x}) \neq \infty$. If $j \in \{1, 2, \dots, r \}$ then by the inclusion $\supp_0(\zeta_{j, k}) \subseteq \supp_0(g_j)$ and \eqref{equation: Zeta Evaluation}, 
\begin{equation*}
\zeta_{j, k}(\check{x}) = 
\begin{cases}
z_{j, k, x}, \text{ if $\check{x} \in S_j$,} \\
0, \text{ otherwise.}
\end{cases}
\end{equation*}
Therefore, $${\sum}_{j = 1}^r \zeta_{j, k}(\check{x}) \leq {\sum}_{j = 1}^r z_{j, k, x} \ll f(\check{x})$$ and by Lemma \ref{lemma: Way Below 2}, $\sum_{j = 1}^r \zeta_{j, k} \ll f$ for $k \in \{0, 1 \}$. So $\dim \Cu_E \leq 1$. To see that $\dim \Cu_E \leq 1$, note that because $E$ does not satisfy Condition (K) by assumption, $C^*(E)$ has a subquotient which is stably isomorphic to $C(\mathbb{T})$. So by Theorem \ref{introtheorem: Cu E Iso} and \cite[Proposition 3.5]{TV22}, $$\dim \Cu_E = \dim \Cu(C^*(E)) \geq \dim \Lsc(\mathbb{T}, \overline{\mathbb{N}}) = 1.$$ Therefore, $\dim \Cu_E = 1$. 
\end{proof}

Using Proposition \ref{proposition: Cu E Covering Dimension}, we will now compute the covering dimension of the Cuntz semigroup of an arbitrary graph C*-algebra, proving Theorem \ref{introtheorem: Covering Dimension}. 

\begin{proof}[Proof of Theorem \ref{introtheorem: Covering Dimension}]
By \cite[Theorem 2.11]{DT05}, it suffices to prove Theorem \ref{introtheorem: Covering Dimension} in the case where $E$ is row-finite. So assume that $E$ is a row-finite graph. If $E$ satisfies Condition (K) then $C^*(E)$ has real rank zero and $\dim \Cu(C^*(E)) = 0$ by \cite[Theorem 5.7]{TV22}. \newline
\phantom{} \newline
Now assume that $E$ does not satisfy Condition (K). Then $C^*(E)$ can be written as an inductive limit of unital graph C*-algebras $\{ C^*(E_i) \}_{i > 0}$ such that if $i > 0$ then $E_i$ is a (finite) subgraph of $E$ and $E_i \subseteq E_{i+1}$; see for instance, \cite[Lemma 3.1]{AMP07} and \cite[Lemma 3.3]{AMP07}.  Since $E$ does not satisfy Condition (K) then there exists $j > 0$ such that the subgraph $E_j$ does not have Condition (K). Hence, $E_{j + n}$ does not have Condition (K) for $n \geq 0$ and by Proposition \ref{proposition: Cu E Covering Dimension} and \cite[Proposition 3.10]{TV22}, 
\begin{equation*}
\dim \Cu(C^*(E)) \leq {\liminf}_{i > 0} \dim \Cu(C^*(E_i)) = 1. 
\end{equation*}
This completes the proof. 
\end{proof}

\subsection{Relatively purely large extensions}
\label{subsec: Relatively Purely Large}

In general, it is not true that any unital graph C*-algebra is built from extensions that are either purely large or direct sums; one such example of a graph C*-algebra, which was kindly pointed out by Chris Schafhauser, comes from the graph
\begin{equation*}
\begin{tikzpicture}[baseline=(current  bounding  box.center)] 
\node (1) at (-2, 0) {$a$};
\node (2) at (0, 0) {$b$};
\node (3) at (2, 0) {$c$};
\node (4) at (4, 0) {$d$};
\draw[->] (2) to (1);
\draw[->] (2) to (3);
\draw[->] (4) to (3);
\draw[->] (1) to [out = 45, in = 135, looseness = 10] (1);
\draw[->] (1) to [out = 225, in = 315, looseness = 10] (1);
\draw[->] (3) to [out = 45, in = 135, looseness = 10] (3);
\draw[->] (3) to [out = 225, in = 315, looseness = 10] (3);
\end{tikzpicture}
\end{equation*}
Denoting the above graph by $E$, we observe that by \cite[Lemma 1.1]{HS03} and \cite[Proposition 4.11]{BG24}, the extensions
\begin{equation} \label{equation: RPL Example 1}
\begin{tikzcd}
0 \arrow[r] & I_{\{ b, c, d \}} \arrow[r] & C^*(E) \arrow[r] & \mathcal{O}_2 \arrow[r] & 0
\end{tikzcd}
\end{equation}
and
\begin{equation} \label{equation: RPL Example 2}
\begin{tikzcd}
0 \arrow[r] & I_{\{ a, b, d \}} \arrow[r] & C^*(E) \arrow[r] & \mathcal{O}_2 \arrow[r] & 0
\end{tikzcd}
\end{equation}
are both not essential and hence, cannot be purely large. Focusing on the first extension in \eqref{equation: RPL Example 1}, let $x \in \Cu(C^*(E)) \backslash \Cu(I_{\{ b, c, d \}})$. Although $x$ does not automatically absorb any element of $\Cu(I_{\{b, c, d \}})$, if $x$ happens to lie in the ideal $\Cu(I_{\{a, b \}})$, then it absorbs any element of the smaller ideal $\Cu(I_{\{ b \}})$ by Theorem \ref{introtheorem: General Cuntz Krieger}. A similar phenomenon is witnessed for the extension in \eqref{equation: RPL Example 2}. These examples suggest that although unital extensions involving graph C*-algebras associated to finite graphs which are not direct sums may not be purely large, they instead satisfy a weaker property called \emph{relative pure largeness}, which was first introduced in \cite[Definition 3.1]{HW26}. 

\begin{definition}[{\cite[Definition 3.1]{HW26}}] \label{definition: RPL}
Let 
\begin{equation*}
\begin{tikzcd}
0 \arrow[r] & I \arrow[r] & A \arrow[r] & B \arrow[r] & 0 
\end{tikzcd}
\end{equation*}
be an extension of C*-algebras and $J$ be an ideal of $I$. The extension is said to be \textbf{purely large relative to} $J$ if for $a \in J_+$ and $b \in A_+ \backslash I$, $a \precsim b$. 
\end{definition}

Note that Definition \ref{definition: RPL} is a slight modification of \cite[Definition 3.1]{HW26} because unlike \cite[Definition 3.1]{HW26}, we do not assume that the ideal $I$ is stable. We recover pure largeness from Definition \ref{definition: RPL} in the special case where $J = I$. Below, we prove Lemma \ref{lemma: RPL Extensions}, which is a generalisation of the observation highlighted for the extension in \eqref{equation: RPL Example 1}. To prepare, we remind the reader of the bijective correspondence of hereditary subsets in \eqref{equation: Hereditary Skeleton} and the bijection $\mathcal{I}$ between ideal lattices defined in \eqref{equation: Ideal Lattice 2}.

\begin{lemma} \label{lemma: RPL Extensions}
Let $E$ be a finite graph satisfying the standing assumptions. Let $H$ be a hereditary subset of $E^0$ such that the quotient $C^*(E \backslash H)$ is gauge-simple. Let $H'$ be the hereditary set
\begin{equation*}
H' = \{ w \in H \; \vert \; \text{There exists $x \in E^0 \backslash H$ such that $w \geq x$} \}. 
\end{equation*}
If $f \in \Cu_E \backslash \mathcal{I}(I_H)$ and $g \in \Cu_E$ is such that $\supp_0(g) \subseteq \langle H' \rangle$ then $f + g = f$. 
\end{lemma}

\begin{proof}
By \eqref{equation: Ideal In Terms Of Support}, $\supp_0(f) \not \subseteq \langle H \rangle$ and $\supp_0(f)^\dagger \not \subseteq H$ as hereditary subsets of $E^0$, where $\supp_0(f)^\dagger$ is defined in \eqref{equation: Dagger}. We claim that $H' \subseteq \supp_0(f)^\dagger$. To this end, assume that $w \in H'$. Then $w \in H$ and there exists $x \in E^0 \backslash H$ such that $w \geq x$. Since $\supp_0(f)^\dagger \not \subseteq H$ then there exists $z \in \supp_0(f)^\dagger$ such that $z \in E^0 \backslash H$. By assumption, the quotient $C^*(E \backslash H)$ is gauge-simple. So by Lemma \ref{lemma: Gauge Simple} and the fact that the subgraph $E \backslash H$ satisfies the standing assumptions, $\check{x} = \check{z}$. So $w \geq x \geq z$ in $E$ and since $\supp_0(f)^\dagger$ is a hereditary subset of $E^0$, $w \in \supp_0(f)^\dagger$. So $H' \subseteq \supp_0(f)^\dagger$. \newline
\phantom{} \newline
Next, we claim that if $\check{a} \in \langle H' \rangle_t$ then $f(\check{a}) = \infty$. Since $H' \subseteq \supp_0(f)^{\dagger}$ then by \eqref{equation: Hereditary Skeleton}, $\langle H' \rangle \subseteq \supp_0(f)$ and $\check{a} \in \supp_0(f)$. To see that $\check{a}$ is not a terminal vertex for $\supp_0(f)$, observe that in $E$, $a$ lies in $H'$ and thus, admits a finite path to $z$. But since $z \not \in H$, $z$ cannot admit a finite path back to $a$. Hence, $\check{a} \geq \check{z}$ and $\check{z} \not \geq \check{a}$ in $S(E)$. So $\check{a}$ is not a terminal vertex for $\supp_0(f)$ and $f(\check{a}) = \infty$. Consequently, $f$ is infinite-valued on all of $\langle H' \rangle$ and hence, $f$ absorbs any function $g \in \Cu_E$ such that $\supp_0(g) \subseteq \langle H' \rangle$. 
\end{proof}

We now use Lemma \ref{lemma: RPL Extensions} to obtain a specific decomposition of $\Cu_E$ which uses the hereditary sets $H$ and $H'$.

\begin{corollary} \label{corollary: Extension Picture}
Let $E$ be a finite graph satisfying the standing assumptions. Let $H$ be a hereditary subset of $E^0$ such that the quotient $C^*(E \backslash H)$ is gauge-simple. Let $H'$ be the hereditary set
\begin{equation*}
H' = \{ w \in H \; \vert \; \text{There exists $x \in E^0 \backslash H$ such that $w \geq x$} \}. 
\end{equation*}
Then there is an isomorphism of Cu-semigroups
\begin{equation*}
\Cu_E \cong \Cu_H \sqcup \; \big( \Cu_{E \backslash H'} \backslash \Cu_{H \backslash H'} \big).
\end{equation*}
On the RHS, if $f \in \Cu_H$, $g \in \Cu_{E \backslash H'} \backslash \Cu_{H \backslash H'}$ and $\rho: \Cu_E \rightarrow \Cu_{E \backslash H'}$ is the surjective Cu-morphism from Lemma \ref{lemma: Cu E Minus H} then $f + g = \rho(f) + g$ and $f < g$ if and only if $\rho(f) \leq g$ in $\Cu_{E \backslash H'}$. 
\end{corollary}

\begin{proof}
Let $i_H: \Cu_H \rightarrow \Cu_E$ be the order embedding defined in \eqref{equation: Inclusion}. Recall that $i_H$ defines a Cu-isomorphism from $\Cu_H$ to the ideal $\mathcal{I}(I_H)$ of $\Cu_E$. Define a map $\phi$ by
\begin{equation*}
\phi(f) = 
\begin{cases}
i_H^{-1}(f), \text{ viewed as an element of $\Cu_H$ if $f \in \mathcal{I}(I_H)$,} \\
\rho(f), \text{ viewed as an element of $\Cu_{E \backslash H'} \backslash \Cu_{H \backslash H'}$ otherwise.}
\end{cases}
\end{equation*}
To see that $\phi$ is additive, assume that $f \in \mathcal{I}(I_H)$ and $g \in \Cu_E \backslash \mathcal{I}(I_H)$. By definition of addition on $\Cu_H \sqcup \; (\Cu_{E \backslash H'} \backslash \Cu_{H \backslash H'} )$,
\begin{equation*}
\phi(f + g) = \rho(f + g) = \rho(i_H^{-1}(f)) + \rho(g) = i_H^{-1}(f) + \rho(g) = \phi(f) + \phi(g). 
\end{equation*}
To see that $\phi$ is an order-embedding, assume that $f, g \in \Cu_E$ satisfy $\phi(f) \leq \phi(g)$. If $f, g \in \mathcal{I}(I_H)$ then the inequality $f \leq g$ follows from the fact that $i_H$ is a Cu-isomorphism. Now assume that $f \in \Cu_E$ is arbitrary and $g \in \Cu_E \backslash \mathcal{I}(I_H)$. Then $\rho(f) \leq \rho(g)$ in $\Cu_{E \backslash H'}$ and by the Cu-isomorphism $\Cu_{E \backslash H'} \cong \Cu_E / \Cu_{H'}$ (see Proposition \ref{proposition: First Iso Thm Cu E}) together with Lemma \ref{lemma: RPL Extensions}, there exists $h \in \Cu_E$ such that $\supp_0(h) \subseteq \langle H' \rangle$ and $f \leq g + h = g$ in $\Cu_E$. Hence, $\phi$ is an order-embedding. Finally, $\phi$ is surjective because $i_H^{-1}$ and $\rho$ are both surjective Cu-morphisms. Therefore, $\phi$ is an isomorphism of positively ordered monoids and we are done by Remark \ref{remark: POM Iso Between Cu Semigroups}. 
\end{proof}

\begin{remark}
Observe that in Corollary \ref{corollary: Extension Picture}, $E \backslash H' = (E \backslash H) \sqcup (H \backslash H')$. Hence, 
\begin{equation*}
\Cu_{E \backslash H'} \backslash \Cu_{H \backslash H'} \cong (\Cu_{E \backslash H} \backslash \{ 0 \}) \oplus \Cu_{H \backslash H'}
\end{equation*}
and by Corollary \ref{corollary: Extension Picture}, we have a Cu-isomorphism
\begin{equation} \label{equation: Extension Picture Block}
\Cu_E \cong \Cu_H \; \sqcup \; \big( (\Cu_{E \backslash H} \backslash \{ 0 \}) \oplus \Cu_{H \backslash H'} \big).
\end{equation}
We describe the addition and order on the RHS of \eqref{equation: Extension Picture Block}. Let $f \in \Cu_H$, $g \in \Cu_{E \backslash H} \backslash \{ 0 \}$ and $h \in \Cu_{H \backslash H'}$. Then $f + (g, h) = (g, \rho(f) + h)$ and $f < (g, h)$ if and only if $\rho(f) \leq h$ in $\Cu_{H \backslash H'}$. The addition and order behave as usual when restricted to each disjoint component. \newline
\phantom{} \newline
An advantage the characterisation of $\Cu_E$ in \eqref{equation: Extension Picture Block} has over the one in Corollary \ref{corollary: Extension Picture} is the presence of $\Cu_{E \backslash H}$, which is isomorphic to $\{0, \infty \}$, $\overline{\mathbb{N}}$ or $\Lsc(\mathbb{T}, \overline{\mathbb{N}})$ by Lemma \ref{lemma: Base Case}. These are the Cuntz semigroups of the three building blocks for graph C*-algebras coming from a finite graph. So with \eqref{equation: Extension Picture Block}, one can clearly see how they combine to form $\Cu_E \cong \Cu(C^*(E))$ when $E$ satisfies the standing assumptions. 
\end{remark}

Next, we combine Corollary \ref{corollary: Extension Picture} with Theorem \ref{introtheorem: Cu E Iso} to prove the following corollary.

\begin{corollary} \label{corollary: Extension Dichotomy}
Let $E$ be a finite graph. The graph C*-algebra $C^*(E)$ is built by recursively taking extensions which are either direct sums or relatively purely large.
\end{corollary}

To set up the proof, we first remind the reader of the construction of the saturation of a hereditary subset (see for instance, the remark after \cite[Definition 2.5]{Tom04}). 

\begin{definition} \label{definition: Saturation}
Let $E$ be a finite graph and $H$ be a hereditary subset of $E^0$. Define $H_0 = H$. Suppose that we have defined $H_n$ for some $n \in \{0, 1, \dots \}$. Define the set $H_{n+1}$ by
\begin{equation*}
H_{n+1} = H_n \cup \left\{ v \in E^0 \; \Big\vert \; 
\begin{matrix}
\text{$0 < \lvert r^{-1}(v) \rvert < \infty$ and} \\
\text{$r(e) = v$ implies $s(e) \in H_n$} 
\end{matrix}
 \right\}.  
\end{equation*}
The saturation $H_s$ is the set 
\begin{equation*}
H_s = {\bigcup}_{n = 0}^{\infty} H_n. 
\end{equation*}
\end{definition}

It is straightforward to check that the saturation of a hereditary set is hereditary and saturated.

\begin{proof}[Proof of Corollary \ref{corollary: Extension Dichotomy}]
If $C^*(E)$ is gauge-simple then there is nothing to show. If $C^*(E)$ is not gauge-simple then there exists a proper, hereditary, saturated subset $F$ of $E^0$ such that the quotient $C^*(E \backslash F)$ is gauge-simple. Let $F'$ be the saturation of the hereditary subset
\begin{equation*}
F_0 = \{ w \in F \; \vert \; \text{There exists $x \in E^0 \backslash F$ such that $w \geq x$} \}.
\end{equation*}
Assume that $F' = \emptyset$. If $v, w \in E^0$, $v \in F$ and $v \geq w$ then $w \in F$. In this case, $F$ is an isolated subgraph of $E$ and $$C^*(E) \cong C^*(F) \oplus C^*(E \backslash F).$$ Now assume that $F' \neq \emptyset$. By Proposition \ref{proposition: Collapse Graph}, there exist a finite graph $E'$ satisfying the standing assumptions and a hereditary subset $H$ of $(E')^0$ such that $C^*(E') \otimes K \cong C^*(E) \otimes K$ and $C^*(H) \otimes K \cong C^*(F) \otimes K$. Observe that if $v \in F'$ and $v$ supports at least one cycle then by Definition \ref{definition: Saturation}, $v \in F_0$. Thus, if $v \in F' \backslash F_0$ then $v$ cannot support a cycle and subsequently, the applications of moves (S) and (CO) in the proof of Proposition \ref{proposition: Collapse Graph} remove all vertices originally lying in $F' \backslash F_0$. Consequently, the moves from the proof of Proposition \ref{proposition: Collapse Graph} transform $F'$ to the hereditary set 
\begin{equation*}
H' = \{ w \in H \; \vert \; \text{There exists $x \in (E')^0 \backslash H$ such that $w \geq x$} \}. 
\end{equation*}
So $C^*(F') \otimes K \cong C^*(H') \otimes K$ and by Theorem \ref{introtheorem: Cu E Iso} and Corollary \ref{corollary: Extension Picture}, 
\begin{IEEEeqnarray*}{rCl}
\Cu(C^*(E)) & \cong & \Cu(C^*(E')) \cong \Cu_{E'}
\\
& \cong & \Cu_H \sqcup \; \big( \Cu_{E' \backslash H'} \backslash \Cu_{H \backslash H'} \big)
\\
& \cong & \Cu(C^*(H)) \sqcup \big( \Cu(C^*(E' \backslash H')) \backslash \Cu(C^*(H \backslash H')) \big)
\\
& \cong & \Cu(C^*(F)) \sqcup \big( \Cu(C^*(E \backslash F')) \backslash \Cu(C^*(F \backslash F')) \big).
\end{IEEEeqnarray*} 
Finally by definition of $\Cu(C^*(F)) \sqcup ( \Cu(C^*(E \backslash F')) \backslash \Cu(C^*(F \backslash F')) )$ in Corollary \ref{corollary: Extension Picture}, $C^*(E)$ is purely large relative to the ideal $I_{F'}$. This completes the proof. 
\end{proof}

\subsection{Inf-semilattice ordered Cu-semigroups and pullbacks}
\label{subsec: Inf}

It was proved in \cite[Theorem 3.8]{APRT22} that the Cuntz semigroup of a separable C*-algebra with stable rank one is \emph{inf-semilattice ordered}, meaning that the infimum of any two Cuntz classes exists and addition distributes over the infimum operation. 

\begin{definition} \label{definition: Inf}
An ordered monoid $S$ is said to be \textbf{inf-semilattice ordered} if for $x, y \in S$, the infimum $x \wedge y$ exists and for $x, y, z \in S$, $$(x + z) \wedge (y + z) = (x \wedge y) + z.$$
\end{definition}

It is straightforward to check that the Cuntz semigroups of the three building blocks for graph C*-algebras coming from finite graphs --- $\Lsc(\mathbb{T}, \overline{\mathbb{N}}), \overline{\mathbb{N}}$ and $\{0, \infty \}$ --- are all inf-semilattice ordered; see \cite[Theorem 3.8]{APRT22} for instance. Hence, it is natural to ask whether the Cuntz semigroup of a graph C*-algebra associated to a finite graph is inf-semilattice ordered. We prove below that this is indeed the case.

\begin{proposition} \label{proposition: Cu E Inf}
Let $E$ be a finite graph satisfying the standing assumptions. Then $\Cu_E$ is inf-semilattice ordered. Hence, if $F$ is an arbitrary finite graph then $\Cu(C^*(F))$ is inf-semilattice ordered. 
\end{proposition}

\begin{proof}
Assume that $E$ is a finite graph satisfying the standing assumptions and $f, g \in \Cu_E$. Define $f \wedge g$ to be the pointwise infimum of $f$ and $g$. By direct computation, $f \wedge g \in \Cu_E$. If $\check{v} \in S(E)^0$, then by using the fact that $\overline{\mathbb{N}}, \{ 0, \infty \}$ and $\Lsc(\mathbb{T}, \overline{\mathbb{N}})$ are all inf-semilattice ordered, $$((f + h) \wedge (g + h))(\check{v}) = (f + h)(\check{v}) \wedge (g + h)(\check{v}) = (f \wedge g)(\check{v}) + h(\check{v}).$$ Since the infimum operation on $\Lsc(\mathbb{T}, \overline{\mathbb{N}})$ is pointwise and the order on $\Cu_E$ is pointwise, $(f + h) \wedge (g + h) = (f \wedge g) + h$ and $\Cu_E$ is inf-semilattice ordered. The final statement follows from Theorem \ref{introtheorem: Cu E Iso}. 
\end{proof}

It is known that if $A$ is such that $\Cu(A)$ is inf-semilattice ordered, then $\Cu(A)$ and certain quotients of it can be computed as suitable pullbacks of Cu-semigroups. 

\begin{lemma}[see {\cite[Corollary 3.13]{APRT22}}] \label{lemma: Inf Pullback}
Let $A$ be a separable C*-algebra and suppose that $\Cu(A)$ is inf-semilattice ordered. Let $I, J \subseteq A$ be closed two-sided ideals of $A$. Then the following commutative square of $\Cu$-semigroups is a pullback square:
\begin{equation*}
\begin{tikzcd}
\Cu(A/(I \cap J)) \arrow[r] \arrow[d] & \Cu(A/I) \arrow[d] \\
\Cu(A/J) \arrow[r] & \Cu(A/(I + J)). 
\end{tikzcd}
\end{equation*} 
\end{lemma}

Lemma \ref{lemma: Inf Pullback} is proved in exactly the same way as \cite[Corollary 3.13]{APRT22}; although \cite[Corollary 3.13]{APRT22} is stated for a separable C*-algebra $A$ with stable rank one, its proof only uses the fact that $\Cu(A)$ is inf-semilattice ordered. Combining Theorem \ref{introtheorem: Cu E Iso}, Proposition \ref{proposition: Cu E Inf} and Lemma \ref{lemma: Inf Pullback}, we find that in certain cases, the Cuntz semigroup of a graph C*-algebra coming from a finite graph can be computed as a suitable pullback of two Cu-semigroups.

\begin{corollary} \label{corollary: Pullback}
Let $E$ be a finite graph and $H_1, H_2$ be hereditary, saturated subsets of $E^0$. Let $(H_1 \cup H_2)_s$ denote the saturation of the hereditary set $H_1 \cup H_2$ (see Definition \ref{definition: Saturation}). Then the following square is a pullback of Cu-semigroups:
\begin{equation*}
\begin{tikzcd}
\Cu(C^*(E \backslash (H_1 \cap H_2))) \arrow[r] \arrow[d] & \Cu(C^*(E \backslash H_1)) \arrow[d] \\
\Cu(C^*(E \backslash H_2)) \arrow[r] & \Cu(C^*(E \backslash (H_1 \cup H_2)_s)).
\end{tikzcd}
\end{equation*}
\end{corollary}

\bibliography{DPhilBibliography9}

@article{TWW17,
	author = {Tikuisis, Aaron and White, Stuart and Winter, Wilhelm},
	doi = {10.4007/annals.2017.185.1.4},
	fjournal = {Annals of Mathematics. Second Series},
	issn = {0003-486X,1939-8980},
	journal = {Ann. of Math. (2)},
	mrclass = {46L05 (47L40)},
	mrnumber = {3583354},
	mrreviewer = {Dinesh\ Jayantilal\ Karia},
	number = {1},
	pages = {229--284},
	title = {Quasidiagonality of nuclear {$C^\ast$}-algebras},
	url = {https://doi.org/10.4007/annals.2017.185.1.4},
	volume = {185},
	year = {2017}}

@misc{EHSSTW26,
	archiveprefix = {arXiv},
	author = {Samuel Evington and Shanshan Hua and Christopher Schafhauser and Apurva Seth and Aaron Tikuisis and Stuart White},
	eprint = {2610.09390},
	howpublished = {preprint (arXiv: 2610.09390 [math.OA])},
	primaryclass = {math.OA},
	title = {Regular simple nuclear {$C^\ast$}-algebras},
	url = {https://arxiv.org/abs/2610.09390},
	year = {2026}}

@misc{Thi26,
	archiveprefix = {arXiv},
	author = {Hannes Thiel},
	eprint = {2610.06761},
	howpublished = {preprint (arXiv: 2610.06761 [math.OA])},
	primaryclass = {math.OA},
	title = {Ranks of operators and pureness of {$C^\ast$}-algebras},
	url = {https://arxiv.org/abs/2610.06761},
	year = {2026}}

@article{Ror04b,
	author = {R{\o}rdam, Mikael},
	doi = {10.1142/S0129167X04002661},
	fjournal = {International Journal of Mathematics},
	issn = {0129-167X},
	journal = {Internat. J. Math.},
	mrclass = {46L35 (19K14 46L05 46L06)},
	mrnumber = {2106263},
	mrreviewer = {Nadia S. Larsen},
	number = {10},
	pages = {1065--1084},
	title = {The stable and the real rank of {$\mathcal{Z}$}-absorbing {$C^\ast$}-algebras},
	url = {http://dx.doi.org/10.1142/S0129167X04002661},
	volume = {15},
	year = {2004},
	zbl = {1077.46054}}

@book{Sak71,
	author = {Sakai, Sh\^oichir\^o},
	mrclass = {46L05 (46L10)},
	mrnumber = {442701},
	mrreviewer = {Richard\ I.\ Loebl},
	pages = {xii+253},
	publisher = {Springer-Verlag, New York-Heidelberg},
	series = {Ergebnisse der Mathematik und ihrer Grenzgebiete [Results in Mathematics and Related Areas]},
	title = {{$C^\ast$}-algebras and {$W^\ast$}-algebras},
	volume = {Band 60},
	year = {1971}}

@misc{CGSTW26,
	archiveprefix = {arXiv},
	author = {Jos{\'e} R. Carri{\'o}n and James Gabe and Christopher Schafhauser and Aaron Tikuisis and Stuart White},
	eprint = {2307.06480},
	howpublished = {preprint (arXiv: 2307.06480v4 [math.OA])},
	primaryclass = {math.OA},
	title = {Classifying *-homomorphisms {I}: {U}nital simple nuclear {$C^\ast$}-algebras},
	url = {https://arxiv.org/abs/2307.06480},
	year = {2026}}

@incollection{DF81,
	author = {Dupr\'e, Maurice J. and Fillmore, Peter A.},
	booktitle = {Topics in modern operator theory ({T}imi\c soara/{H}erculane, 1980)},
	isbn = {3-7643-1244-0},
	mrclass = {46M20 (46L99)},
	mrnumber = {672817},
	mrreviewer = {S.\ Popa},
	pages = {71--79},
	publisher = {Birkh\"auser Verlag, Basel-Boston, Mass.},
	series = {Oper. Theory Adv. Appl.},
	title = {Triviality theorems for {H}ilbert modules},
	volume = {2},
	year = {1981}}

@article{Rob13b,
	author = {Robert, Leonel},
	fjournal = {Comptes Rendus Math\'ematiques de l'Acad\'emie des Sciences. La Soci\'et\'e{} Royale du Canada. Mathematical Reports of the Academy of Science. The Royal Society of Canada},
	issn = {0706-1994,2816-5810},
	journal = {C. R. Math. Acad. Sci. Soc. R. Can.},
	mrclass = {46L05 (46L08 46L35)},
	mrnumber = {3098039},
	mrreviewer = {Joav\ Orovitz},
	number = {1},
	pages = {22--32},
	title = {The {C}untz semigroup of some spaces of dimension at most two},
	volume = {35},
	year = {2013}}

@article{Cun81,
	author = {Cuntz, Joachim},
	doi = {10.2307/1971137},
	fjournal = {Annals of Mathematics. Second Series},
	issn = {0003-486X},
	journal = {Ann. of Math. (2)},
	mrclass = {46L05 (16A54 18G99 46M20 58G12)},
	mrnumber = {604046},
	mrreviewer = {Vern\ Paulsen},
	number = {1},
	pages = {181--197},
	title = {{$K$}-theory for certain {$C^{\ast}$}-algebras},
	url = {https://doi.org/10.2307/1971137},
	volume = {113},
	year = {1981}}

@article{ERR13,
	author = {Eilers, S{\o}ren and Restorff, Gunnar and Ruiz, Efren},
	doi = {10.1016/j.jfa.2013.05.006},
	fjournal = {Journal of Functional Analysis},
	issn = {0022-1236,1096-0783},
	journal = {J. Funct. Anal.},
	mrclass = {46L35},
	mrnumber = {3056712},
	mrreviewer = {Jia-jie\ Hua},
	number = {3},
	pages = {449--468},
	title = {Classifying {$C^{\ast}$}-algebras with both finite and infinite subquotients},
	url = {https://doi.org/10.1016/j.jfa.2013.05.006},
	volume = {265},
	year = {2013}}

@article{Sor13,
	author = {S{\o}rensen, Adam P. W.},
	doi = {10.1017/S0143385712000260},
	fjournal = {Ergodic Theory and Dynamical Systems},
	issn = {0143-3857,1469-4417},
	journal = {Ergodic Theory Dynam. Systems},
	mrclass = {05C62 (37A55 46L05)},
	mrnumber = {3082546},
	number = {4},
	pages = {1199--1220},
	title = {Geometric classification of simple graph algebras},
	url = {https://doi.org/10.1017/S0143385712000260},
	volume = {33},
	year = {2013}}

@article{BP04,
	author = {Bates, Teresa and Pask, David},
	doi = {10.1017/S0143385703000348},
	fjournal = {Ergodic Theory and Dynamical Systems},
	issn = {0143-3857,1469-4417},
	journal = {Ergodic Theory Dynam. Systems},
	mrclass = {37B10 (46L05 46L55)},
	mrnumber = {2054048},
	mrreviewer = {Mark\ Tomforde},
	number = {2},
	pages = {367--382},
	title = {Flow equivalence of graph algebras},
	url = {https://doi.org/10.1017/S0143385703000348},
	volume = {24},
	year = {2004}}

@article{Wil73,
	author = {Williams, R. F.},
	doi = {10.2307/1970908},
	fjournal = {Annals of Mathematics. Second Series},
	issn = {0003-486X},
	journal = {Ann. of Math. (2)},
	mrclass = {58F20 (28A65)},
	mrnumber = {331436},
	mrreviewer = {C.\ B.\ Thomas},
	pages = {120--153; errata, ibid. (2) 99 (1974), 380--381},
	title = {Classification of subshifts of finite type},
	url = {https://doi.org/10.2307/1970908},
	volume = {98},
	year = {1973}}

@article{NR12,
	author = {Nagy, Gabriel and Reznikoff, Sarah},
	doi = {10.1112/jlms/jdr073},
	fjournal = {Journal of the London Mathematical Society. Second Series},
	issn = {0024-6107,1469-7750},
	journal = {J. Lond. Math. Soc. (2)},
	mrclass = {46L05},
	mrnumber = {2927813},
	mrreviewer = {Yangping\ Jing},
	number = {3},
	pages = {889--908},
	title = {Abelian core of graph algebras},
	url = {https://doi.org/10.1112/jlms/jdr073},
	volume = {85},
	year = {2012}}

@article{ERRS18,
	author = {Eilers, S{\o}ren and Restorff, Gunnar and Ruiz, Efren and S{\o}rensen, Adam P. W.},
	doi = {10.4153/CJM-2017-016-7},
	fjournal = {Canadian Journal of Mathematics. Journal Canadien de Math\'ematiques},
	issn = {0008-414X,1496-4279},
	journal = {Canad. J. Math.},
	mrclass = {46L35 (37B10 46L55 46L80)},
	mrnumber = {3759003},
	mrreviewer = {Mingchu\ Gao},
	number = {2},
	pages = {294--353},
	title = {Geometric classification of graph {$C^\ast$}-algebras over finite graphs},
	url = {https://doi.org/10.4153/CJM-2017-016-7},
	volume = {70},
	year = {2018}}

@article{ERRS21,
	author = {Eilers, S{\o}ren and Restorff, Gunnar and Ruiz, Efren and S{\o}rensen, Adam P. W.},
	doi = {10.1215/00127094-2021-0060},
	fjournal = {Duke Mathematical Journal},
	issn = {0012-7094,1547-7398},
	journal = {Duke Math. J.},
	mrclass = {46L35 (19K99 37A55 46L55 46L80)},
	mrnumber = {4302548},
	mrreviewer = {Bipul\ Saurabh},
	number = {11},
	pages = {2421--2517},
	title = {The complete classification of unital graph {$C^\ast$}-algebras: geometric and strong},
	url = {https://doi.org/10.1215/00127094-2021-0060},
	volume = {170},
	year = {2021}}

@article{APT20,
	author = {Antoine, Ramon and Perera, Francesc and Thiel, Hannes},
	doi = {10.1112/jlms.12343},
	fjournal = {Journal of the London Mathematical Society. Second Series},
	issn = {0024-6107,1469-7750},
	journal = {J. Lond. Math. Soc. (2)},
	mrclass = {46L05 (03C20 06B35 18Axx 19K14 46M07 46M15 46M40)},
	mrnumber = {4186119},
	number = {3},
	pages = {994--1029},
	title = {Cuntz semigroups of ultraproduct {$C^\ast$}-algebras},
	url = {https://doi.org/10.1112/jlms.12343},
	volume = {102},
	year = {2020}}

@misc{CGSTW,
	author = {Jos{\'e} R. Carri{\'o}n and James Gabe and Christopher Schafhauser and Aaron Tikuisis and Stuart White},
	howpublished = {manuscript in preparation},
	primaryclass = {math.OA},
	title = {Classifying *-homomorphisms {II}}}

@book{Bla06,
	author = {Blackadar, B.},
	doi = {10.1007/3-540-28517-2},
	isbn = {978-3-540-28486-4; 3-540-28486-9},
	mrclass = {46L05 (46L10 46L80)},
	mrnumber = {2188261},
	mrreviewer = {Paul\ Jolissaint},
	note = {Theory of $C^*$-algebras and von Neumann algebras, Operator Algebras and Non-commutative Geometry, III},
	pages = {xx+517},
	publisher = {Springer-Verlag, Berlin},
	series = {Encyclopaedia of Mathematical Sciences},
	title = {Operator algebras},
	url = {https://doi.org/10.1007/3-540-28517-2},
	volume = {122},
	year = {2006}}

@article{War80,
	author = {Warfield, Jr., R. B.},
	fjournal = {Pacific Journal of Mathematics},
	issn = {0030-8730,1945-5844},
	journal = {Pacific J. Math.},
	mrclass = {16A54 (16A65)},
	mrnumber = {615693},
	mrreviewer = {K.\ R.\ Goodearl},
	number = {2},
	pages = {457--485},
	title = {Cancellation of modules and groups and stable range of endomorphism rings},
	url = {http://projecteuclid.org/euclid.pjm/1102778738},
	volume = {91},
	year = {1980}}

@article{Rie83,
	author = {Rieffel, Marc A.},
	doi = {10.1112/plms/s3-46.2.301},
	fjournal = {Proceedings of the London Mathematical Society. Third Series},
	issn = {0024-6115,1460-244X},
	journal = {Proc. London Math. Soc. (3)},
	mrclass = {46L05 (46H20 46M20)},
	mrnumber = {693043},
	mrreviewer = {David\ Handelman},
	number = {2},
	pages = {301--333},
	title = {Dimension and stable rank in the {$K$}-theory of {$C^\ast$}-algebras},
	url = {https://doi.org/10.1112/plms/s3-46.2.301},
	volume = {46},
	year = {1983}}

@article{KR00,
	author = {Kirchberg, Eberhard and R{\o}rdam, Mikael},
	fjournal = {American Journal of Mathematics},
	issn = {0002-9327,1080-6377},
	journal = {Amer. J. Math.},
	mrclass = {46L05 (46L35)},
	mrnumber = {1759891},
	mrreviewer = {Judith\ A.\ Packer},
	number = {3},
	pages = {637--666},
	title = {Non-simple purely infinite {$C^\ast$}-algebras},
	url = {http://muse.jhu.edu/journals/american_journal_of_mathematics/v122/122.3kirchberg.pdf},
	volume = {122},
	year = {2000}}

@article{BK04,
	author = {Blanchard, Etienne and Kirchberg, Eberhard},
	doi = {10.1016/j.jfa.2003.06.008},
	fjournal = {Journal of Functional Analysis},
	issn = {0022-1236,1096-0783},
	journal = {J. Funct. Anal.},
	mrclass = {46L35 (19K56 46L80)},
	mrnumber = {2032998},
	mrreviewer = {Carla\ Farsi},
	number = {2},
	pages = {461--513},
	title = {Non-simple purely infinite {$C^\ast$}-algebras: the {H}ausdorff case},
	url = {https://doi.org/10.1016/j.jfa.2003.06.008},
	volume = {207},
	year = {2004}}

@article{KNZ17,
	author = {Kaftal, Victor and Ng, Ping Wong and Zhang, Shuang},
	doi = {10.4153/CJM-2016-015-3},
	fjournal = {Canadian Journal of Mathematics. Journal Canadien de Math\'ematiques},
	issn = {0008-414X,1496-4279},
	journal = {Canad. J. Math.},
	mrclass = {46L05 (46L35 46L45 47C15)},
	mrnumber = {3612090},
	mrreviewer = {Mohammad\ Sal\ Moslehian},
	number = {2},
	pages = {373--407},
	title = {Strict comparison of positive elements in multiplier algebras},
	url = {https://doi.org/10.4153/CJM-2016-015-3},
	volume = {69},
	year = {2017}}

@article{HS04,
	author = {Hong, Jeong Hee and Szyma{\'n}ski, Wojciech},
	doi = {10.2969/jmsj/1191418695},
	fjournal = {Journal of the Mathematical Society of Japan},
	issn = {0025-5645,1881-1167},
	journal = {J. Math. Soc. Japan},
	mrclass = {46L55 (46L05)},
	mrnumber = {2023453},
	mrreviewer = {Mark\ Tomforde},
	number = {1},
	pages = {45--64},
	title = {The primitive ideal space of the {$C^\ast$}-algebras of infinite graphs},
	url = {https://doi.org/10.2969/jmsj/1191418695},
	volume = {56},
	year = {2004}}

@article{DT05,
	author = {Drinen, D. and Tomforde, M.},
	doi = {10.1216/rmjm/1181069770},
	fjournal = {The Rocky Mountain Journal of Mathematics},
	issn = {0035-7596,1945-3795},
	journal = {Rocky Mountain J. Math.},
	mrclass = {46L05 (46L55)},
	mrnumber = {2117597},
	mrreviewer = {Jeong\ Hee\ Hong},
	number = {1},
	pages = {105--135},
	title = {The {$C^\ast$}-algebras of arbitrary graphs},
	url = {https://doi.org/10.1216/rmjm/1181069770},
	volume = {35},
	year = {2005}}

@article{DHS03,
	author = {Deicke, Klaus and Hong, Jeong Hee and Szyma{\'n}ski, Wojciech},
	doi = {10.1512/iumj.2003.52.2350},
	fjournal = {Indiana University Mathematics Journal},
	issn = {0022-2518,1943-5258},
	journal = {Indiana Univ. Math. J.},
	mrclass = {46L05 (46L55 46L80)},
	mrnumber = {2001940},
	mrreviewer = {Nadia\ S.\ Larsen},
	number = {4},
	pages = {963--979},
	title = {Stable rank of graph algebras. {T}ype {I} graph algebras and their limits},
	url = {https://doi.org/10.1512/iumj.2003.52.2350},
	volume = {52},
	year = {2003}}

@article{JPS01,
	author = {Jeong, J. A. and Park, G. H. and Shin, D. Y.},
	doi = {10.2140/pjm.2001.200.331},
	fjournal = {Pacific Journal of Mathematics},
	issn = {0030-8730,1945-5844},
	journal = {Pacific J. Math.},
	mrclass = {46L05 (46L35 46L80)},
	mrnumber = {1868695},
	mrreviewer = {Takahiro\ Sudo},
	number = {2},
	pages = {331--343},
	title = {Stable rank and real rank of graph {$C^\ast$}-algebras},
	url = {https://doi.org/10.2140/pjm.2001.200.331},
	volume = {200},
	year = {2001}}

@article{Tom04,
	author = {Tomforde, Mark},
	doi = {10.1090/S0002-9939-04-07411-8},
	fjournal = {Proceedings of the American Mathematical Society},
	issn = {0002-9939,1088-6826},
	journal = {Proc. Amer. Math. Soc.},
	mrclass = {46L05},
	mrnumber = {2051143},
	mrreviewer = {Nadia\ S.\ Larsen},
	number = {6},
	pages = {1787--1795},
	title = {Stability of {$C^\ast$}-algebras associated to graphs},
	url = {https://doi.org/10.1090/S0002-9939-04-07411-8},
	volume = {132},
	year = {2004}}

@misc{ER25,
	archiveprefix = {arXiv},
	author = {S{\o}ren Eilers and Efren Ruiz},
	eprint = {1908.03714},
	howpublished = {preprint (arXiv:1908.03714 [math.OA])},
	primaryclass = {math.OA},
	title = {Refined moves for structure-preserving isomorphism of graph {$C^\ast$}-algebras},
	url = {https://arxiv.org/abs/1908.03714},
	year = {2025}}

@article{HS02,
	author = {Hong, Jeong Hee and Szyma{\'n}ski, Wojciech},
	doi = {10.1007/s00220-002-0732-1},
	fjournal = {Communications in Mathematical Physics},
	issn = {0010-3616,1432-0916},
	journal = {Comm. Math. Phys.},
	mrclass = {46L65 (46L05 46L35 58B32)},
	mrnumber = {1942860},
	mrreviewer = {Xiao\ Chun\ Fang},
	number = {1},
	pages = {157--188},
	title = {Quantum spheres and projective spaces as graph algebras},
	url = {https://doi.org/10.1007/s00220-002-0732-1},
	volume = {232},
	year = {2002}}

@article{HS03,
	author = {Hong, Jeong Hee and Szyma{\'n}ski, Wojciech},
	doi = {10.2140/pjm.2003.211.249},
	fjournal = {Pacific Journal of Mathematics},
	issn = {0030-8730,1945-5844},
	journal = {Pacific J. Math.},
	mrclass = {46L05 (19K99 46L80 58B32)},
	mrnumber = {2015735},
	mrreviewer = {Mark\ Tomforde},
	number = {2},
	pages = {249--263},
	title = {Quantum lens spaces and graph algebras},
	url = {https://doi.org/10.2140/pjm.2003.211.249},
	volume = {211},
	year = {2003}}

@article{AMP07,
	author = {Ara, P. and Moreno, M. A. and Pardo, E.},
	doi = {10.1007/s10468-006-9044-z},
	fjournal = {Algebras and Representation Theory},
	issn = {1386-923X,1572-9079},
	journal = {Algebr. Represent. Theory},
	mrclass = {46L80 (46L05)},
	mrnumber = {2310414},
	mrreviewer = {Mark\ Tomforde},
	number = {2},
	pages = {157--178},
	title = {Nonstable {$K$}-theory for graph algebras},
	url = {https://doi.org/10.1007/s10468-006-9044-z},
	volume = {10},
	year = {2007}}

@book{Lan95,
	author = {Lance, E. C.},
	doi = {10.1017/CBO9780511526206},
	isbn = {0-521-47910-X},
	mrclass = {46L05 (46H25 46K05 46L80 46M05 47D25)},
	mrnumber = {1325694},
	mrreviewer = {Robert\ S.\ Doran},
	pages = {x+130},
	publisher = {Cambridge University Press, Cambridge},
	series = {London Mathematical Society Lecture Note Series},
	title = {Hilbert {$C^\ast$}-modules: A toolkit for operator algebraists},
	url = {https://doi.org/10.1017/CBO9780511526206},
	volume = {210},
	year = {1995}}

@article{KN06a,
	author = {Kucerovsky, D. and Ng, P. W.},
	doi = {10.7146/math.scand.a-15009},
	fjournal = {Mathematica Scandinavica},
	issn = {0025-5521,1903-1807},
	journal = {Math. Scand.},
	mrclass = {46L05},
	mrnumber = {2289022},
	number = {2},
	pages = {204--216},
	title = {{$S$}-regularity and the corona factorization property},
	url = {https://doi.org/10.7146/math.scand.a-15009},
	volume = {99},
	year = {2006}}

@article{LR95,
	author = {Lin, Hua Xin and R{\o}rdam, Mikael},
	doi = {10.1112/jlms/51.3.603},
	fjournal = {Journal of the London Mathematical Society. Second Series},
	issn = {0024-6107,1469-7750},
	journal = {J. London Math. Soc. (2)},
	mrclass = {46L05 (46L80)},
	mrnumber = {1332895},
	number = {3},
	pages = {603--613},
	title = {Extensions of inductive limits of circle algebras},
	url = {https://doi.org/10.1112/jlms/51.3.603},
	volume = {51},
	year = {1995}}

@article{EK01,
	author = {Elliott, George A. and Kucerovsky, Dan},
	doi = {10.2140/pjm.2001.198.385},
	fjournal = {Pacific Journal of Mathematics},
	issn = {0030-8730,1945-5844},
	journal = {Pacific J. Math.},
	mrclass = {46L05 (46M20)},
	mrnumber = {1835515},
	mrreviewer = {Hua\ Xin\ Lin},
	number = {2},
	pages = {385--409},
	title = {An abstract {V}oiculescu-{B}rown-{D}ouglas-{F}illmore absorption theorem},
	url = {https://doi.org/10.2140/pjm.2001.198.385},
	volume = {198},
	year = {2001}}

@article{TV23,
	author = {Thiel, Hannes and Vilalta, Eduard},
	doi = {10.1090/tran/8880},
	fjournal = {Transactions of the American Mathematical Society},
	issn = {0002-9947,1088-6850},
	journal = {Trans. Amer. Math. Soc.},
	mrclass = {46L05 (19K14 46L80 46L85)},
	mrnumber = {4608430},
	mrreviewer = {Vladimir\ Manuilov},
	number = {7},
	pages = {4713--4744},
	title = {The {G}lobal {G}limm {P}roperty},
	url = {https://doi.org/10.1090/tran/8880},
	volume = {376},
	year = {2023}}

@article{CEI08,
	author = {Coward, Kristofer T. and Elliott, George A. and Ivanescu, Cristian},
	doi = {10.1515/CRELLE.2008.075},
	fjournal = {Journal f\"ur die Reine und Angewandte Mathematik. [Crelle's Journal]},
	issn = {0075-4102,1435-5345},
	journal = {J. Reine Angew. Math.},
	mrclass = {46L35 (19K14 46L05 46L08 46L80 46M15)},
	mrnumber = {2458043},
	mrreviewer = {Francesc\ Perera},
	pages = {161--193},
	title = {The {C}untz semigroup as an invariant for {$C^\ast$}-algebras},
	url = {https://doi.org/10.1515/CRELLE.2008.075},
	volume = {623},
	year = {2008}}

@article{OPR12,
	author = {Ortega, Eduard and Perera, Francesc and R{\o}rdam, Mikael},
	doi = {10.1093/imrn/rnr013},
	fjournal = {International Mathematics Research Notices. IMRN},
	issn = {1073-7928,1687-0247},
	journal = {Int. Math. Res. Not. IMRN},
	mrclass = {46L35 (06F05 46L05 46L80)},
	mrnumber = {2874927},
	mrreviewer = {Aaron\ Tikuisis},
	number = {1},
	pages = {34--66},
	title = {The corona factorization property, stability, and the {C}untz semigroup of a {$C^\ast$}-algebra},
	url = {https://doi.org/10.1093/imrn/rnr013},
	year = {2012}}

@article{APT18,
	author = {Antoine, Ramon and Perera, Francesc and Thiel, Hannes},
	doi = {10.1090/memo/1199},
	fjournal = {Memoirs of the American Mathematical Society},
	isbn = {978-1-4704-2797-9; 978-1-4704-4282-8},
	issn = {0065-9266,1947-6221},
	journal = {Mem. Amer. Math. Soc.},
	mrclass = {06B35 (06F05 15A69 16W80 18B35 19K14 46L05 54F05)},
	mrnumber = {3756921},
	number = {1199},
	pages = {viii+191},
	title = {Tensor products and regularity properties of {C}untz semigroups},
	url = {https://doi.org/10.1090/memo/1199},
	volume = {251},
	year = {2018}}

@misc{APTV24,
	archiveprefix = {arXiv},
	author = {Ramon Antoine and Francesc Perera and Hannes Thiel and Eduard Vilalta},
	eprint = {2406.11052},
	howpublished = {preprint (arXiv: 2406.11052 [math.OA])},
	primaryclass = {math.OA},
	title = {Pure {$C^\ast$}-algebras},
	url = {https://arxiv.org/abs/2406.11052},
	year = {2024}}

@article{APS11,
	author = {Antoine, Ramon and Perera, Francesc and Santiago, Luis},
	doi = {10.1016/j.jfa.2011.02.016},
	fjournal = {Journal of Functional Analysis},
	issn = {0022-1236,1096-0783},
	journal = {J. Funct. Anal.},
	mrclass = {46L35 (46L05 46L80 46M15)},
	mrnumber = {2774057},
	mrreviewer = {Valentin\ Deaconu},
	number = {10},
	pages = {2844--2880},
	title = {Pullbacks, {$C(X)$}-algebras, and their {C}untz semigroup},
	url = {https://doi.org/10.1016/j.jfa.2011.02.016},
	volume = {260},
	year = {2011}}

@article{Ror92,
	author = {R{\o}rdam, Mikael},
	doi = {10.1016/0022-1236(92)90106-S},
	fjournal = {Journal of Functional Analysis},
	issn = {0022-1236,1096-0783},
	journal = {J. Funct. Anal.},
	mrclass = {46L05 (46L85)},
	mrnumber = {1172023},
	mrreviewer = {Mahmood\ Khoshkam},
	number = {2},
	pages = {255--269},
	title = {On the structure of simple {$C^\ast$}-algebras tensored with a {UHF}-algebra. {II}},
	url = {https://doi.org/10.1016/0022-1236(92)90106-S},
	volume = {107},
	year = {1992}}

@article{TV24,
	author = {Thiel, Hannes and Vilalta, Eduard},
	doi = {10.4171/jncg/526},
	fjournal = {Journal of Noncommutative Geometry},
	issn = {1661-6952,1661-6960},
	journal = {J. Noncommut. Geom.},
	mrclass = {46L05 (19K14 46L85)},
	mrnumber = {4705669},
	mrreviewer = {Amaury\ Freslon},
	number = {1},
	pages = {231--263},
	title = {Nowhere scattered {$C^\ast$}-algebras},
	url = {https://doi.org/10.4171/jncg/526},
	volume = {18},
	year = {2024}}

@incollection{APT11,
	author = {Ara, Pere and Perera, Francesc and Toms, Andrew S.},
	booktitle = {Aspects of operator algebras and applications},
	doi = {10.1090/conm/534/10521},
	isbn = {978-0-8218-4905-7},
	mrclass = {46L35 (19K14 46L05 46L80)},
	mrnumber = {2767222},
	mrreviewer = {Vladimir\ Manuilov},
	pages = {1--71},
	publisher = {Amer. Math. Soc., Providence, RI},
	series = {Contemp. Math.},
	title = {{$K$}-theory for operator algebras. {C}lassification of {$C^*$}-algebras},
	url = {https://doi.org/10.1090/conm/534/10521},
	volume = {534},
	year = {2011}}

@article{BC09,
	author = {Brown, Nathanial P. and Ciuperca, Alin},
	doi = {10.1016/j.jfa.2008.12.004},
	fjournal = {Journal of Functional Analysis},
	issn = {0022-1236,1096-0783},
	journal = {J. Funct. Anal.},
	mrclass = {46L08 (46L05)},
	mrnumber = {2523343},
	mrreviewer = {Vladimir\ Manuilov},
	number = {1},
	pages = {332--339},
	title = {Isomorphism of {H}ilbert modules over stably finite {$C^*$}-algebras},
	url = {https://doi.org/10.1016/j.jfa.2008.12.004},
	volume = {257},
	year = {2009}}

@article{Cun78,
	author = {Cuntz, Joachim},
	doi = {10.1007/BF01421922},
	fjournal = {Mathematische Annalen},
	issn = {0025-5831,1432-1807},
	journal = {Math. Ann.},
	mrclass = {46L05},
	mrnumber = {467332},
	mrreviewer = {J.\ W.\ Bunce},
	number = {2},
	pages = {145--153},
	title = {Dimension functions on simple {$C^\ast$}-algebras},
	url = {https://doi.org/10.1007/BF01421922},
	volume = {233},
	year = {1978}}

@article{ABP11,
	author = {Antoine, Ramon and Bosa, Joan and Perera, Francesc},
	doi = {10.1142/S0129167X11007057},
	fjournal = {International Journal of Mathematics},
	issn = {0129-167X,1793-6519},
	journal = {Internat. J. Math.},
	mrclass = {46L35 (06F05 19K14)},
	mrnumber = {2812090},
	mrreviewer = {Aaron\ Tikuisis},
	number = {6},
	pages = {837--861},
	title = {Completions of monoids with applications to the {C}untz semigroup},
	url = {https://doi.org/10.1142/S0129167X11007057},
	volume = {22},
	year = {2011}}

@article{Rob12,
	author = {Robert, Leonel},
	doi = {10.1016/j.aim.2012.07.010},
	fjournal = {Advances in Mathematics},
	issn = {0001-8708,1090-2082},
	journal = {Adv. Math.},
	mrclass = {46L35},
	mrnumber = {2970466},
	mrreviewer = {Changguo\ Wei},
	number = {5},
	pages = {2802--2836},
	title = {Classification of inductive limits of 1-dimensional {NCCW} complexes},
	url = {https://doi.org/10.1016/j.aim.2012.07.010},
	volume = {231},
	year = {2012}}

@article{APRT22,
	author = {Antoine, Ramon and Perera, Francesc and Robert, Leonel and Thiel, Hannes},
	doi = {10.1215/00127094-2021-0009},
	fjournal = {Duke Mathematical Journal},
	issn = {0012-7094,1547-7398},
	journal = {Duke Math. J.},
	mrclass = {46L05 (06F05 19K14 46L08 46L35)},
	mrnumber = {4364731},
	number = {1},
	pages = {33--99},
	title = {{$C^\ast$}-algebras of stable rank one and their {C}untz semigroups},
	url = {https://doi.org/10.1215/00127094-2021-0009},
	volume = {171},
	year = {2022}}

@article{RS21,
	author = {Robert, Leonel and Santiago, Luis},
	doi = {10.7146/math.scand.a-121016},
	fjournal = {Mathematica Scandinavica},
	issn = {0025-5521,1903-1807},
	journal = {Math. Scand.},
	mrclass = {46L35},
	mrnumber = {4299231},
	mrreviewer = {Luo\ Yi\ Shi},
	number = {1},
	pages = {131--160},
	title = {A revised augmented {C}untz semigroup},
	url = {https://doi.org/10.7146/math.scand.a-121016},
	volume = {127},
	year = {2021}}

@article{RW10,
	author = {R{\o}rdam, Mikael and Winter, Wilhelm},
	doi = {10.1515/CRELLE.2010.039},
	fjournal = {Journal f\"ur die Reine und Angewandte Mathematik. [Crelle's Journal]},
	issn = {0075-4102,1435-5345},
	journal = {J. Reine Angew. Math.},
	mrclass = {46L35 (46L05 46L85)},
	mrnumber = {2658184},
	mrreviewer = {Hiroyuki\ Osaka},
	pages = {129--155},
	title = {The {J}iang-{S}u algebra revisited},
	url = {https://doi.org/10.1515/CRELLE.2010.039},
	volume = {642},
	year = {2010}}

@misc{HW26,
	archiveprefix = {arXiv},
	author = {Shanshan Hua and Stuart White},
	eprint = {2601.08779},
	howpublished = {preprint (arXiv: 2601.08779 [math.OA])},
	primaryclass = {math.OA},
	title = {Uniqueness for embeddings of nuclear {$C^\ast$}-algebras into type {II$_{1}$}-factors},
	url = {https://arxiv.org/abs/2601.08779},
	year = {2026}}

@article{Tom08,
	author = {Toms, Andrew S.},
	doi = {10.4007/annals.2008.167.1029},
	fjournal = {Annals of Mathematics. Second Series},
	issn = {0003-486X,1939-8980},
	journal = {Ann. of Math. (2)},
	mrclass = {46L35 (19K14 46L80)},
	mrnumber = {2415391},
	mrreviewer = {Francesc\ Perera},
	number = {3},
	pages = {1029--1044},
	title = {On the classification problem for nuclear {$C^\ast$}-algebras},
	url = {https://doi.org/10.4007/annals.2008.167.1029},
	volume = {167},
	year = {2008}}

@misc{BG24,
	archiveprefix = {arXiv},
	author = {Ben Bouwen and James Gabe},
	eprint = {2412.15968},
	howpublished = {preprint (arXiv: 2412.15968 [math.OA])},
	primaryclass = {math.OA},
	title = {A unified approach for classifying simple nuclear {$C^\ast$}-algebras},
	url = {https://arxiv.org/abs/2412.15968},
	year = {2024}}

@article{ENSW25,
	author = {Evington, Samuel and Ng, Abraham C. S. and Sims, Aidan and White, Stuart},
	doi = {10.1007/s00209-025-03787-7},
	fjournal = {Mathematische Zeitschrift},
	issn = {0025-5874,1432-1823},
	journal = {Math. Z.},
	mrclass = {46L05 (46L35 46L85)},
	mrnumber = {4920021},
	mrreviewer = {Ilja\ Gogi\'c},
	number = {4},
	pages = {Paper No. 89, 28},
	title = {Nuclear dimension of extensions of commutative {$C^\ast$}-algebras by {K}irchberg algebras},
	url = {https://doi.org/10.1007/s00209-025-03787-7},
	volume = {310},
	year = {2025}}

@article{CRS10,
	author = {Ciuperca, Alin and Robert, Leonel and Santiago, Luis},
	fjournal = {Journal of Operator Theory},
	issn = {0379-4024,1841-7744},
	journal = {J. Operator Theory},
	mrclass = {46L35 (46L05)},
	mrnumber = {2669433},
	mrreviewer = {Francesc\ Perera},
	number = {1},
	pages = {155--169},
	title = {The {C}untz semigroup of ideals and quotients and a generalized {K}asparov stabilization theorem},
	volume = {64},
	year = {2010}}

@article{TV22,
	author = {Thiel, Hannes and Vilalta, Eduard},
	doi = {10.1016/j.aim.2021.108016},
	fjournal = {Advances in Mathematics},
	issn = {0001-8708,1090-2082},
	journal = {Adv. Math.},
	mrclass = {46L05 (46L85 54F45 55M10)},
	mrnumber = {4355725},
	mrreviewer = {Masaharu\ Kusuda},
	pages = {Paper No. 108016, 44},
	title = {Covering dimension of {C}untz semigroups},
	url = {https://doi.org/10.1016/j.aim.2021.108016},
	volume = {394},
	year = {2022}}

@article{GP26,
	archiveprefix = {arXiv},
	author = {Gardella, Eusebio and Perera, Francesc},
	doi = {10.4171/emss/84},
	eprint = {2212.02290},
	fjournal = {EMS Surveys in Mathematical Sciences},
	issn = {2308-2151,2308-216X},
	journal = {EMS Surv. Math. Sci.},
	mrclass = {46L05 (06F05 19K14 46L35)},
	mrnumber = {5037040},
	number = {1},
	pages = {133--214},
	primaryclass = {math.OA},
	title = {The modern theory of {C}untz semigroups of {$C^\ast$}-algebras},
	url = {https://doi.org/10.4171/emss/84},
	volume = {13},
	year = {2026}}

@article{WZ10,
	author = {Winter, Wilhelm and Zacharias, Joachim},
	doi = {10.1016/j.aim.2009.12.005},
	fjournal = {Advances in Mathematics},
	issn = {0001-8708,1090-2082},
	journal = {Adv. Math.},
	mrclass = {46L05},
	mrnumber = {2609012},
	mrreviewer = {William\ Paschke},
	number = {2},
	pages = {461--498},
	title = {The nuclear dimension of {$C^\ast$}-algebras},
	url = {https://doi.org/10.1016/j.aim.2009.12.005},
	volume = {224},
	year = {2010}}

@article{Rob11,
	author = {Robert, Leonel},
	fjournal = {M\"unster Journal of Mathematics},
	issn = {1867-5778,1867-5786},
	journal = {M\"unster J. Math.},
	mrclass = {46L05 (46L35)},
	mrnumber = {2869254},
	mrreviewer = {Xiao\ Chun\ Fang},
	pages = {65--71},
	title = {Nuclear dimension and {$n$}-comparison},
	volume = {4},
	year = {2011}}

@book{Rae05,
	author = {Raeburn, Iain},
	doi = {10.1090/cbms/103},
	isbn = {0-8218-3660-9},
	mrclass = {46L05 (22D25)},
	mrnumber = {2135030},
	mrreviewer = {Mark\ Tomforde},
	pages = {vi+113},
	publisher = {Conference Board of the Mathematical Sciences, Washington, DC; by the American Mathematical Society, Providence, RI},
	series = {CBMS Regional Conference Series in Mathematics},
	title = {Graph algebras},
	url = {https://doi.org/10.1090/cbms/103},
	volume = {103},
	year = {2005}}

@book{Mur90,
	author = {Murphy, Gerard J.},
	isbn = {0-12-511360-9},
	mrclass = {46Lxx (46-01)},
	mrnumber = {1074574},
	mrreviewer = {E.\ Gerlach},
	pages = {x+286},
	publisher = {Academic Press, Inc., Boston, MA},
	title = {{$C^\ast$}-algebras and operator theory},
	year = {1990}}

@article{KPRR97,
	author = {Kumjian, Alex and Pask, David and Raeburn, Iain and Renault, Jean},
	doi = {10.1006/jfan.1996.3001},
	fjournal = {Journal of Functional Analysis},
	issn = {0022-1236,1096-0783},
	journal = {J. Funct. Anal.},
	mrclass = {46L05 (46L55)},
	mrnumber = {1432596},
	mrreviewer = {Masatoshi\ Enomoto},
	number = {2},
	pages = {505--541},
	title = {Graphs, groupoids, and {C}untz-{K}rieger algebras},
	url = {https://doi.org/10.1006/jfan.1996.3001},
	volume = {144},
	year = {1997}}

@article{KPR98,
	author = {Kumjian, Alex and Pask, David and Raeburn, Iain},
	doi = {10.2140/pjm.1998.184.161},
	fjournal = {Pacific Journal of Mathematics},
	issn = {0030-8730,1945-5844},
	journal = {Pacific J. Math.},
	mrclass = {46L35 (05C20 22A22)},
	mrnumber = {1626528},
	mrreviewer = {Berndt\ Brenken},
	number = {1},
	pages = {161--174},
	title = {Cuntz-{K}rieger algebras of directed graphs},
	url = {https://doi.org/10.2140/pjm.1998.184.161},
	volume = {184},
	year = {1998}}

@article{KN06b,
	author = {Kucerovsky, Dan and Ng, P. W.},
	fjournal = {Houston Journal of Mathematics},
	issn = {0362-1588},
	journal = {Houston J. Math.},
	mrclass = {46L85 (46L05 46L80)},
	mrnumber = {2219330},
	mrreviewer = {Vladimir\ Manuilov},
	number = {2},
	pages = {531--550},
	title = {The corona factorization property and approximate unitary equivalence},
	volume = {32},
	year = {2006}}

@article{Gab24,
	author = {Gabe, James},
	doi = {10.1090/memo/1461},
	fjournal = {Memoirs of the American Mathematical Society},
	isbn = {978-1-4704-6793-7; 978-1-4704-7718-9},
	issn = {0065-9266,1947-6221},
	journal = {Mem. Amer. Math. Soc.},
	mrclass = {46L35 (46L80)},
	mrnumber = {4684339},
	number = {1461},
	pages = {v+115},
	title = {Classification of {$\mathcal{O}_{\infty}$}-stable {$C^\ast$}-algebras},
	url = {https://doi.org/10.1090/memo/1461},
	volume = {293},
	year = {2024}}

@article{BPRS00,
	author = {Bates, Teresa and Pask, David and Raeburn, Iain and Szyma{\'n}ski, Wojciech},
	fjournal = {New York Journal of Mathematics},
	issn = {1076-9803},
	journal = {New York J. Math.},
	mrclass = {46L05 (05C20)},
	mrnumber = {1777234},
	mrreviewer = {Xiao\ Chun\ Fang},
	pages = {307--324},
	title = {The {$C^\ast$}-algebras of row-finite graphs},
	url = {http://nyjm.albany.edu:8000/j/2000/6_307.html},
	volume = {6},
	year = {2000}}

@article{Bro77,
	author = {Brown, Lawrence G.},
	fjournal = {Pacific Journal of Mathematics},
	issn = {0030-8730,1945-5844},
	journal = {Pacific J. Math.},
	mrclass = {46L05},
	mrnumber = {454645},
	mrreviewer = {Jun\ Tomiyama},
	number = {2},
	pages = {335--348},
	title = {Stable isomorphism of hereditary subalgebras of {$C^\ast$}-algebras},
	url = {http://projecteuclid.org/euclid.pjm/1102811431},
	volume = {71},
	year = {1977}}
\bibliographystyle{amsalpha}

\end{document}